\documentclass{amsart}
\usepackage{etex}
\usepackage[utf8]{inputenc}
\usepackage{amssymb}
\usepackage{amsmath,tabu,amsthm}
\usepackage{mathtools}
\usepackage{color}
\usepackage{lscape}
\usepackage{hyperref}
\usepackage{xifthen}
\usepackage{array,booktabs}

\allowdisplaybreaks

\usepackage{tikz}
\usetikzlibrary{decorations.markings}
\usetikzlibrary{decorations.pathreplacing}

\newcommand{\hackcenter}[1]{
 \xy (0,0)*{#1}; \endxy}

\tikzset{->-/.style={decoration={
  markings,
  mark=at position #1 with {\arrow{>}}},postaction={decorate}}}

\tikzset{middlearrow/.style={
        decoration={markings,
            mark= at position 0.5 with {\arrow{#1}} ,
        },
        postaction={decorate}
    }
}

\input xy
\usepackage[all]{xy}
\xyoption{line}
\xyoption{arrow}
\xyoption{color}
\SelectTips{cm}{}

\newtheorem{theorem}{Theorem}[section]
\newtheorem{lemma}[theorem]{Lemma}
\newtheorem{proposition}[theorem]{Proposition}

\newcounter{defcounter}
\newenvironment{definition}{%
    \refstepcounter{defcounter}%
  \noindent\textbf{Definition \arabic{section}.\arabic{defcounter}}%
  \quad
}{%
}
\newcounter{examplecounter}[section]
\newenvironment{example}{%
    \refstepcounter{examplecounter}%
  \noindent\textbf{Example \arabic{section}.\arabic{examplecounter}}%
  \quad
}{%
}
\newenvironment{remark}[1][Remark]{\begin{trivlist}
\item[\hskip \labelsep {\bfseries #1}]}{\end{trivlist}}

\newcommand{\ds}{\displaystyle}
\newcommand{\scs}{\scriptstyle}

\newcommand{\h}[2]{\ifthenelse{\isempty{#2}{}}{h_{#1}}{h_{#1}^{#2}}}

\newcommand{\End}{\operatorname{End}}
\newcommand{\Res}{\operatorname{Res}}
\newcommand{\Heis}{\mathcal{H}_{tw}}
\newcommand{\Cl}{\mathcal{C}\ell_n}

\newcommand{\HC}{\mathfrak{H}^C}

\newcommand\numberthis{\addtocounter{equation}{1}\tag{\theequation}}
\newcommand{\gr}{\operatorname{gr}}

\newcommand{\grD}{\gr (\mathcal{D}^-)}
\newcommand{\Tr}{\operatorname{Tr}}
\newcommand{\TrH}{\Tr(\mathcal{H}_{tw})}
\newcommand{\TrHEv}{\TrH_{\overline{0}}}
\newcommand{\TrHOm}{\Tr^\omega(\mathcal{H}_{tw})_{\overline{0}}}
\newcommand{\TrHPl}{\Tr^>(\mathcal{H}_{tw})_{\overline{0}}}
\newcommand{\TrHMi}{\Tr^<(\mathcal{H}_{tw})_{\overline{0}}}
\newcommand{\Ind}{\operatorname{Ind}}

\newcommand{\refequal}[1]{\xy {\ar@{=}^{#1}
(-1,0)*{};(1,0)*{}};
\endxy}

\title{Trace of the Twisted Heisenberg Category}
\author{Can Ozan O\u{g}uz}
\address{Department of Mathematics\\ University of Southern California\\ Los Angeles, CA}
\email{coguz@usc.edu}

\author{Michael Reeks}
\address{Department of Mathematics\\ University of Virginia \\ Charlottesville, VA}
\email{mar3nf@virginia.edu}

\usepackage{graphicx}
\usepackage{color}

\begin{document}

\maketitle
\begin{abstract}
	We show that the trace decategorification, or zeroth Hochschild homology, of the twisted Heisenberg category defined by Cautis and Sussan is isomorphic to a quotient of $W^-$, a subalgebra of $W_{1+\infty}$ defined by Kac, Wang, and Yan. Our result is a twisted analogue of that by Cautis, Lauda, Licata, and Sussan relating $W_{1+\infty}$ and the trace decategorification of the Heisenberg category.
\end{abstract}

\setcounter{tocdepth}{2}
\tableofcontents

\section{Introduction}

Categorification is the process of enriching an algebraic object by increasing its categorical dimension by one, e.g. passing from a set to a category or from a 1-category to a 2-category. The original object can be recovered through the inverse process of decategorification. The most commonly used decategorification functor is the split Grothendieck group $K_0$, but it is natural to ask whether alternative decategorification functors may give additional insight into the categorified object. One such alternative, advocated in \cite{BGHL}, is the trace decategorification, which often encodes more information than $K_0$. 

The trace, or zeroth Hochschild homology, of a $\Bbbk$-linear additive category $\mathcal{C}$ is the $\Bbbk$-vector space given by 
$$\Tr(\mathcal{C}) := \bigg(\bigoplus_{x\in \operatorname{Ob}(\mathcal{C})} \End_{\mathcal{C}}(X)\bigg)\bigg/ \operatorname{span}_{\Bbbk}\{fg-gf\},$$
where $f$ and $g$ run through all pairs of morphisms $f:x\rightarrow y$ and $g:y\rightarrow x$ with $x,y \in \operatorname{Ob}(\mathcal{C})$. If a $\Bbbk$-linear category $\mathcal{C}$ carries a monoidal structure, then $\operatorname{span}\{fg-gf\}$ is an ideal, and $\Tr(\mathcal{C})$ becomes an algebra where multiplication in the trace is induced from tensor product of $\mathcal{C}$.  The trace has the advantage that it is, unlike $K_0$, invariant under passage to the Karoubi envelope, cf. \cite[Proposition 3.2]{BHLZ}. Since passing to the Karoubi envelope often prevents one from working with diagrams, trace seems to be a more suitable option to decategorify diagrammatic categories.  

The traces of several interesting categories have been computed. in \cite{BHLZ} and \cite{BHLW}, the trace of any categorified type ADE quantum group is shown to be isomorphic to a current algebra. In \cite{SVV}, traces of quiver Hecke algebras are studied. In \cite{EL}, the trace of the Hecke category is shown to be a semidirect product of the Weyl group and a polynomial algebra. A unifying approach to Heisenberg categorifications was given in \cite{Rosso2016} via Frobenius algebras; in \cite{Jack}, the degree zero part of the trace of these categories are computed.  

The trace $\Tr(\mathcal{C})$ is closely related to the $K_0({\mathcal{C}})$ through the Chern character map $$h_\mathcal{C}: K_0(\mathcal{C}) \rightarrow \Tr(\mathcal{C})$$
which sends the isomorphism class of an object to the class of its identity morphism in the trace.
Interestingly, the map $h_\mathcal{C}$ is usually injective, but is often not surjective. Thus, the trace often contains additional structure which has no analogue in the Grothendieck group.

One interesting example in which $h_{\mathcal{C}}$ fails to be surjective is given by the Heisenberg category $\mathcal{H}$ defined in \cite{Khovanov}. It is a $\mathbb{C}$-linear additive monoidal category. Therefore $\Tr(\mathcal{H})$ carries an algebra structure. There is an injective algebra homomorphism from the Heisenberg algebra $\mathfrak{h}$ to $K_0(\mathcal{H})$ (they are conjecturally isomorphic).  In \cite{CLLS}, $\Tr(\mathcal{H})$ is shown to be isomorphic to a quotient of $W_{1+\infty}$, a filtered algebra which is important in conformal field theory. In particular, it properly contains $\mathfrak{h}$ in filtration degree zero. Hence $\Tr(\mathcal{H})$ likely contains more information than $K_0(\mathcal{H})$. This fits into a larger framework, studied in \cite{CLLSS}, involving the elliptic Hall algebra.

 We study a twisted version of Khovanov's Heisenberg category. The twisted Heisenberg algebra $\mathfrak{h}_{tw}$ is a unital associative algebra generated by $h_{m/2}$, $m\in 2\mathbb{Z}+1$, subject to the relations $$\left[h_{\frac{n}{2}}, h_{\frac{m}{2}}\right] = \frac{n}{2} \delta_{n,-m}.$$ In \cite{CS}, a twisted version of the Heisenberg category, denoted $\mathcal{H}_{tw}$, is introduced. It is also a $\mathbb{C}$-linear additive monoidal category, with an additional $\mathbb{Z}/2\mathbb{Z}$-grading. It is proved that $K_0({\mathcal{H}_{tw}})$ contains $\mathfrak{h}_{tw}$ (again, they are conjecturally isomorphic).

The goal of this paper is to study the trace $\TrH$, and determine additional structure analogous to that in the untwisted version. We show that the even part of $\TrH$ is isomorphic as an algebra to a quotient of a subalgebra of $W_{1+\infty}$ that we will denote by $W^-$. We give explicit descriptions of $W_{1+\infty}$ and $W^-$ in Section~\ref{W-algebra}. This confirms the expectation in \cite{CLLS} that there should be a relationship between $\mathcal{H}_{tw}$ and one of two subalgebras of $W_{1+\infty}$ defined in \cite{Wang}.

\begin{theorem} \label{main theorem} There is an algebra isomorphism $$\TrH_{\overline{0}} \longrightarrow  W^-/\langle w_{0,0}, C-1\rangle.$$ \end{theorem}

Even though the isomorphism between $K_0(\mathcal{H}_{tw})$ and the twisted Heisenberg algebra $\mathfrak{h}_{tw}$ is still conjectural, we are able to completely characterize $\TrH$.

To prove Theorem \ref{main theorem}, we first compute sets of algebra generators and relations for both $W^-$ and $\TrHEv$, adapting arguments used in \cite{CLLS} to accommodate the new supercommutative elements arising from the twisting (cf. Section~\ref{trace}). We then study actions of each algebra on its canonical level one Fock space representation. These Fock space representations are isomorphic, and so induce a linear map $\Phi:\TrH\rightarrow W^-$. We prove that $\Phi$ is an algebra homomorphism by studying the actions of  both $W^-$ and $\TrH$ on their Heisenberg subalgebras. Finally, we check that the actions of the generators are identified under $\Phi$, and deduce that $\Phi$ is an algebra isomorphism.

An important tool in studying the connection between these algebras is the relationship between $\TrH$ and the degenerate affine Hecke-Clifford algebra $\HC_n$ of type $A_{n-1}$. The trace of $\HC_n$ as a vector space was computed by the second author in \cite{Mike}. The algebra $\TrH$ admits a triangular decomposition, where $\Tr(\HC_n)$ is identified with the upper (respectively lower) half. This identification simplifies some of the computations and the calculation of the graded rank of $\TrH$.

The structure of the paper is as follows. In Section 2, we describe the $W$-algebra $W^-$ of interest, describe its gradings and a set of generators, and study its Fock space representation. In Section 3, we describe trace decategorification in more detail and present the twisted Heisenberg category studied in \cite{CS}, as well as its gradings. We also identify a copy of the degenerate affine Hecke-Clifford algebra within the trace. In Section 4, we study a subalgebra of $\TrH$ consisting of circular diagrams called bubbles, and describe how they interact with other elements of the trace. Section 5 contains a number of calculations of diagrammatic relations in the trace that are useful for computing a generating set of $\TrHEv$. Finally, in Section 6, we describe a triangular decomposition of the trace, and then establish a generating set. This allows us to prove the desired isomorphism by using the action of each algebra on its Fock space.

{\bf Acknowledgements}
The authors thank Andrea Appel, Victor Kleen, Aaron Lauda,  Joshua Sussan, and Weiqiang Wang for helpful discussions and advice concerning the paper; Maxwell Siegel for his notation suggestions; and Joshua Sussan for suggesting this project.  The first author was partially supported by the NSF grant DMS-1255334. The second author was partially supported by a GRA fellowship in Weiqiang Wang's NSF grant and by a GAANN fellowship.

\section{W-algebra}
In this section, we will recall the W-algebra we are interested in, its structure as a $\mathbb{Z}$-graded and $\mathbb{N}$-filtered algebra, and one of its subalgebras -- the twisted Heisenberg algebra -- as well as their Fock space representations.

%-----------------TWISTED HEISENBERG ALGEBRA------------------%

\subsection{Twisted Heisenberg algebra $\mathfrak{h}_{tw}$}
We recall the definition of the twisted Heisenberg algebra. The twisted Heisenberg algebra $\mathfrak{h}_{tw}$ is a unital associative algebra generated by $h_{n}$ for $n\in \mathbb{Z}+\frac{1}{2}$ subject to the relation that $[h_n,h_m]=n\delta_{n,-m}$.

%\textcolor{red}{We should describe Fock space representation of $\mathfrak{h}_{tw}$ here.}

%------------------W-ALGEBRA  -----------------------------%

\subsection{W-algebra $W^-$}\label{W-algebra}

Let $\mathcal{D}$ denote the Lie algebra of differential operators on the circle. The central extension $\hat{\mathcal{D}}$  of $\mathcal{D}$ is described in \cite{Wang}.  It is generated by  by $C$ and by $w_{k,l}=t^kD^l$ for $l\in\mathbb{Z}$ and $k\in\mathbb{Z}_{\geq 0}$ where $t$ is a variable over $\mathbb{C}$, and $D=t\frac{d}{dt}$, subject to relations that $C$ and $w_{0,0}$ are central, and:
\begin{equation}\label{WForm}
[t^rf(D),t^sg(D)]=t^{r+s}(f(D+s)g(D)-f(D)g(D+r))+\psi(t^rf(D),t^sg(D))C,
\end{equation}
where
\begin{equation}
\psi(t^rf(D),t^sg(D))= \begin{cases} 
      \ds\sum_{-r\leq j\leq -1}f(j)g(j+r) & r=-s\geq 0 \\
      0 & r+s\neq 0 
   \end{cases}
\end{equation}
for $f,g$ polynomials.

The W-algebra $W_{1+\infty}$ is the universal envelopping algebra of $\hat{\mathcal{D}}$. It is shown in \cite{CLLS} that trace of Khovanov's Heisenberg category is isomorphic to a quotient of $W_{1+\infty}$. 

%$W_{1+\infty}$ is equipped with a $\mathbb{Z}$-grading called the rank grading where $\omega_{l,k}$ is in degree $l$ and $C$ is in degree 0. Moreover it has a $\mathbb{Z}_{\geq0}-$filtration where $\omega_{l,k}$ is in degree $\leq k$. We will refer to this filtration as differential filtration.  

In this paper, we are interested in the universal enveloping algebra of a central extension of a Lie subalgebra of $\mathcal{D}$ fixed by a degree preserving anti-involution. Define the map:

\begin{eqnarray*}
\sigma:&\mathcal{D}&\longrightarrow \mathcal{D} \\
&1&\mapsto \sigma(1) = -1\\
&t&\mapsto \sigma(t)=-t\\
&D&\mapsto \sigma(D)=-D.
\end{eqnarray*}

This is a degree preserving anti-involution of $\mathcal{D}$, and the Lie subalgebra fixed by $-\sigma$ is $$\mathcal{D}^-:=\{a\in\mathcal{D}|\sigma(a)=-a\}.$$ Let $\hat{\mathcal{D}}^-$ be a central extension of $\mathcal{D^-}$ where the $2$-cocycle is the restriction of the $2$-cocycle $\psi$ given above. Therefore $\hat{\mathcal{D}}^-$ is a Lie subalgebra of $\hat{\mathcal{D}}$.

  More explicitly, $\hat{\mathcal{D}}^-$ is the Lie algebra over the vector space spanned by $\{C\}\cup \{ t^{2k-1} g(D+(2k-1)/2); \ g\text{ even}\} \cup \{ t^{2k} f(D+k) ;\ f\text{ odd}\}$ where $k\in\mathbb{Z}$ and even and odd refer to even and odd polynomial functions. Its Lie bracket is given by Equation \eqref{WForm}.

Denote by $W^-$ the universal enveloping algebra of $\hat{\mathcal{D}}^-$. Our main result relates the trace of twisted Heisenberg category to a quotient of $W^-$.

$$\begin{tikzpicture}[scale=1.1]
\node[font=\Large] at (-0.5,1.5) {$\mathcal{D}$};
\draw[->] (0,1.5) to (3,1.5);
\node[font=\Large] at (3.4,1.55) {$\widehat{\mathcal{D}}$};
\node at (1.4,1.7) {central extension};
\draw[->] (3.9,1.5) to (6.7,1.5);
\node[font=\Large] at (7.5,1.5) {$W_{1+\infty}$};
\node at (5.3,1.7) {enveloping algebra};

\node[rotate=90,scale=1.5] at (-.5,.8) {$\subset$};
\node[rotate=90,scale=1.5] at (3.4,.8) {$\subset$};
\node[rotate=90,scale=1.5] at (7.4,.8) {$\subset$};
\node[text width=1cm] at (-1,.8) {fixed by -$\sigma$};

\node[font=\Large] at (-0.4,0.05) {$\mathcal{D}^-$};
\draw[->] (0,0) to (3,0);
\node[font=\Large] at (3.4,0.05) {$\widehat{\mathcal{D}^-}$};
\node at (1.4,0.2) {central extension};
\draw[->] (3.9,0) to (6.7,0);
\node[font=\Large] at (7.5,0) {$W^-$};
\node at (5.3,.2) {enveloping algebra};
\end{tikzpicture}
$$
%$W^-$ has an induced rank grading and an induced differential filtration from $\hat{\mathcal{D}}$. 

Note that not all $w_{k,\ell}$ are contained in $W^-$.

\begin{example}
When $k-\ell$ is an even integer, $w_{k,\ell}\not\in W^-$. Moreover, the difference $k-\ell$ being odd is not sufficient. For example, $t^2D=w_{2,1}\not\in W^-$ since an element starting with $t^2$ should be followed by $f(D+1)$ where $f$ is an odd polynomial function. Hence $t^2D=w_{2,1}\not\in W^-$ but $t^2(D+1)=t^2D+t^2=w_{2,1}+  w_{2,0}\in W^-$ (and, indeed, $\sigma(t^2 (D+1)) = t^2(-D-1) = -  t^2(D+1)$).
\end{example}

\subsection{Gradings on $W^-$}

There is a natural $ \mathbb{Z}^{\geq 0}$ filtration of $W^-$ called the \emph{differential filtration} with $w_{k, \ell}$ in degree $ \ell$; denote this filtration by $|\cdot|_{dot}$. It is convenient to define an additional filtration: the \emph{difference filtration}, where $w_{k, \ell}$ is in degree $\ell-k$, denoted $|\cdot|_{diff}$. That this is a filtration follows from the fact that $W^-$ also carries a filtration with $w_{k, \ell}$ in degree $k$.

These filtrations are compatible, so we have a $(\mathbb{Z}\times \mathbb{Z}^{\geq 0})$-filtration with with an element $f=t^j g(D-j/2) \in W^-$ in bidegree $\leq(|f|_{diff}, |f|_{dot}) = (\operatorname{deg}(g)-j, \operatorname{deg}(g))$, where $\operatorname{deg}(g)$ is the polynomial degree of $g(w) \in \mathbb{C}[w]$. Define the following subalgebras of $W^-$: $$W^{-,>} = \mathbb{C}\langle t^j g(D-j/2) | \deg(g)-j\geq 1\rangle;$$   $$W^{-,<} = \mathbb{C}\langle t^j g(D-j/2) | \deg(g)-j\leq  1\rangle;$$  $$W^{-,0} = \mathbb{C}\langle g(D) |\deg(g) \text{ odd}\rangle.$$
Let $W^{-,{\omega}}[\leq r, \leq k]$ denote the set of elements in difference degree $\leq r$ and differential degree $\leq k$, with $\omega \in \{>,<,0\}$.

Denote by $\gr W^-$ the associated graded object with respect to this filtration. Hence $\gr W^-$ is $(\mathbb{Z}\times \mathbb{Z}^{\geq 0})$-graded with $|w_{k,\ell}| = (\ell-k, \ell)$. For $\omega \in \{>,<,0\}$, define a generating series for the graded dimension of $\gr (W^-)^{\omega}$ by $$P_{\gr (\mathcal{W^-})^\omega}(t,q) = \sum_{r\in \mathbb{Z}} \sum_{k\in \mathbb{Z}, k\geq 0} \dim \gr (W^-)^\omega[r,k] t^r q^k.$$ 

\begin{proposition}\label{grDim Walg} The graded dimensions of $\gr (W^-)^>$ and $\gr (W^-)^<$ are given by:

$$P_{\grD^>} = \prod_{r\geq 0} \prod_{k>0} \frac{1}{1-t^{2r+1} q^k};$$
$$P_{\grD^<} = \prod_{r\leq 0} \prod_{k>0} \frac{1}{1-t^{2r-1} q^k}. $$
\end{proposition}
\begin{proof}
The algebra $W^-$ is generated by elements of the form $t^j g(D-j/2)$, where $ \deg(g)-j$ is odd. Hence $\grD^>$ is freely generated by elements $w_{k, \ell}$ with $k-\ell$ odd; such elements have bidgegree $(k-\ell, \ell)$. The proposition follows. \end{proof}

Let $W^-_{r,s}$ denote the rank $r$, differential filtration $s$ part of $W^-$. It is easy to see that the differential filtration zero part of $W^-$, namely $\ds \bigcup_{r\in\mathbb{Z}} W^-_{r,0}$, is spanned as a vector space by $\{C\}\cup\{t^{2n+1}\}_{n\in\mathbb{Z}}$. As an algebra, we have that
\begin{equation}
[t^{2n+1},t^{2m+1}]=(2n+1)\delta_{n,-m}
\end{equation}
  Hence we have an isomorphism between the differential filtration zero part of $W^-$ and the twisted Heisenberg algebra $\mathfrak{h}_{tw}$ given by:
  \begin{eqnarray*}
\phi:&\mathfrak{h}_{tw}&\longrightarrow \ds \bigcup_{r\in\mathbb{Z}} W^-_{r,0}\\
&h_{\frac{2n+1}{2}}&\mapsto \frac{1}{\sqrt{2}} t^{2n+1}
\end{eqnarray*}
where $n\in\mathbb{Z}$.

%------------------------------GENERATING SET OF W^- AS AN ALGEBRA---------------------------------%

\subsection{Generators of the algebra $W^-$}

The following lemma describes a generating set for $W^-$ as an algebra.

\begin{lemma} \label{GenSetW} The algebra $W^-/\langle w_{0,0}, C\rangle$ is generated by $w_{1,0}$, $w_{0,3}$, and $w_{\pm 2,1}  \pm w_{\pm 2,0}$. \end{lemma}
\begin{proof} Let $t^k g(D-{k/2})$ be an arbitrary element of $W^-$. Without loss of generality, we may assume $g$ is a monic monomial of the form $g(w) = w^\ell$ with $\ell-k$ odd, since lower terms in $g$ are just monomials of this form with lower degree, and thus can be generated separately. Therefore, we have \begin{equation} \label{lowerTerms} t^k g(D-k/2) = \sum_{i=0}^\ell \binom{\ell}{i}(-1)^{\ell - i} (k/2)^{\ell - i} t^k D^i. \end{equation} The leading term of this element with respect to differential degree is $t^k D^\ell.$ We will generate the leading term first, and address lower terms afterwards. There are two cases, depending on the parities of $k$ and $\ell$.

First, suppose that $k=2n$ is even and $\ell = 2m+1$ is odd (recall that $k$ and $\ell$ must have opposite parity in $W^-$). Hence, we must generate $w_{\pm 2n, 2m+1}$. The following calculations are easy, using Formula \ref{WForm}:
$$ [w_{-2,1} - w_{-2,0}, w_{1,0}] = w_{-1,0}, $$
$$ [w_{1,0}, w_{0,3}] = -3(w_{1,2}+ w_{1,1})-w_{1,0},$$
\begin{equation}\label{1,2n} [w_{1,2b},w_{0,3}] = -3w_{1, 2b+2} + O(w_{1,2b+1}),\end{equation}

where $O(\omega)$ refers to terms with lower differential degree than $\omega$. Hence, starting with $w_{1,2} - w_{1,1}$, we can use the Equation \eqref{1,2n} above to generate $w_{1,2b}$ for any $b$. Now we have:

\begin{equation}\label{2n+1,0}[w_{\pm 2a, 1}, w_{1,0}] = w_{\pm 2a+1,0},\end{equation}
\begin{equation} \label{2n,0} [w_{\pm 2a+1, 0}, w_{1,2} - w_{1,1}] = -(4a+2)w_{2a+2,1} - (2a+1)(2a+2)w_{2a+2,0}.\end{equation}

Thus, starting from $w_{2,1} + w_{2,0}$, we can generate $w_{2a,1}$ for any $a$. Finally, we have:

\begin{equation}\label{0,2n}[w_{-1,0}, w_{1,2b}] = \sum_{i=0}^{2b-1} \binom{2b}{i} (-1)^{2b-i+1} w_{0,i} = w_{0,2b-1} + O(w_{0,2b-2}),\end{equation}
\begin{align*}[w_{\pm 2a, 1}, w_{0,2b-1}] &= -\sum_{i=0}^{2b-2} \binom{2b-1}{i} (\pm 1)^{2b-i}(2)^{2b-2-i} t^{2a}D^{i+1}  \\&= w_{2a,2b-1} + O(w_{2a,2b-2}).\end{align*}

So, we can generate a polynomial with leading term $w_{\pm 2n, 2m+1}$.

Next, suppose that $k = 2n+1$ is odd and positive and $\ell = 2m$ is even. Using Formula \eqref{WForm}, we have:

\begin{align*}[w_{2a+1,0}, w_{0,2b+1}] &=  t^{2a+1} \sum_{i=0}^{2b} \binom{2b+1}{i} (2a+1)^{2b+1 -i} D^i \\&= w_{2a+1, 2b} + O(w_{2a+1, 2b-1}).\end{align*}

Now Equations \eqref{2n+1,0} and \eqref{0,2n} give that we can generate $w_{2a+1,0}$ and $w_{0,2b+1}$. Hence we can generate a polynomial with leading term $w_{2a+1,2b}.$

Finally, assume that $k=-(2n+1)$ is odd and $n=2m$ is even. Using Formula \eqref{WForm}, we have:

$$[w_{-2a,1}, w_{1,0}] = w_{1-2a,0}.$$

By Equation \eqref{2n,0}, we can therefore generate $w_{-(2a+1),0}$ for any $a$. Next, note that:

$$[w_{-1,0}, w_{1,2b}] = -\sum_{i=0}^{2b-1} \binom{2b-1}{i} (-1)^{2b-1-i} D^i = w_{0,2b-1} + O(w_{0,2b-2}).$$

By Equation \eqref{1,2n}, we can generator $w_{0, 2b+1}$ for any $b$. Finally, we have

\begin{align*} [w_{-(2a+1),0}, w_{0,2b-1}]& = t^{-(2a+1)} \sum_{i=0}^{2n-2} \binom{2n-1}{i} (-1)^{2n-i} (2a+1)^{2n-1-i} D^{i} \\&= w_{-(2a+1), 2b-2} + O(w_{-(2a+1), 2b-3}).\end{align*}

Thus, we can generate a polynomial with leading term $w_{-(2n+1), 2m}.$

It remains to adjust the lower terms of these equations so that they match those in Equation \eqref{lowerTerms}. But note that each equation used above to generate the leading term results in lower terms which lie in different filtrations of $W^-$. Therefore we can adjust the coefficients of lower terms by scaling individual equations above. Since there is no dependency between these equations, we can choose constant coefficients for the generators so that our generated polynomial has the correct lower terms.
\end{proof}

%--------------------------FOCK SPACE REPRESENTATION OF W^-  --------------------------------------------%

\subsection{Fock space representation of $W^-$}

The algebra $W^-$ inherits a Fock space representation from $W_{1+\infty}$. Let $W^{-,\geq} = W^{-,0} \oplus W^{-,>}$. For parameters $c,d \in \mathbb{C}$, let $\mathbb{C}_{c,d}$ be a one-dimensional module for $W^{-,\geq}$ on which each $w_{k,\ell}$ with $(k,\ell) \not= (0,0)$ acts as zero, $C$ acts as $c$, and $w_{0,0}$ acts as $d$. Let $\mathcal{M}_{c,d} := \Ind_{W^{-,\geq}}^{W^-} \mathbb{C}_{c,d}$. This induced module possesses the following properties:

\begin{proposition}\label{FockSpaceW} \cite{AFMO,FKRW} The $W^-$-module $\mathcal{M}_{c,d}$ has a unique irreducible quotient $\mathcal{V}_{c,d}$, which is isomorphic as a vector space to $\mathbb{C}[w_{-1,0}, w_{-2,0},w_{-3,0},\ldots ]$. \end{proposition}

\begin{proposition}\label{FockSpaceW faithful} \cite{SV} The action of $W^-/(C-1,w_{0,0})$ is faithful on $\mathcal{V}_{1,0}$. \end{proposition}
\begin{proof} This follows immmediately from the argument in \cite{SV} for $W_{1+\infty}$ because $W^-$ is a subalgebra. \end{proof}

Proposition \ref{FockSpaceW} allows us to the compute the action of the generators on $\mathcal{V}_{1,0}$, which we record for convenience below. 

\begin{proposition} \label{FockSpaceW action} Let $k$ be a positive integer. The generators of $W^-$ act on $\mathcal{V}_{1,0}$ as follows:  
\begin{align*} [w_{1,0}, w_{-k,0}] &= \delta_{1,k},\\
 [w_{-2,1}-w_{-2,0}, w_{-k,0}] &= (k+2) w_{-(k+2),0} ,\\
 [w_{2,1} + w_{2,0}, w_{-k,0}] &= -(k+2)w_{2-k,0} ,\\
 [w_{0,3}, w_{-k,0}] &= 3k w_{-k,2} - 3k^2 w_{-k,1} + k^3 w_{-k,0}.\end{align*}

 \end{proposition}

%----------TWISTED HEISENBERG CATEGORY-----------%

\section{Twisted Heisenberg category}\label{twistedHeisenberg}

We will now describe the main object of interest in the paper, the twisted Heisenberg category $\mathcal{H}_{tw}$. After defining the category, we recall the trace decategorification functor and some of its properties. We then describe some filtrations of $\TrH$, identify a copy of the degenerate affine Hecke-Clifford algebra $\HC_n$, and describe the trace of $\HC_n$.   Finally, we identify a set of distinguished elements in $\TrH$ which generate the nonzero filtration degrees of the algebra.

\subsection{Definition of $\mathcal{H}_{tw}$}
The twisted Heisenberg category $\mathcal{H}^{t}$ is defined in \cite{CS} as the Karoubi envelope of a $\mathbb{C}$-linear $\mathbb{Z}/2\mathbb{Z}$-graded additive monoidal category, whose moprhisms are described diagrammatically. There is an injective algebra homomorpshim from $\mathfrak{h}_{tw}$ to the split Grothendieck group of the twisted Heisenberg caterogy $K_0({\mathcal{H}^{t}})$. As in the untwisted case, this map is conjecturally surjective.
   
   The object of our main interest is the trace decategorification or zeroth Hochschild homology of $\mathcal{H}^{t}$. It is shown in \cite[Proposition 3.2]{BGHL} that trace of an additive category is isomorphic to the trace of its Karoubi envelope,. Therefore, we can work with the non-idempotent completed version of $\mathcal{H}^t$. We will denote it by $\mathcal{H}_{tw}$. Focusing our attention to $\mathcal{H}_{tw}$ allows us to work with the diagrammatics introduced in \cite{CS}. 

     The category $\mathcal{H}_{tw}$ is the $\mathbb{C}$-linear, $\mathbb{Z}/2\mathbb{Z}$-graded monoidal additive category whose objects are generated by $P$ and $Q$. A generic object is a sequence of $P$'s and $Q$'s. The morphisms of $\mathcal{H}_{tw}$ are generated by oriented planar diagrams up to boundary fixing isotopies, with generators 
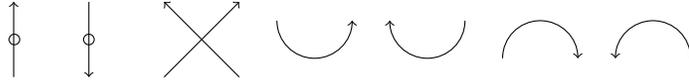
\begin{equation*}
    \hspace{0.9cm}
\begin{tikzpicture}
\draw[->](-2,-0.25) to (-2,0.75);
\draw (-1.99,0.25) circle [radius=2pt];
\draw[<-](-1,-0.25) to (-1,0.75);
\draw (-1,0.25) circle [radius=2pt];
\draw[->] (0,-0.25) to (1,0.75);
\draw[->] (1,-0.25) to (0,0.75);
\draw[->] (1.5,0.5) arc (180:360:5mm);
\draw[<-] (3,0.5) arc (180:360:5mm);
\draw[->] (4.5,0) arc (180:0:5mm);
\draw[<-] (6,0) arc (180:0:5mm);
\end{tikzpicture}
\end{equation*}
where the first diagram corresponds to a map $P\rightarrow P\{1\}$ and the second diagram corresponds to a map $Q\rightarrow Q\{1\}$, where $\{1\}$ denotes the $\mathbb{Z}/2\mathbb{Z}$-grading shift. The first two diagrams above have degree one, and the last five have degree zero. The identity morphisms of $P$ and $Q$ are indicated by an undecorated upward and downward pointing arrow, respectively. These generators satisfy the following relations:

\begin{equation}\label{R2}
\begin{tikzpicture}[baseline=(current bounding box.center),scale=0.75]
\draw [->](0,0) to [out=45,in=-45] (0,2);
\draw [->](0.5,0) to [out=135,in=-135] (0.5,2);
\end{tikzpicture}\hspace{6pt}
=
\hspace{6pt}
\begin{tikzpicture}[baseline=(current bounding box.center),scale=0.75]
\draw[->] (1.5,0) to (1.5,2);
\draw[->] (2,0) to (2,2);
\end{tikzpicture}
\hspace{1cm}
\begin{tikzpicture}[baseline=(current bounding box.center),scale=0.75]
\draw [->](0,0) to [out=45,in=-45] (0,2);
\draw [<-](0.5,0) to [out=135,in=-135] (0.5,2);
\end{tikzpicture}\hspace{6pt}
=
\hspace{6pt}
\begin{tikzpicture}[baseline=(current bounding box.center),scale=0.75]
\draw[->] (1.5,0) to (1.5,2);
\draw[<-] (2,0) to (2,2);
\end{tikzpicture}
\hspace{1.5cm}
\begin{tikzpicture}[baseline=(current bounding box.center),scale=0.75]
\draw [->](0,0) to (1.5,2);
\draw [->](1.5,0) to (0,2);
\draw[->](0.75,0) to [out=45,in=-45] (0.75,2);
\end{tikzpicture}\hspace{6pt}
=
\hspace{6pt}
\begin{tikzpicture}[baseline=(current bounding box.center),scale=0.75]
\draw [->](0,0) to (1.5,2);
\draw [->](1.5,0) to (0,2);
\draw[->](0.75,0) to [out=135,in=225] (0.75,2);
\end{tikzpicture}
\end{equation}

\begin{equation}\label{R3}
\begin{tikzpicture}[baseline=(current bounding box.center),scale=0.75]
\draw [<-](0,0) to [out=45,in=-45] (0,2);
\draw [->](0.5,0) to [out=135,in=-135] (0.5,2);
\end{tikzpicture}\hspace{6pt}
=
\hspace{6pt}
\begin{tikzpicture}[baseline=(current bounding box.center),scale=0.75]
\draw[<-] (1.5,0) to (1.5,2);
\draw[->] (2,0) to (2,2);
\end{tikzpicture}\hspace{6pt}
-
\hspace{6pt}
\begin{tikzpicture}[baseline=(current bounding box.center),scale=0.75]
\draw[->] (3,2) arc (180:360:5mm);
\draw[<-] (3,0) arc (180:0:5mm);
\end{tikzpicture}\hspace{6pt}
-
\hspace{6pt}
\begin{tikzpicture}[baseline=(current bounding box.center),scale=0.75]
\draw[->] (3,2) arc (180:360:5mm);
\draw[<-] (3,0) arc (180:0:5mm);
\draw (3.05,1.8) circle [radius=3pt];
\draw (3.95,0.2) circle [radius=3pt];
\end{tikzpicture}
\end{equation}

\begin{equation}\label{d00=1}
\begin{tikzpicture}[baseline=(current bounding box.center),scale=0.75]
\draw[->] (3,2) arc (-180:180:5mm);
\end{tikzpicture}\hspace{6pt}
=1
\hspace{1.5cm}
\begin{tikzpicture}[baseline=(current bounding box.center),scale=1.5]
\draw (1,0) to [out=90, in=-75](0.95,0.5);
\draw (0.95,0.5) arc (5:355:2mm);
\draw[->] (0.95,0.46) to [out=75, in=270] (1,1);
\end{tikzpicture}\hspace{6pt}
=0
\end{equation}

\begin{equation}
    \begin{tikzpicture}[baseline=(current bounding box.center),scale=0.75]
    \draw[->] (0,0) to (1,2);
    \draw[->] (1,0) to (0,2);
    \draw (0.25,0.5) circle [radius=3pt];
    \end{tikzpicture}\hspace{6pt}
    =
    \hspace{6pt}
    \begin{tikzpicture}[baseline=(current bounding box.center),scale=0.75]
    \draw[->] (0,0) to (1,2);
    \draw[->] (1,0) to (0,2);
    \draw (0.75,1.5) circle [radius=3pt];
    \end{tikzpicture}
    \hspace{1cm}
    \begin{tikzpicture}[baseline=(current bounding box.center),scale=0.75]
    \draw[->] (0,0) to (1,2);
    \draw[->] (1,0) to (0,2);
    \draw (0.75,0.5) circle [radius=3pt];
    \end{tikzpicture}\hspace{6pt}
    =
    \hspace{6pt}
    \begin{tikzpicture}[baseline=(current bounding box.center),scale=0.75]
    \draw[->] (0,0) to (1,2);
    \draw[->] (1,0) to (0,2);
    \draw (0.25,1.5) circle [radius=3pt];
    \end{tikzpicture}
\end{equation}

\begin{equation}\label{caps}
    \begin{tikzpicture}[baseline=(current bounding box.center)]
    \draw[<-] (3,2) arc (180:0:5mm);
    \draw (3.05,2.2) circle [radius=2.5pt];
    \end{tikzpicture}\hspace{6pt}
    =
    \hspace{6pt}-\hspace{4pt}
    \begin{tikzpicture}[baseline=(current bounding box.center)]
    \draw[<-] (3,2) arc (180:0:5mm);
    \draw (3.95,2.2) circle [radius=2.5pt];
    \end{tikzpicture}
    \hspace{1cm}
    \begin{tikzpicture}[baseline=(current bounding box.center)]
    \draw[->] (3,2) arc (180:0:5mm);
    \draw (3.05,2.2) circle [radius=2.5pt];
    \end{tikzpicture}\hspace{6pt}
    =
    \hspace{6mm}
    \begin{tikzpicture}[baseline=(current bounding box.center)]
    \draw[->] (3,2) arc (180:0:5mm);
    \draw (3.95,2.2) circle [radius=2.5pt];
    \end{tikzpicture}
\end{equation}

\begin{equation}\label{cups}
    \begin{tikzpicture}[baseline=(current bounding box.center)]
    \draw[->] (3,2) arc (180:360:5mm);
    \draw (3.05,1.8) circle [radius=2.5pt];
    \end{tikzpicture}\hspace{6pt}
    =
    \hspace{6mm}
    \begin{tikzpicture}[baseline=(current bounding box.center)]
    \draw[->] (3,2) arc (180:360:5mm);
    \draw (3.95,1.8) circle [radius=2.5pt];
    \end{tikzpicture}
    \hspace{1cm}
    \begin{tikzpicture}[baseline=(current bounding box.center)]
    \draw[<-] (3,2) arc (180:360:5mm);
    \draw (3.05,1.8) circle [radius=2.5pt];
    \end{tikzpicture}\hspace{6pt}
    =
    \hspace{6pt}-\hspace{4pt}
    \begin{tikzpicture}[baseline=(current bounding box.center)]
    \draw[<-] (3,2) arc (180:360:5mm);
    \draw (3.95,1.8) circle [radius=2.5pt];
    \end{tikzpicture}
\end{equation}

\begin{equation}\label{d01=0}
    \begin{tikzpicture}[baseline=(current bounding box.center),scale=0.75]
    \draw[->] (0,0) to (0,2);
    \end{tikzpicture}\hspace{6pt}
    =
    \hspace{6pt}
    \begin{tikzpicture}[baseline=(current bounding box.center),scale=0.75]
    \draw[->] (0,0) to (0,2);
    \draw (0,0.6) circle [radius=3pt];
    \draw (0,1.2) circle [radius=3pt];
    \end{tikzpicture}
    \hspace{1cm}
    \begin{tikzpicture}[baseline=(current bounding box.center),scale=0.75]
    \draw[<-] (0,0) to (0,2);
    \end{tikzpicture}\hspace{6pt}
    =
    \hspace{6pt}-\hspace{1mm}
    \begin{tikzpicture}[baseline=(current bounding box.center),scale=0.75]
    \draw[<-] (0,0) to (0,2);
    \draw (0,0.6) circle [radius=3pt];
    \draw (0,1.2) circle [radius=3pt];
    \end{tikzpicture}
\hspace{1cm}
    \begin{tikzpicture}[baseline=(current bounding box.center),scale=0.75]
    \draw[->] (3,2) arc (-180:180:5mm);
    \draw (3.95,2.2) circle [radius=3pt];
    \end{tikzpicture}\hspace{6pt}
    =\hspace{6pt}0
\end{equation}

\begin{equation}\label{cicj=-cjci}
    \begin{tikzpicture}[baseline=(current bounding box.center),scale=0.75]
    \draw[->] (0,0) to (0,2);
    \draw (0,1.6) circle [radius=3pt];
    \draw[fill=] (0.2,1) circle [radius=0.3pt];
    \draw[fill=] (0.4,1) circle [radius=0.3pt];
    \draw[fill=] (0.6,1) circle [radius=0.3pt];
    \draw[->] (0.8,0) to (0.8,2);
    \draw (0.8,0.3) circle [radius=3pt];
    \end{tikzpicture}\hspace{6pt}
    =\hspace{6pt}-\hspace{6pt}
     \begin{tikzpicture}[baseline=(current bounding box.center),scale=0.75]
    \draw[->] (0,0) to (0,2);
    \draw (0,0.3) circle [radius=3pt];
    \draw[fill=] (0.2,1) circle [radius=0.3pt];
    \draw[fill=] (0.4,1) circle [radius=0.3pt];
    \draw[fill=] (0.6,1) circle [radius=0.3pt];
    \draw[->] (0.8,0) to (0.8,2);
    \draw (0.8,1.6) circle [radius=3pt];
    \end{tikzpicture}. \hspace{2cm}
 \begin{tikzpicture}[baseline=(current bounding box.center),scale=0.75]
    \draw[<-] (0,0) to (0,2);
    \draw (0,1.6) circle [radius=3pt];
    \draw[fill=] (0.2,1) circle [radius=0.3pt];
    \draw[fill=] (0.4,1) circle [radius=0.3pt];
    \draw[fill=] (0.6,1) circle [radius=0.3pt];
    \draw[<-] (0.8,0) to (0.8,2);
    \draw (0.8,0.3) circle [radius=3pt];
    \end{tikzpicture}\hspace{6pt}
    =\hspace{6pt}-\hspace{6pt}
     \begin{tikzpicture}[baseline=(current bounding box.center),scale=0.75]
    \draw[<-] (0,0) to (0,2);
    \draw (0,0.3) circle [radius=3pt];
    \draw[fill=] (0.2,1) circle [radius=0.3pt];
    \draw[fill=] (0.4,1) circle [radius=0.3pt];
    \draw[fill=] (0.6,1) circle [radius=0.3pt];
    \draw[<-] (0.8,0) to (0.8,2);
    \draw (0.8,1.6) circle [radius=3pt];
    \end{tikzpicture}.
\end{equation}
Also, if we let 
\begin{equation}
    \begin{tikzpicture}[baseline=(current bounding box.center),scale=1.5]
\draw (1,0) to [out=90, in=85](1.05,0.5);
\draw (1.05,0.5) arc (-175:175:2mm);
\draw[->] (1.05,0.46) to [out=95, in=270] (1,1);
\end{tikzpicture}\hspace{6pt}
    :=
    \hspace{6pt}
    \begin{tikzpicture}[baseline=(current bounding box.center),scale=0.75]
    \draw[->](0,0) to (0,2);
    \draw[fill](0,1) circle[radius=3pt];
    \end{tikzpicture}
\end{equation}
we get the following relations:

\begin{equation}\label{anticommute}
    \begin{tikzpicture}[baseline=(current bounding box.center),scale=0.75]
    \draw[->](0,0) to (0,2);
    \draw[fill](0,0.6) circle[radius=3pt];
    \draw(0,1.2) circle[radius=3pt];
    \end{tikzpicture}\hspace{6pt}
    =
    \hspace{6pt}-\hspace{4pt}
    \begin{tikzpicture}[baseline=(current bounding box.center),scale=0.75]
    \draw[->](0,0) to (0,2);
    \draw(0,0.6) circle[radius=3pt];
    \draw[fill](0,1.2) circle[radius=3pt];
    \end{tikzpicture}
\end{equation}

\begin{equation}
    \begin{tikzpicture}[baseline=(current bounding box.center),scale=0.75]
    \draw[->] (0,0) to (0,2);
    \draw[fill] (0,1.6) circle [radius=3pt];
    \draw[fill] (0.2,1) circle [radius=0.3pt];
    \draw[fill] (0.4,1) circle [radius=0.3pt];
    \draw[fill] (0.6,1) circle [radius=0.3pt];
    \draw[->] (0.8,0) to (0.8,2);
    \draw[fill] (0.8,0.3) circle [radius=3pt];
    \end{tikzpicture}\hspace{6pt}
    =
    \hspace{6pt}
     \begin{tikzpicture}[baseline=(current bounding box.center),scale=0.75]
    \draw[->] (0,0) to (0,2);
    \draw[fill] (0,0.3) circle [radius=3pt];
    \draw[fill] (0.2,1) circle [radius=0.3pt];
    \draw[fill] (0.4,1) circle [radius=0.3pt];
    \draw[fill] (0.6,1) circle [radius=0.3pt];
    \draw[->] (0.8,0) to (0.8,2);
    \draw[fill] (0.8,1.6) circle [radius=3pt];
    \end{tikzpicture}
    \hspace{1.5cm}
    \begin{tikzpicture}[baseline=(current bounding box.center),scale=0.75]
    \draw[->] (0,0) to (0,2);
    \draw (0,1.6) circle [radius=3pt];
    \draw[fill] (0.2,1) circle [radius=0.3pt];
    \draw[fill] (0.4,1) circle [radius=0.3pt];
    \draw[fill] (0.6,1) circle [radius=0.3pt];
    \draw[->] (0.8,0) to (0.8,2);
    \draw[fill] (0.8,0.3) circle [radius=3pt];
    \end{tikzpicture}\hspace{6pt}
    =
    \hspace{6pt}
     \begin{tikzpicture}[baseline=(current bounding box.center),scale=0.75]
    \draw[->] (0,0) to (0,2);
    \draw (0,0.3) circle [radius=3pt];
    \draw[fill] (0.2,1) circle [radius=0.3pt];
    \draw[fill] (0.4,1) circle [radius=0.3pt];
    \draw[fill] (0.6,1) circle [radius=0.3pt];
    \draw[->] (0.8,0) to (0.8,2);
    \draw[fill] (0.8,1.6) circle [radius=3pt];
    \end{tikzpicture}
\end{equation}

\begin{equation}\label{dotSlide: bottomLeft}
\begin{tikzpicture}[baseline=(current bounding box.center),scale=0.75]
    \draw[->](0,0) to (1,2);
    \draw[->](1,0) to (0,2);
    \draw[fill](0.25,0.5) circle[radius=3pt];
\end{tikzpicture}\hspace{6pt}
=
\hspace{6pt}
\begin{tikzpicture}[baseline=(current bounding box.center),scale=0.75]
    \draw[->](0,0) to (1,2);
    \draw[->](1,0) to (0,2);
    \draw[fill](0.75,1.5) circle[radius=3pt];
\end{tikzpicture}\hspace{6pt}
+\hspace{2mm}
\begin{tikzpicture}[baseline=(current bounding box.center),scale=0.75]
    \draw[->](0,0) to (0,2);
    \draw[->](0.5,0) to (0.5,2);
\end{tikzpicture}
\hspace{2mm}+\hspace{2mm}
\begin{tikzpicture}[baseline=(current bounding box.center),scale=0.75]
    \draw[->](0,0) to (0,2);
    \draw[->](0.5,0) to (0.5,2);
    \draw(0,1.2) circle[radius=3pt];
    \draw(0.5,0.6) circle[radius=3pt];
\end{tikzpicture}
\end{equation}

\begin{equation}\label{dotSlide: topLeft}
\begin{tikzpicture}[baseline=(current bounding box.center),scale=0.75]
    \draw[->](0,0) to (1,2);
    \draw[->](1,0) to (0,2);
    \draw[fill](0.25,1.5) circle[radius=3pt];
\end{tikzpicture}\hspace{6pt}
=
\hspace{6pt}
\begin{tikzpicture}[baseline=(current bounding box.center),scale=0.75]
    \draw[->](0,0) to (1,2);
    \draw[->](1,0) to (0,2);
    \draw[fill](0.75,0.5) circle[radius=3pt];
\end{tikzpicture}\hspace{6pt}
+
\hspace{2mm}
\begin{tikzpicture}[baseline=(current bounding box.center),scale=0.75]
    \draw[->](0,0) to (0,2);
    \draw[->](0.5,0) to (0.5,2);
\end{tikzpicture}\hspace{2mm}
-
\hspace{2mm}
\begin{tikzpicture}[baseline=(current bounding box.center),scale=0.75]
    \draw[->](0,0) to (0,2);
    \draw[->](0.5,0) to (0.5,2);
    \draw(0,1.2) circle[radius=3pt];
    \draw(0.5,0.6) circle[radius=3pt];
\end{tikzpicture}
\end{equation}

If $x$ and $y$ are morphisms, the diagram corresponding to $x\otimes y$ is obtained by placing the diagram of $y$ to the right of the diagram of $x$. Since the relative positions of the hollow dots are important, we will work with the convention that the hollow dots in the diagram of $y$ will be placed below the height of hollow dots in the diagram of $x$.

%---- y- is obtained from placing the diagram of $y$ to the right of the diagram of $x$. However since the relative heights of the hollow dots are important, we will work with the convention that all the hollow dots of $y$ should be placed below the height of hollow dots of $x$.--------------DESCRIPTION OF TRACE DECATEGORIFICATION-----------------%
\subsection{Trace decategorification}\label{trace}
In \cite{BGHL}, the trace or zeroth Hochschild homology of a $\Bbbk$-linear additive category $\mathcal{C}$ is proposed as an alternative decategorification functor. Here we will recall its definition, and point out one subtlety occuring in our case due to the supercommutative nature of hollow dots and solid dots.

Let $\mathcal{C}$ be a $\Bbbk$-linear additive category. Then its trace decategorification, denoted $\Tr(\mathcal{C})$, is defined as follows:

\begin{equation*}
	\Tr(\mathcal{C})\simeq \bigg( \bigoplus_{x\in \mathcal{O}b(\mathcal{C})} \End(x) \bigg) \big/\mathcal{I}, 
\end{equation*}
where $\mathcal{I}$ is the ideal generated by $\operatorname{span}_{\Bbbk}\{fg-gf\}$ for all $f:x\rightarrow y$ and $g:y\rightarrow x$ , $x,y\in \mathcal{O}b(\mathcal{C})$. Note that here we quotient out by an ideal, so $\Tr(\mathcal{C})$ has an algebra structure.

Trace decategorification has a nice diagrammatic interpretation, in which we consider our string diagrams to be drawn on an annulus instead of a plane. The annulus recaptures the trace relation $fg=gf$ diagrammatically since we can slide $f$ or $g$ around the annulus to change their composition order.

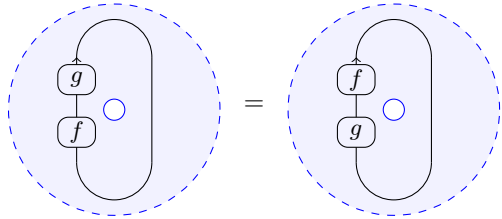
\begin{equation*}
	\begin{tikzpicture}[baseline=(current bounding box.center),rounded corners]
	\filldraw[blue,dashed,fill=blue!5!white] (0.5,1) circle [radius=40pt];
	\filldraw[fill,blue,fill=white] (0.5,1) circle [radius=4pt];
	\draw (0,0.3) to (0,0.5);
	\draw (-0.25,0.5) rectangle (0.25,0.9);
	\draw (0,0.9) to (0,1.2);
	\draw (-0.25,1.2) rectangle (0.25,1.6);
	\draw[->] (0,1.6) to (0,1.7);
	\draw (0,1.7) arc(180:0:5mm);
	\draw (1,1.7) to (1,0.3);
	\draw (1,0.3) arc(360:180:5mm);
	\node at (0,0.7){\small $f$};
	\node at (0,1.4){\small $g$};
	\end{tikzpicture}\hspace{6pt}
	=
	\hspace{6pt}
	\begin{tikzpicture}[baseline=(current bounding box.center),rounded corners]
	\filldraw[blue,dashed,fill=blue!5!white] (0.5,1) circle [radius=40pt];
	\filldraw[fill,blue,fill=white] (0.5,1) circle [radius=4pt];
	\draw (0,0.3) to (0,0.5);
	\draw (-0.25,0.5) rectangle (0.25,0.9);
	\draw (0,0.9) to (0,1.2);
	\draw (-0.25,1.2) rectangle (0.25,1.6);
	\draw[->] (0,1.6) to (0,1.7);
	\draw (0,1.7) arc(180:0:5mm);
	\draw (1,1.7) to (1,0.3);
	\draw (1,0.3) arc(360:180:5mm);
	\node at (0,0.7){\small $g$};
	\node at (0,1.4){\small $f$};
	\end{tikzpicture}
\end{equation*}

As described in Section \ref{twistedHeisenberg}, $\Heis$ has a $\mathbb{Z}/2\mathbb{Z}$-grading where \begin{tikzpicture}[baseline=(current bounding box.center),scale=0.75]
\draw[->](-2,-0.25) to (-2,0.75);
\draw (-2,0.25) circle [radius=3pt];
\end{tikzpicture} and 
\begin{tikzpicture}[baseline=(current bounding box.center),scale=0.75]
\draw[<-](-1,-0.25) to (-1,0.75);
\draw (-1,0.25) circle [radius=3pt];
\end{tikzpicture} have degree one, and other generating diagrams have degree zero. We also have supercommutativity relations (\ref{cicj=-cjci}) and (\ref{anticommute}) and supercyclicity relations (\ref{caps}) and (\ref{cups}). These relations have several interesting diagrammatic consequences.

\begin{example}
	Working with relation (\ref{anticommute}), we have the following compuation:
	\begin{equation*}
		\begin{tikzpicture}[baseline=(current bounding box.center),rounded corners]
		\filldraw[blue,dashed,fill=blue!5!white] (0.5,1) circle [radius=40pt];
		\filldraw[fill,blue,fill=white] (0.5,1) circle [radius=4pt];
		\draw[->] (0,0.3) to (0,1.7);
		\draw (0,1.7) arc(180:0:5mm);
		\draw (1,1.7) to (1,0.3);
		\draw (1,0.3) arc(360:180:5mm);
		\draw (0,0.7) circle [radius=2pt];
		\fill (0,1.4) circle [radius=2pt];
		\end{tikzpicture}\hspace{6pt}
		=
		\hspace{6pt}
		\begin{tikzpicture}[baseline=(current bounding box.center),rounded corners]
		\filldraw[blue,dashed,fill=blue!5!white] (0.5,1) circle [radius=40pt];
		\filldraw[fill,blue,fill=white] (0.5,1) circle [radius=4pt];
		\draw[->] (0,0.3) to (0,1.7);
		\draw (0,1.7) arc(180:0:5mm);
		\draw (1,1.7) to (1,0.3);
		\draw (1,0.3) arc(360:180:5mm);
		\fill (0,0.7) circle [radius=2pt];
		\draw (0,1.4) circle [radius=2pt];
		\end{tikzpicture}\hspace{6pt}
		=
		\hspace{6pt}-\hspace{4pt}
		\begin{tikzpicture}[baseline=(current bounding box.center),rounded corners]
		\filldraw[blue,dashed,fill=blue!5!white] (0.5,1) circle [radius=40pt];
		\filldraw[fill,blue,fill=white] (0.5,1) circle [radius=4pt];
		\draw[->] (0,0.3) to (0,1.7);
		\draw (0,1.7) arc(180:0:5mm);
		\draw (1,1.7) to (1,0.3);
		\draw (1,0.3) arc(360:180:5mm);
		\draw (0,0.7) circle [radius=2pt];
		\fill (0,1.4) circle [radius=2pt];
		\end{tikzpicture}
		=0.
	\end{equation*}
	Here the first equality is obtained by sending the solid dot around the annulus using trace relation, and the second equality is a consequence of relation (\ref{anticommute}). Therefore the above diagram is equal to zero in the trace.
\end{example}

\begin{example}\label{supertrace}
	To demonstrate the subtlety with supercyclicity relations (\ref{caps}) and (\ref{cups}), consider the following situation:
	\begin{equation*}
		\begin{tikzpicture}[baseline=(current bounding box.center),rounded corners]
		\filldraw[blue,dashed,fill=blue!5!white] (0.5,1) circle [radius=40pt];
		\filldraw[fill,blue,fill=white] (0.5,1) circle [radius=4pt];
		\draw[->] (0,0.3) to (0,1.7);
		\draw (0,1.7) arc(180:0:5mm);
		\draw (1,1.7) to (1,0.3);
		\draw (1,0.3) arc(360:180:5mm);
		\draw (0,1.4) circle [radius=2pt];
		\end{tikzpicture}\hspace{6pt}
		=\hspace{6pt}-\hspace{4pt}
		\begin{tikzpicture}[baseline=(current bounding box.center),rounded corners]
		\filldraw[blue,dashed,fill=blue!5!white] (0.5,1) circle [radius=40pt];
		\filldraw[fill,blue,fill=white] (0.5,1) circle [radius=4pt];
		\draw[->] (0,0.3) to (0,1.7);
		\draw (0,1.7) arc(180:0:5mm);
		\draw (1,1.7) to (1,0.3);
		\draw (1,0.3) arc(360:180:5mm);
		\draw (0,0.7) circle [radius=2pt];
		\end{tikzpicture}
	\end{equation*}
	If we denote \begin{tikzpicture}[baseline=(current bounding box.center),scale=0.75]
	\draw[->](-2,-0.25) to (-2,0.75);
	\draw (-2,0.25) circle [radius=3pt];
	\end{tikzpicture} by $f$, with the usual trace relation, we would get $f\circ id=id\circ f$. However, in this case, we gain an extra negative sign from the supercyclicity relations. So, we must replace the usual trace relation $fg=gf$ with the supertrace relation $fg=(-1)^{|f||g|}gf$ in the ideal $\mathcal{I}$, where $|f|,|g|$ are the degrees of $f$ and $g$ with respect to the $\mathbb{Z}/2\mathbb{Z}$ grading.  This example can be generalized to show that composition of an an odd morphism with a cycle of odd length is zero in the supertrace, since it will be equal to its negative when a hollow dot travels around the annulus and arrives to its original position.
\end{example}

We wish to restrict our study to the following subalgebra of the trace.

\begin{definition} The \emph{even trace} of $\Heis$ is defined by
	\begin{equation*}
		\TrHEv\simeq \bigg( \bigoplus_{x\in \mathcal{O}b(\Heis)} \End_{\overline{0}}(x) \bigg) \big/\mathcal{I}_{\overline{0}}
	\end{equation*}
	where $\End_{\overline{0}}(x)$ consists of even degree endomorphisms and $\mathcal{I}_{\overline{0}}$ is its ideal generated by $\operatorname{span}_{\mathbb{C}}\{fg-gf\}$ for all $f:x\rightarrow y$ and $g:y\rightarrow x$ , $x,y\in \mathcal{O}b(\Heis)$.
\end{definition}

This is the restriction of the trace to only the \emph{even} part (with respect to the $\mathbb{Z}/2\mathbb{Z}$ grading induced by $\deg(c_i) = 1$). The odd part of the trace is not zero (it contains, e.g., $\begin{tikzpicture}[baseline=(current bounding box.center),scale=0.75]
	\draw[->](-2,-0.25) to (-2,0.75);
	\draw (-2,0.25) circle [radius=3pt];
	\end{tikzpicture}$), but is not interesting from a representation theoretic viewpoint as explained above. The example of trace functions on the finite Hecke-Clifford algebra in \cite[Section 4.1]{WW} demonstrates the importance of the even trace.

Wan and Wang study the space of trace functions on the finite Hecke-Clifford algebra $\mathcal{H}_n$: linear functions $\phi: \mathcal{H}_n \rightarrow \mathbb{C}$ such that $\phi([h,h'])=0$ for all $h,h'\in \mathcal{H}_n$, and $\phi(h)=0$ for all $h\in (\mathcal{H}_n)_{\overline{1}}$. This latter requirement encodes the information that odd elements act with zero trace on any $\mathbb{Z}_2$-graded $\mathcal{H}_n$-module (because multiplication by an odd element results in a shift in degree). The space of such trace functions is clearly canonically isomorphic to the dual of the even cocenter, rather than of the full cocenter. The same observation holds for the trace of the affine Hecke-Clifford algebra, as studied in \cite{Mike}.

We will see in Section 4 that the structure of $\TrH$ is largely controlled by the even trace of the degenerate affine Hecke-Clifford algebra in type $A$; we therefore do not lose interesting representation-theoretic information by restricting to $\TrHEv$, and greatly simplify our calculations by doing so. For instance, the ambiguity in the supercyclicity relations identified in Example 3.\ref{supertrace} does not interfere with calculations in $\TrHEv$.

Since $\mathcal{I}_{\overline{0}}$ is an ideal of $\bigoplus_{x\in \mathcal{O}b(\Heis)}\End_{\overline{0}}(x)$, the compositions $fg$ and $gf$ must be even morphisms, even though individually $f$ and $g$ may be odd morphisms. This situation is analogous to the even cocenter of the degenerate affine Hecke-Clifford algebra studied in \cite{Mike}, where Clifford generators $c_i$ do no appear individually (as they are odd generators), but still have an impact on the cocenter via the relation $c_i^2=-1$.

Diagrammatically, the above definition means that we will have an even number of hollow dots on our diagrams. In a diagram with $2n$ hollow dots, sliding one around the annulus from top to the bottom will multiply the diagram by $(-1)^{2n-1}(-1)=1$ where $(-1)^{2n-1}$ is a result of changing relative height with the remaining $2n-1$ hollow dots using relation (\ref{cicj=-cjci}) and $(-1)$ is the result of sliding it through a clockwise cup using relation (\ref{cups}).

\begin{remark}
	For the sake of clarity, when working with diagrams in the even trace we will not draw them on an annulus, but will instead draw them inside square brackets, e.g.  
	$ \left[ \hackcenter{\begin{tikzpicture}[baseline=(current bounding box.center),scale=0.5]
		\hspace{1mm}
		\draw[->] (0,0) to (1,2);
		\draw[->] (1,0) to (0,2);
		\end{tikzpicture}}\;\: \right]$. This notation refers to the equivalence class of the diagram in $\TrHEv$.
\end{remark}

Our main theorem will relate $\TrHEv$ and $W^-$. In particular, we will establish that the correspondence in Table 1 gives an isomorphism between $\TrHEv$ and $W^-$. Recall that $w_{k,\ell} = t^{\ell} D^k \in W^-$. 
\setlength{\arrayrulewidth}{.2mm}
\setlength{\tabcolsep}{3pt}
\renewcommand{\arraystretch}{2.3}
\begin{table}[h!]
	\centering
\begin{tabular}{ |c|c|c|c|  }
	\hline
	$\TrHEv$    & bidegree ($k$-$l$,$k$) & values of ($l$,$k$) & $W^-$ \\[8pt]
	\hline
	\hspace{8pt}
	\begin{tikzpicture}[baseline=(current bounding box.center),scale=0.75]
	\draw[<-] (3,2) arc (-180:180:5mm);
	\fill (3.95,2.2) circle [radius=2pt];
	\node at (4.3,2.2) {\small $2a$};
	\end{tikzpicture} & ($2a$+$1$,$2a$+$1$) & (0,$2a$+$1$) & $-2w_{0,2a+1}$\\[8pt]
	
	\hline
	$\left[\;\; \begin{tikzpicture}[baseline=(current bounding box.center),scale=0.75]
	\draw[->] (0,0) to (0,1);
	\end{tikzpicture}\;\; \right]$ & ($1$,$0$) & ($-1$,$0$) & $\sqrt{2}w_{-1,0}$\\[8pt]
	
	\hline
	$\left[\;\; \begin{tikzpicture}[baseline=(current bounding box.center),scale=0.75]
	\draw[->] (0,0) to (0,1);
	\fill  (0,.5) circle[radius=2pt];
	\node at (0.2,0.5) {\small $2$};
	\end{tikzpicture} \right]$ & ($3$,$2$) & ($-1$,$2$) & $\sqrt{2}(w_{-1,2}-w_{-1,1})$\\[8pt]
	
	\hline
	$\left[\;\; \begin{tikzpicture}[baseline=(current bounding box.center),scale=0.75]
	\draw[<-] (0,0) to (0,1);
	\end{tikzpicture}\;\; \right]$ & ($-1$,$0$) & ($1$,$0$) & $\sqrt{2}w_{1,0}$\\[8pt]

	\hline
	\hspace{-.5cm}
	$\left[ \begin{tikzpicture}[baseline=(current bounding box.center),scale=0.5]
	%% Separate lines by 0.8
	\draw[->] (3.2,0) .. controls (3.2,1.25) and (0,.25) .. (0,2)
	node[pos=0.85, shape=coordinate](X){};
	\draw[->] (0,0) .. controls (0,1) and (.8,.8) .. (.8,2);
	\draw[->] (.8,0) .. controls (.8,1) and (1.6,.8) .. (1.6,2);
	\draw[->] (2.4,0) .. controls (2.4,1) and (3.2,.8) .. (3.2,2);
	\node at (1.6,.35) {$\dots$};
	\node at (2.4,1.65) {$\dots$};
	%\filldraw  (X) circle (2pt);
	%\node at (-0.48,1.7) {\small $2a$};
	\end{tikzpicture}\;\; \right]$ & ($n$,0) & ($-n$,0) & $\sqrt{2}w_{-n,0}$\\

		\hline
	\hspace{-.5cm}
	$\left[ \begin{tikzpicture}[baseline=(current bounding box.center),scale=0.5]
	%% Separate lines by 0.8
	\draw[->] (3.2,0) .. controls (3.2,1.25) and (0,.25) .. (0,2)
	node[pos=0.85, shape=coordinate](X){};
	\draw[->] (0,0) .. controls (0,1) and (.8,.8) .. (.8,2);
	\draw[->] (.8,0) .. controls (.8,1) and (1.6,.8) .. (1.6,2);
	\draw[->] (2.4,0) .. controls (2.4,1) and (3.2,.8) .. (3.2,2);
	\node at (1.6,.35) {$\dots$};
	\node at (2.4,1.65) {$\dots$};
	%\filldraw  (X) circle (2pt);
	\filldraw  (.05,1.6) circle (3pt);
	\node at (-0.48,1.7) {\small $2a$};
	\end{tikzpicture}\;\; \right]$ & ($n$+$a$,$a$) & ($-n$,$a$) & $\sqrt{2}(w_{-n,a}+O(w_{-n,a-1}))$\\ 
	
	\hline
	$\left[\;\; \begin{tikzpicture}[baseline=(current bounding box.center),scale=0.75]
	\draw[->] (0,-0.25) to (1,0.75);
	\draw[->] (1,-0.25) to (0,0.75);
	\filldraw (.19, .55) circle[radius=2pt];
	\end{tikzpicture}\;\; \right]$
	 + $\left[\;\; \begin{tikzpicture}[baseline=(current bounding box.center),scale=0.75]
	\draw[->] (0,-0.25) to (1,0.75);
	\draw[->] (1,-0.25) to (0,0.75);
	\filldraw (.8, .55) circle[radius=2pt];
	\end{tikzpicture}\;\; \right]$ & ($3$,$1$) & ($-2$,$1$) & $2\sqrt{2}(w_{-2,1}-w_{-2,0})$\\[8pt]

	\hline
\end{tabular}
\caption{Correspondence between $\TrHEv$ and $W^-$}
\label{table:1}
\end{table} \\

\subsection{Degenerate affine Hecke-Clifford algebra}
% \textcolor{red}{More explanation, if we are to refer to semidirect products and such later}
 
 We recall the definition of the degenerate affine Hecke-Clifford algebra of type $A_{n-1}$, denoted $\HC_n$, which was first studied in \cite{Naz}. Let $\mathcal{C}\ell_n$ be the Clifford algebra with generators $c_1, \ldots, c_n$, subject to the relations: $$c_i^2 = -1 \qquad \text{for } 1\leq i \leq n,$$ $$c_i c_j = - c_j c_i \qquad \text{if } i\not=j.$$ The symmetric group $S_n$ has a natural action on $\mathcal{C}\ell_n$ by permuting the generators.

Define the Sergeev algebra, or finite Hecke-Clifford algebra of type $A_{n-1}$, to be the semidirect product $$\mathbb{S} := \mathcal{C}\ell_n \rtimes \mathbb{C} S_n$$ corresponding to this action.

The degenerate affine Hecke algebra, $\HC_n$, is isomorphic as a vector space to $\mathbb{S} \otimes \mathbb{C}[x_1, \ldots, x_n]$. It is an associative unital algebra over $\mathbb{C}[u]$, where $u$ is a formal parameter usually set to 1, generated by $s_1,s_2,...,s_{n-1}$, $x_1,x_2,...,x_n$, and $c_1,c_2,...,c_n$ subject to relations making $\mathbb{C}[x_1, \ldots, x_n]$, $C\ell_V$, and $\mathbb{C} S_n$ subalgebras, along with the additional relations: 
  
\begin{align*}
x_i c_i &= -c_i x_i, \quad x_i c_j = c_j x_i \quad (i\not= j) ,\\
\sigma c_i & = c_{\sigma(i)} \sigma \quad (1\leq i \leq n, \sigma \in S_n),\\
x_{i+1}s_{i}-s_{i}x_{i} &  =u(1-c_{i+1}c_{i})  , \\
x_{j}s_{i} &  =s_{i}x_{j}\quad (j\neq i,i+1) .
\end{align*}

It also has a $\mathbb{Z}/{2\mathbb{Z}}$ grading via $\deg(s_i)=\deg(x_i)=0$ and $\deg(c_i)=1$. This algebra is also called the affine Sergeev algebra, and later on we will see that it controls the endomorphisms of up strands in $\mathcal{H}_{tw}$.

\subsection{Trace of the degenerate affine Hecke-Clifford algebra}

The second author computes the trace (or zeroth Hochschild homology) of the even part of $\HC_n$ as a vector space in \cite{Mike}, where he gives an explicit description of a vector space basis for types $A$, $B$ and $D$. Here we recall the result for type $A$.

Let $I$ be the standard root system of type $A_{n-1}$, and let $W= S_n$ be the Weyl group. For a partition $\lambda = (\lambda_1, \lambda_2, \ldots, \lambda_k) \vdash n$, let $J_\lambda$ be the unique minimal subset of $I$ (up to conjugation by $W$) such that $W_{J_\lambda}$ contains an element of cycle type $\lambda.$ Let $w_{\lambda} \in W$ be the element $(1,\ldots, \lambda_1)(\lambda_1 +1, \ldots, \lambda_1 + \lambda_2)\ldots (n-\lambda_n +1, \ldots, n)$. Then $w_\lambda \in W_{J_\lambda}$. 

Let $V$ be the standard representation of $\mathfrak{sl}_n$, with basis $\{x_1, \ldots, x_n\}$. Denote by $V^2$ the vector space with basis $\{x_1^2, \ldots, x_n^2\}$.

Finally, fix a basis $\{f_{J_{\lambda};i}\}$ of the vector space $S((V^2)^{W_{J_\lambda}})^{N_W(W_{J_\lambda})}$, where $S(U)$ denotes the symmetric algebra of the vector space $U$, and $N_W(W_{J_\lambda})$ denotes the normalizer of the parabolic subgroup $W_{J_\lambda}$ in $W$. We have the following description of a basis for $\Tr(\HC_n)_{\overline{0}}$ in type $A$.

\begin{proposition}\label{cocenter HC}\cite[Theorem 5.4]{Mike}
 The set $\{w_{\lambda}f_{J_{\lambda;i}}\}_{\lambda\in\mathcal{OP}_n}$ is a basis of $\Tr(\HC_n)_{\overline{0}}$, where $\mathcal{OP}_n$ is the set of partitions of $n$ with all odd parts.
\end{proposition}

\begin{example}
Let $n=3$. Then we have $\mathcal{OP}_3=\{(1,1,1),(3)\}$. 

For $\lambda=(1,1,1)$, we have $w_\lambda = 1$ and $J_{\lambda} = \emptyset$, since  ${W}_{J_\lambda} = \{1\}$. Thus $N_{W}({W}_{J_\lambda}) = {S_3}$. So, we choose a basis $\{f_{J_{\lambda};i}\}$ of the vector space $S((V^2)^{W_J})^{N_W(W_J)} = \mathbb{C}[x_1^2 , x_2^2, x_3^2]^{S_3}$, i.e. the symmetric polynomials in 3 variables. We can take $\{f_{J_{\lambda};i}\} = \{s_\nu\}$, the Schur polynomials in 3 variables.

For $\lambda=(3)$, we have $w_{\lambda}=(123)$, a $3$-cycle in $S_3$. Thus $J_\lambda = I$, ${W}_{J_\lambda}=W$ and $N_{W}(W_{J_\lambda})=W$. Therefore $f_{J_{\lambda},i}$ is a basis of $\mathbb{C}[x_1^2+x_2^2+x_3^2]$, polynomials in the variable $(x_1^2+x_2^2+x_3^2)$ (in this case, the $N_W(W_{J_\lambda})$-invariance is superfluous).

Therefore a basis of $Tr(\HC_3)_{\overline{0}}$ is given by $\{ s_\nu\}\cup \{(123)(x_1^2+ x_2^2+ x_3^2)^{n}\}_{n\in \mathbb{N}}$ where $\{s_\nu\}$ are the Schur polynomials in 3 variables.

Note that this bases does not contain any classes indexed by partitions with even parts. Correspondingly, we will see that degree zero diagrams in $\TrHEv$ containing even cycles are zero.
\end{example}

\subsection{Distinguished elements $h_n$}
Define the elements: 
\begin{align*}
\h{n}{(x_1^{j_1} \cdots x_n^{j_n})  (c_1^{\epsilon_1} \cdots c_n^{\epsilon_n})} &:= \left[ \;\hackcenter{
\begin{tikzpicture}[scale=0.8]
 %% Separate lines by 0.8
  \draw[->] (3.2,0) .. controls (3.2,1.25) and (0,.25) .. (0,2)
     node[pos=0.85, shape=coordinate](X){};
  \draw[->] (0,0) .. controls (0,1) and (.8,.8) .. (.8,2);
  \draw[->] (.8,0) .. controls (.8,1) and (1.6,.8) .. (1.6,2);
  \draw[->] (2.4,0) .. controls (2.4,1) and (3.2,.8) .. (3.2,2);
  \node at (1.6,.35) {$\dots$};
  \node at (2.4,1.65) {$\dots$};
  %\filldraw  (X) circle (2pt);
  \filldraw  (.05,1.6) circle (2pt);
  \filldraw  (.78,1.6) circle (2pt);  % I put the dot .02 to the left to make it look nice
  \filldraw  (1.58,1.6) circle (2pt);
  \filldraw  (3.18,1.6) circle (2pt);
%%%
 \draw  (.35,1.2) circle (2.5pt);
  \draw  (.62,1.2) circle (2.5pt); 
  \draw  (1.425,1.2) circle (2.5pt);
  \draw  (3.025,1.2) circle (2.5pt);
  %% I cant remember a better way of doing this
  \node at (-0.18,1.7) {$\scs j_1$};
  \node at (.55,1.7) {$\scs j_2$};
  \node at (1.35,1.7) {$\scs j_3$};
   \node at (3.45,1.7) {$\scs j_n$};
 \node at  (.051,1.25) {$ \epsilon_1$};
  \node at  (.99,1.25) {$\epsilon_2$}; 
  \node at  (1.76,1.25) {$\epsilon_3$};
  \node at  (3.37,1.25) {$\epsilon_n$};
\end{tikzpicture}}
\; \right],
\\
\h{-n}{(x_1^{j_1} \cdots x_n^{j_n})  (c_1^{\epsilon_1} \cdots c_n^{\epsilon_n})} &:= \left[ \;\hackcenter{
\begin{tikzpicture}[scale=0.8]
 %% Separate lines by 0.8
  \draw[<-] (3.2,0) .. controls (3.2,1.25) and (0,.25) .. (0,2)
     node[pos=0.85, shape=coordinate](X){};
  \draw[<-] (0,0) .. controls (0,1) and (.8,.8) .. (.8,2);
  \draw[<-] (.8,0) .. controls (.8,1) and (1.6,.8) .. (1.6,2);
  \draw[<-] (2.4,0) .. controls (2.4,1) and (3.2,.8) .. (3.2,2);
  \node at (1.6,.35) {$\dots$};
  \node at (2.4,1.65) {$\dots$};
  %\filldraw  (X) circle (2pt);
  \filldraw  (.05,1.6) circle (2pt);
  \filldraw  (.78,1.6) circle (2pt);  % I put the dot .02 to the left to make it look nice
  \filldraw  (1.58,1.6) circle (2pt);
  \filldraw  (3.18,1.6) circle (2pt);
%%%
 \draw  (.35,1.2) circle (2.5pt);
  \draw  (.62,1.2) circle (2.5pt); 
  \draw  (1.425,1.2) circle (2.5pt);
  \draw  (3.025,1.2) circle (2.5pt);
  %% I cant remember a better way of doing this
  \node at (-0.18,1.7) {$\scs j_1$};
  \node at (.55,1.7) {$\scs j_2$};
  \node at (1.35,1.7) {$\scs j_3$};
   \node at (3.45,1.7) {$\scs j_n$};
 \node at  (.051,1.25) {$ \epsilon_1$};
  \node at  (.99,1.25) {$\epsilon_2$}; 
  \node at  (1.76,1.25) {$\epsilon_3$};
  \node at  (3.37,1.25) {$\epsilon_n$};
\end{tikzpicture}}
\; \right],\nonumber
\end{align*}

where $\epsilon_i \in \{0,1\}$. In both of these elements, we consider the hollow dots to be descending in height from left to right, so that the dot labeled $\epsilon_1$ is the highest.\\

\begin{remark}
These elements are analogues to those denoted $h_{\pm n} \otimes (x_1^{j_1} \cdots x_n^{j_n}) $ in \cite{CLLS}. 
\end{remark}

Additionally, set $$\h{n}{\sum x_i^{j_i}} = \sum \h{n}{x_i^{j_i}}.$$

\begin{lemma} \label{basic hn facts} For $n\geq 1$ and $1\leq i \leq n-1$ we have \begin{enumerate} \item $\h{\pm n}{x_i}  =\h{\pm n}{x_{i+1}}   \pm \h{\pm i}{}  \h{\pm  (n-i)}{}  $. 
\item $\h{\pm n}{ x_i   c_j} = - \h{\pm n}{x_{i+1}c_{j+1}} $. \end{enumerate}
\end{lemma} 
\begin{proof} Part (1) is just \cite[Lemma 14]{CLLS}, except our solid dot sliding relation through crossing involves an extra term with hollow bubbles. But cycles with single hollow dot are zero since sending the hollow dot around the annulus gives us the same diagram with a negative sign. For the above calculations, our $n$-cycles split into smaller cycles with single hollow dot at least on one of them. The proof of part 2 depends on the relative position of $i$ and $j$, but is a straightforward computation. \end{proof}

Let $w\in S_n$, and define the elements:
$$f_{w; j_1, \ldots, j_n; \epsilon_1, \ldots, \epsilon_n}= 
\hackcenter{\begin{tikzpicture}
    \draw[very thick][->] (-.55,0) -- (-.55,1.5);
    \draw[very thick][->] (.55,0) -- (.55,1.5);
    \draw[fill=white!20,] (-.8,.4) rectangle (.8,.8);
    \node () at (0,.55) {$w$};
    \node () at (0,1.25) {$\cdots$};
    \node () at (0,.25) {$\cdots$};
\filldraw  (-.55,1.25) circle (2pt);
\filldraw  (.55,1.25) circle (2pt);
\draw  (-.55,.9) circle (2pt);
\draw  (.55,.9) circle (2pt);
\draw (-.75,.9) node {$\epsilon_1$};
\draw (.8,.9) node {$\epsilon_n$};
\draw (-.75,1.25) node {$j_1$};
\draw (.8,1.25) node {$j_n$};
\end{tikzpicture}}
$$
and
$$f_{w; j_1, \ldots, j_n; \epsilon_1, \ldots, \epsilon_n}= 
\hackcenter{\begin{tikzpicture}
    \draw[very thick][<-] (-.55,0) -- (-.55,1.5);
    \draw[very thick][<-] (.55,0) -- (.55,1.5);
    \draw[fill=white!20,] (-.8,.4) rectangle (.8,.8);
    \node () at (0,.55) {$w$};
    \node () at (0,1.25) {$\cdots$};
    \node () at (0,.25) {$\cdots$};
\filldraw  (-.55,1.25) circle (2pt);
\filldraw  (.55,1.25) circle (2pt);
\draw  (-.55,.9) circle (2pt);
\draw  (.55,.9) circle (2pt);
\draw (-.75,.9) node {$\epsilon_1$};
\draw (.8,.9) node {$\epsilon_n$};
\draw (-.75,1.25) node {$j_1$};
\draw (.8,1.25) node {$j_n$};
\end{tikzpicture}}.
$$

\begin{lemma}\label{red to hn lm} Let $w\in S_n$ and $(n_1, \ldots, n_r)$ be a composition of $n$. Then $$[f_{\pm w; j_1, \ldots, j_n; \epsilon_1, \ldots, \epsilon_n}] = \sum d_{n_1, \ldots, n_r} \h{n_1}{  p_{n_1} c_{n_1}} \ldots \h{n_r}{ p_{n_r}  c_{n_r}}$$ for constants $d_{n_1, \ldots, n_r} \in \mathbb{C}$, polynomials $p_{n_i}$ in $i$ variables, and elements $c_{n_i}$ consisting of at most $i$ Clifford generators (e.g. $c_{n_3} = \{ c_1^{\epsilon_1} c_2^{\epsilon_2} c_3^{\epsilon_3} | \epsilon_i \in \{0,1\}\}$). \end{lemma}

\begin{proof}We proceed by induction on $\sum {\epsilon_i}$. The base case is $\sum \epsilon_i = 0$; then $$[f_{\pm w; j_1, \ldots, j_n; \epsilon_1, \ldots, \epsilon_n}] = [f_{\pm w; j_1, \ldots, j_n}] $$ and we apply \cite[Lemma 15]{CLLS}. 

Now assume the statement is true for $\sum \epsilon_i = k$ for all $k<m\leq n$. Take $(\epsilon_1, \ldots, \epsilon_n)$ so that $\sum \epsilon_i = m$. Choose $g\in S_n$ such that $gwg^{-1} = w_\lambda$, where $\lambda$ is the cycle type of $w$ (so $gwg^{-1} = (s_1 \ldots s_{n_1 -1})\ldots (s_{n_1 + \ldots + n_{r-1}} \ldots s_{n_1 +\ldots +n_r -1})$).  Let $p= x_1^{j_1} \ldots x_n^{j_n}$ and $c= c_1^{\epsilon_1} \ldots c_n^{\epsilon_n}$.  Then we have $$f_{\pm w; j_1, \ldots, j_n; \epsilon_1, \ldots, \epsilon_n} = pcw = (-1)^\epsilon cpw$$ where $$\epsilon = \sum_{\epsilon_i = 1} j_i.$$ 

Thus conjugating by $g$ gives that \begin{align*} gpcwg^{-1}& = (-1)^{\epsilon} gcpwg^{-1} \\ &=(-1)^\epsilon (g.c) gpwg^{-1} \\&= (-1)^{\epsilon} \left[ (g.c)(g.p)gwg^{-1} + (g.c) p_L wg^{-1}\right], \end{align*} where $p_L$ is a polynomial of degree less than $j_1 + \ldots + j_n$. Note that $gwg^{-1}$ is a product of cycles, so the first term in the above expression has the correct form. In the second term, we have $\{i | \epsilon_{g(i)} =1\} \leq m$ (strict inequality can occur if $g$ has fixed points). If $\{i | \epsilon_{g(i)} =1\} < m$, we are done by induction, so assume that we have equality.

Now repeat the process on the second term, choosing a $g' \in S_n$ such that $g'(wg^{-1}) (g')^{-1}$ is a product of cycles, and conjugating $(g.c) p_L wg^{-1}$. Each application of this process results in one term in which the symmetric group element is a product of cycles (which has the desired form), and one term with the degree of the polynomial part strictly lesser and the degree of the Clifford part weakly lesser.

 If the degree of the Clifford part ever strictly decreases, we are done. If not, the conjugation will eventually reduce the degree of the polynomial part to 0, so we have an element of the form $c' \sigma$, $c' \in C\ell_n$ and $\sigma \in S_n$. Choose a $g''\in S_n$ such that $g'' \sigma (g'')^{-1}$ is a product of cycles; then $$ g'' c \sigma (g'')^{-1} =  (g''c) g'' \sigma (g'')^{-1}.$$ This now has the desired form.
\end{proof}

\begin{proposition} \label{red to hn} Let $w\in S_n$ and $(n_1, \ldots, n_r)$ a composition of $n$. Then $$[f_{\pm w; j_1, \ldots, j_n; \epsilon_1, \ldots, \epsilon_n}] = \sum d_{n_1, \ldots, n_r} \h{ \pm n_1}{x_1^{\ell_1}   c_1^{k_1}}\ldots \h{\pm n_r}{x_1^{\ell_r} c_{1}^{k_r}}$$ where $d_{n_1, \ldots, n_r} \in \mathbb{C}$ and $\ell_1, \ldots, \ell_r, k_1, \ldots, k_r \in \mathbb{N}$. \end{proposition}
\begin{proof} This follows immediately from the preceding lemmas. \end{proof} 

Proposition \ref{red to hn} allows us to write any element in $\TrHPl$ or $\TrHMi$ as a linear combination of the elements $\h{n}{}$. We will therefore direct our attention to these elements in future computations.

\subsection{Gradings in $\TrHEv$}

The next lemma follows from diagrammatic computations in the next section. We record it here for convenience of terminology.
\begin{lemma}\label{rank filtration} The algebra $\TrHEv$ is $\mathbb{Z}$-filtered where $\deg\left(\h{n}{x_1^{2a}}\right) \leq n$ for any $a\geq 0$. \end{lemma}

This is called the rank filtration. Denote by $\TrHPl$ (resp. $\TrHMi$) the subalgebra of $\TrHEv$ generated by $\h{n}{x_1^{2a}}$, $n\geq 1$  (resp. $n\leq 1$).

\begin{lemma}\label{dot filtration} The algebra $\TrHEv$ is $\mathbb{Z}^{\geq 0}$-filtered where $\deg\left(\h{n}{ x_1^{2a}}\right)\leq a$ for any $a \geq 0$. \end{lemma}
\begin{proof} Dots can slide through crossings modulo a correction term containing fewer dots. \end{proof}

This is called the dot filtration, and corresponds to the differential filtration (given by $\deg(w_{\ell, k}) = k$) in $W^-$. 

These filtrations are compatible, so $\TrHEv$ is $(\mathbb{Z}\times \mathbb{Z}^{\geq 0})$-filtered with $\h{n}{ x_1^{2a}}$ in bidegree $(n, a)$. For $\omega \in \{>,<,0\}$ denote the associated graded object by $\gr \TrHOm$. Define a generating series for the graded dimension of $\gr \TrHOm$ by $$P_{\gr \TrHOm}(t,q) = \sum_{r\in \mathbb{Z}} \sum_{k\in \mathbb{Z}, k\geq 0 } \dim \gr \TrHOm[r,k] t^r q^k.$$

The following is an easy calculation using Proposition \ref{triangularDecomposition} and Proposition \ref{cocenter HC}. They are not used in the proof of the main result, but we record them here for convenience.
\begin{proposition} \label{grDim Tr} The graded dimensions of $\gr \TrHPl$ and $\gr \TrHMi$ are given by: 
	$$P_{\grD^>} = \prod_{r\geq 0} \prod_{k>0} \frac{1}{1-t^{2r+1} q^k};$$
	$$P_{\grD^<} = \prod_{r\leq 0} \prod_{k>0} \frac{1}{1-t^{2r-1} q^k}. $$
\end{proposition}

Note that the rank grading and dot gradings are shifted by 1 for clockwise bubbles (so $d_2$ is in bidegree $(1,2)$ and $d_4$ is in bidegree $(1,3)$). This is a consequence of the decomposition formula in Lemma \ref{bubbleDecomp}. 

%--------------------------------------------------BUBBLES--------------------------------------------------%

\section{Bubbles}

We investigate the endomorphisms of 1 in $\TrH$, known as bubbles. We prove that all bubbles can be written in terms of clockwise bubbles, and deduce formulas for moving bubbles past strands in the trace.

\subsection{Definition and basic properties}
 
 Elements of $\End_{\mathcal{H}_{tw}}(1)$ are $\mathbb{C}$-linear combinations of possibly intersecting or nested closed diagrams, which may have dots. We can always separate the nested pieces, and resolve any crossing that occur between different closed diagrams using the defining relations and end up with non intersecting, not nested closed oriented diagrams. Each one can be deformed into an oriented circle, possibly with dots, via an isotopy. A single closed, oriented, non self intersecting diagram is called a bubble. They are the building blocks of  endomorphisms of the identity object in $\mathcal{H}_{tw}$.
 
We define 
\begin{equation*}
    \bar{d}_{k,l}:=\hspace{6pt}
    \begin{tikzpicture}[baseline=(current bounding box.center)]
    \draw[->] (3,2) arc (-180:180:5mm);
\fill (3.95,2.2) circle [radius=2pt];
\draw (3.95,1.8) circle [radius=2pt];
\node at (4.2,1.8) {$l$};
\node at (4.2,2.2) {$k$};
    \end{tikzpicture}
    \hspace{0.5cm}
    \text{and}
    \hspace{0.5cm}
    d_{k,l}:=\hspace{6pt}
    \begin{tikzpicture}[baseline=(current bounding box.center)]
    \draw[<-] (3,2) arc (-180:180:5mm);
\fill (3.95,2.2) circle [radius=2pt];
\draw (3.95,1.8) circle [radius=2pt];
\node at (4.2,1.8) {$l$};
\node at (4.2,2.2) {$k$};
    \end{tikzpicture}
    \hspace{0.5cm}
    \text{for}
    \hspace{2mm} k,l\in\mathbb{Z}_{\geq 0}.
\end{equation*}

   Given any closed diagram with any configuration of dots, it is possible to collect the hollow dots and the solid dots together, possibly after multiplying the diagram by $-1$, by using relation \eqref{anticommute}. Solid dots move freely along caps and cups, and hollow dots may capture a negative sign while moving along caps or cups, depending on the orientation. After regrouping, we may assume that the dots are placed on the right middle side of the diagram as above.
   
   Moreover, using the left two equations in relation \eqref{d01=0}, we can erase a pair of hollow dots, possibly by changing the sign of the diagram.
   
   Therefore the set $\{d_{k,l},\bar{d}_{k,l}|k\in\mathbb{Z}_{\geq0}, l\in\{0,1\}\}$ is a spanning set for $\End_{\mathcal{H}_{tw}}(1)$.

In our defining relations, we have that 

\begin{equation*}
\bar{d}_{0,0}=\hspace{6pt}
\begin{tikzpicture}[baseline=(current bounding box.center)]
\draw[->] (3,2) arc (-180:180:5mm);
%\node at (3.5,1.2) {$i$};
%\node at (3.5,2.8) {$i$};
\end{tikzpicture}\hspace{6pt}
=1
\hspace{1cm} \text{and}
\hspace{1cm}
\bar{d}_{0,1}=\hspace{6pt}
\begin{tikzpicture}[baseline=(current bounding box.center)]
\draw[->] (3,2) arc (-180:180:5mm);
\draw (4,2) circle [radius=2pt];
%\node at (3.5,1.2) {$i$};
%\node at (3.5,2.8) {$i$};
\end{tikzpicture}\hspace{6pt}
=0.
\end{equation*}
Further, we have the following. 
\begin{lemma}\label{hollow dots}
We have that $\bar{d}_{k,1}=0$ and $d_{k,1}=0$ for all non-negative integers $k$.
\end{lemma}

\begin{proof}
An example computation shows that
\begin{equation*}
    \bar{d}_{1,1}=\hspace{6pt}\begin{tikzpicture}[baseline=(current bounding box.center)]
    \draw[->] (3,2) arc (-180:180:5mm);
\draw (3.95,2.2) circle [radius=2pt];
\fill (3.95,1.8) circle [radius=2pt];
    \end{tikzpicture}\hspace{6pt}
    =
    \hspace{6pt}-\hspace{2mm}
    \begin{tikzpicture}[baseline=(current bounding box.center)]
    \draw[->] (3,2) arc (-180:180:5mm);
\fill (3.95,2.2) circle [radius=2pt];
\draw (3.95,1.8) circle [radius=2pt];
\end{tikzpicture}\hspace{6pt}
   =
   \hspace{6pt}-\hspace{2mm}
   \begin{tikzpicture}[baseline=(current bounding box.center)]
   \draw[->] (3,2) arc (-180:180:5mm);
\fill (3.75,1.6) circle [radius=2pt];
\draw (3.95,1.8) circle [radius=2pt];
   \end{tikzpicture}\hspace{6pt}
   =
   \hspace{6pt}-\bar{d}_{1,1}=0,
\end{equation*}
where in the second equality, negative sign comes from relation (11), and the third equality comes from sliding the solid dot around.

More generally, if we have $k$ solid dots where $k$ is an even integer, then sliding the hollow dot around the circle and passing it through $k$ solid dots multiplies the diagram by $(-1)^{k+1}$, so the diagram is zero.
If $k$ is an odd number, sliding a solid dot around the circle and passing it through a hollow dot catches a minus sign, so these diagrams are zero as well.

These arguments do not depend on the orientation of the bubble, hence the result follows.
\end{proof}

From now on, we will assume that the second index in $\bar{d}_{k,l}$ and $d_{k,l}$ is always zero. We will omit it from our notation and write $d_k$ instead of $d_{k,0}$, and  $\bar{d}_k$ instead of $\bar{d}_{k,0}$.

\begin{lemma}\label{odd dots}
We have that $d_{2n+1}=\bar{d}_{2n+1}=0$ for all non-negative integers $n$.
\end{lemma}

\begin{proof}

Note that
\begin{equation*}
\bar{d}_1=\hspace{6pt}
\begin{tikzpicture}[baseline=(current bounding box.center)]
\draw[->] (3,2) arc (-180:180:5mm);
\fill (4,2) circle [radius=2pt];
%\node at (3.5,1.2) {$i$};
%\node at (3.5,2.8) {$i$};
\end{tikzpicture}\hspace{6pt}
=
\hspace{6pt}
\begin{tikzpicture}[baseline=(current bounding box.center)]
\draw[->] (3,2) arc (-180:180:5mm);
\draw (4,2) circle [radius=2pt];
\draw (3.95,2.2) circle [radius=2pt];
\fill (3.95,1.8) circle [radius=2pt];
%\node at (3.5,1.2) {$i$};
%\node at (3.5,2.8) {$i$};
\end{tikzpicture}\hspace{6pt}
=
\hspace{6pt}-\hspace{6pt}
\begin{tikzpicture}[baseline=(current bounding box.center)]
\draw[->] (3,2) arc (-180:180:5mm);
\draw (4,2) circle [radius=2pt];
\draw (3.05,2.2) circle [radius=2pt];
\fill (3.95,1.8) circle [radius=2pt];
%\node at (3.5,1.2) {$i$};
%\node at (3.5,2.8) {$i$};
\end{tikzpicture}\hspace{6pt}
=
\hspace{6pt}
\begin{tikzpicture}[baseline=(current bounding box.center)]
\draw[->] (3,2) arc (-180:180:5mm);
\fill (4,2) circle [radius=2pt];
\draw (3.95,2.2) circle [radius=2pt];
\draw (3.05,1.8) circle [radius=2pt];
%\node at (3.5,1.2) {$i$};
%\node at (3.5,2.8) {$i$};
\end{tikzpicture}\hspace{6pt}
=
\hspace{6pt}
\begin{tikzpicture}[baseline=(current bounding box.center)]
\draw[->] (3,2) arc (-180:180:5mm);
\fill (4,2) circle [radius=2pt];
\draw (3.95,2.2) circle [radius=2pt];
\draw (3.95,1.8) circle [radius=2pt];
%\node at (3.5,1.2) {$i$};
%\node at (3.5,2.8) {$i$};
\end{tikzpicture}\hspace{6pt}\end{equation*} $$
=
\hspace{6pt}-\hspace{6pt}
\begin{tikzpicture}[baseline=(current bounding box.center)]
\draw[->] (3,2) arc (-180:180:5mm);
\draw (4,2) circle [radius=2pt];
\draw (3.95,2.2) circle [radius=2pt];
\fill (3.95,1.8) circle [radius=2pt];
%\node at (3.5,1.2) {$i$};
%\node at (3.5,2.8) {$i$};
\end{tikzpicture}\hspace{6pt}
=
\hspace{6pt}-\hspace{6pt}
\begin{tikzpicture}[baseline=(current bounding box.center)]
\draw[->] (3,2) arc (-180:180:5mm);
\fill (4,2) circle [radius=2pt];
%\node at (3.5,1.2) {$i$};
%\node at (3.5,2.8) {$i$};
\end{tikzpicture}\hspace{6pt}
=\hspace{6pt}0.
$$
The same arguments works for any odd number of solid dots and works for clockwise oriented bubbles.
\end{proof}

\begin{lemma} \label{bubbleDecomp}
We have that $$\ds \bar{d}_{2n}=\sum_{2a+2b=2n-2}\begin{tikzpicture}[baseline=(current bounding box.center)]
    \draw[->] (0,0) arc (0:180:5mm);
    \draw (0,0) arc (360:180:5mm);
    \draw[<-] (0.1,0) arc (-180:0:5mm);
    \draw (0.1,0) arc (180:0:5mm);
    \fill (-0.2,0.4) circle[radius=2pt];
    \fill (0.9,0.4) circle[radius=2pt];
    \node at (-0.04,0.6) {$2a$};
    \node at (1.05,0.6) {$2b$};
    \end{tikzpicture}
    =\sum_{2a+2b=2n-2}\bar{d}_{2a}d_{2b}$$
    for any integers $a,b$ and $n\geq1$.
\end{lemma}

\begin{proof}
For the $n=1$ case, we have the following computation:

\begin{align*}
    \bar{d}_2=\hspace{6pt}
    \begin{tikzpicture}[baseline=(current bounding box.center)]
        \draw[->] (3,2) arc (-180:180:5mm);
\fill (3.95,2.2) circle [radius=2pt];
\fill (3.95,1.8) circle [radius=2pt];
    \end{tikzpicture}\hspace{6pt}
    &=
    \hspace{6pt}
    \begin{tikzpicture}[baseline=(current bounding box.center)]
    \draw[->] (0,0) arc (10:180:5mm);
    \draw (0,-0.18) arc (350:180:5mm);
    \draw (0.1,0) arc (170:-170:5mm);
    \draw (-0.02,-0.2) to (0.1,0.01);
    \draw (-0.01,0.02) to (0.11,-0.2);
    \fill (-0.1,0.23) circle[radius=2pt];
    \end{tikzpicture}\hspace{6pt}
    \\&=
    \hspace{6pt}
    \begin{tikzpicture}[baseline=(current bounding box.center)]
    \draw[->] (0,0) arc (10:180:5mm);
    \draw (-1,-0.08) arc (180:350:5mm);
    \draw (0.1,0) arc (170:-170:5mm);
    \draw (-0.02,-0.2) to (0.1,0.01);
    \draw (-0.01,0.02) to (0.11,-0.2);
    \fill (0.2,0.23) circle[radius=2pt];
    \end{tikzpicture}
    \hspace{6pt}
    +\hspace{6pt}
    \begin{tikzpicture}[baseline=(current bounding box.center)]
    \draw[->] (0,0) arc (0:180:5mm);
    \draw (0,0) arc (360:180:5mm);
    \draw[<-] (0.1,0) arc (-180:0:5mm);
    \draw (0.1,0) arc (180:0:5mm);
    \end{tikzpicture}
    \hspace{6pt}
    +\hspace{6pt}
    \begin{tikzpicture}[baseline=(current bounding box.center)]
    \draw[->] (0,0) arc (0:180:5mm);
    \draw (0,0) arc (360:180:5mm);
    \draw[<-] (0.1,0) arc (-180:0:5mm);
    \draw (0.1,0) arc (180:0:5mm);
    \draw (-0.04,-0.15) circle[radius=2pt];
    \draw (0.13,0.1) circle[radius=2pt];
    \end{tikzpicture}\hspace{6pt}
    =\hspace{6pt}d_0
\end{align*}
where the first diagram on right hand side is zero since it contains a left curl, the second term is $\bar{d}_0d_0=d_0$ and the last term is zero by Lemma \ref{hollow dots}.

 For general $n$, if you replace one of the solid dots with a right-twist curl, and slide the remaining $2n-1$ dots through the crossings using relations \ref{dotSlide: bottomLeft} and \ref{dotSlide: topLeft} repeatedly, we will get many resolution terms, consisting of a sum of product of counterclockwise and clockwise bubbles, some with only solid dots, some with hollow dots as well. The terms with hollow dots are zero, and so are the terms with an odd number of solid dots. Also, the figure eight shape contains a left twist curl, so it is zero as well, which proves the statement. 
\end{proof}

%-------------------ALGEBRAIC INDEPENDENCE OF BUBBLES-----------------------------%

\subsection{Algebraic independence of bubbles}

 A categorified Fock space action for $\mathcal{H}_{tw}$ is described in \cite[Section 6.3]{CS}.  $\mathcal{H}_{tw}$ acts on the category $\mathfrak{S}$, whose objects are induction and restriction functors between $\mathbb{Z}/2\mathbb{Z}$-graded finite dimensional $\mathbb{S}_n$-modules, for all $n\geq1$. Morphisms of $\mathfrak{S}$ are natural transformations between the induction and restriction functors. 

Following Khovanov's approach from \cite{Khovanov}, let $\mathfrak{S}_n$ be the subcategory of $\mathfrak{S}$, whose objects start with induction or restriction functors from $\mathbb{Z}/2\mathbb{Z}$-graded finite dimensional $\mathbb{S}_n$-modules. For every $n\in\mathbb{Z}_{\geq1}$, we have a functor $F_n:\mathcal{H}_{tw}\rightarrow\mathfrak{S}_n$ sending $P$ to $\Ind_n^{n+1}$ and sending $Q$ to $\Res_n^{n-1}$.

Note that $F_n$ sends $\End_{\mathcal{H}_{tw}}(1)$ to the center of $\mathbb{S}_n$, which is same as the center of $\mathbb{C} [S_n]$. 

Explicit descriptions of the actions of a crossing, a cup and a cap are provided in \cite{CS}. We would like to study the action of clockwise bubbles to show their algebraic independence. Note that $d_{2k}$  is obtained as composition of a cup, $k$ copies of $\h{1}{x_1}$ and a cap.
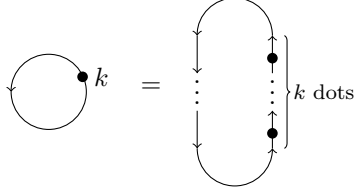
\begin{equation*}
\begin{tikzpicture}[baseline=(current bounding box.center)]
\draw[->] (3,2) arc (-180:180:5mm);
\fill (3.95,2.2) circle [radius=2pt];
\node at (4.2,2.2) {$k$};
\end{tikzpicture}\hspace{6pt}
=
\hspace{6pt}
\begin{tikzpicture}[baseline=(current bounding box.center)]
\draw[->] (1,-3.5) arc (0:180:5mm);
\fill (1,-3.8) circle[radius=2pt];
\draw[->] (1,-4) to (1,-3.5);
\draw[->] (0,-3.5) to (0,-4);
\node at (0,-4.15) {\vdots};
\node at (1,-4.15) {\vdots};
\fill (1,-4.8) circle[radius=2pt];
\draw[->] (1,-5) to (1,-4.5);
\draw[->] (0,-4.5) to (0,-5);
\draw[->] (0,-5) arc (-180:0:5mm);
\draw [decorate,decoration={brace,amplitude=2pt},xshift=-4pt,yshift=0pt] (1.3,-3.5) -- (1.3,-5) node [black,midway,xshift=15pt] {\footnotesize $k$ dots};
\end{tikzpicture}
\end{equation*}

Therefore to study the action of $d_{2k}$, we need to know the action of $\h{1}{x_1}$ in addition to actions of cups and caps. Now $\h{1}{x_1}$ is defined as a combination of caps, cups and crossings:
\begin{equation}
\begin{tikzpicture}[baseline=(current bounding box.center),scale=1.15]
\draw[->] (0,0) to (0,1);
\fill  (0,.5) circle[radius=2pt];
\end{tikzpicture}\hspace{6pt}
=
\hspace{6pt}
\begin{tikzpicture}[baseline=(current bounding box.center),scale=0.75]
\draw[->] (0,0) to (0,.5);
\draw[<-] (0.5,0.5) arc (-180:0:3mm);
\draw[->] (0,.5) to (.5,1);
\draw[->] (.5,.5) to (0,1);
\draw[->] (1.1,1) to (1.1,.5);
\draw[->] (0,1) to (0,1.5);
\draw[<-] (1.1,1) arc (0:180:3mm);
\end{tikzpicture}
\end{equation}

Using the explicit description of Fock space representation of $\mathcal{H}$ in \cite{CS}, we compute the required actions. These computations give that $\h{-1}{x_1}$ acts by sending $$1 \mapsto J_{n+1}=\ds\sum_{i=1}^{n}(1-c_{n+1}c_i)(i,n+1).$$ This is the $(n+1)$-st even Jucys-Murphy element. Therefore $\h{1}{x_1^2}$ acts by multiplication by $J_{n+1}^2$. This is analogous to the untwisted case where the same element acts as multiplication by a Jucys-Murphy element.

 Finally, the action of the bubble $d_{2k}$ is given by multiplication by $$\ds\sum_{i=1}^{n}(i\leftrightarrow n+1)J_{n+1}^{2k}(n+1\leftrightarrow i)-c_n(i\leftrightarrow n+1)J_{n+1}^{2k}(n+1\leftrightarrow i)c_1,$$ 
 where $(i\leftrightarrow n)$ denotes the $n$-cycle $s_is_{i+1}....s_{n-1}$.

Here we can apply the filtration argument on the number of disturbances of permutations as done in \cite[Section 4]{Khovanov} to obtain the following.

\begin{proposition}\label{bubbles}
The elements $\{d_{2k}\}_{k\geq0}$ are algebraically independent, i.e. there is an isomorphism $$\End_{\mathcal{H}_{tw}}(1)\cong k[d_0,d_2,d_4,...].$$
\end{proposition}
Therefore the bubbles are algebraically independent, and they form of a copy of a polynomial ring in infinitely many variables.

\subsection{Counter-clockwise bubble slide moves}

In order to describe $\TrH$ as a vector space, it would be convenient to have a standard form for our diagrams in the trace. In particular, we want to collect all the bubbles appearing in a diagram on the rightmost part of the diagram. In order to do so, we must describe how bubbles slide through upward and downward strands. Note that since we can work with local relations, the bubbles don't have to interact with solid dots or crossings, they can simply slide through under a crossing or under a solid dot.

All calculations in this section take place in the trace, though we omit the brackets in some situations for readability.

\begin{lemma}\label{bubbleslide to curl}
We have that $\ds[\bar{d}_{2n},\h{1}{}]=2\sum_{k=1}^n \left[\; \begin{tikzpicture}[baseline=(current bounding box.center),scale=1.75]
\draw (1,0) to [out=90, in=-75](0.95,0.5);
\draw (0.95,0.5) arc (5:355:2mm);
\draw[->] (0.95,0.46) to [out=75, in=270] (1,1);
\fill (0.8,0.67) circle (1pt);
\node at (0.8,0.8) {\small 2k-1};
\end{tikzpicture}\;\right]$ in $\TrHEv$ for any positive integer $n$.
\end{lemma}

\begin{proof}
The proof is a direct computation, given below:
\begin{align*}
\begin{tikzpicture}[baseline=(current bounding box.center),scale=0.75]
\draw[->] (2.5,1) arc (-180:180:5mm);
\draw[->] (4,0) to (4,2);
\fill (3,1.5) circle (2pt);
\node at (3,1.8) {\small $2n$};
\end{tikzpicture}\hspace{6pt}
 =
 \hspace{6pt}
\begin{tikzpicture}[baseline=(current bounding box.center),scale=0.75]
\draw[->] (2.5,1) arc (-180:180:5mm);
\draw[->] (3.3,0) to (3.3,2);
\fill (3,1.5) circle (2pt);
\node at (2.5,1.8) {\small $2n$};
\end{tikzpicture}\hspace{6pt}
&=
\hspace{6pt}
\begin{tikzpicture}[baseline=(current bounding box.center),scale=0.75]
\draw[->] (2.5,1) arc (-180:180:5mm);
\draw[->] (3.3,0) to (3.3,2);
\fill (3,1.5) circle (2pt);
\fill (3.5,1) circle (2pt);
\node at (2.7,1.8) {\small $2n$-$1$};
\node at (3.6,1) {};
\end{tikzpicture}\hspace{6pt}
+
\hspace{6pt}
\begin{tikzpicture}[baseline=(current bounding box.center),scale=1.60]
\draw (1,0) to [out=90, in=-75](0.95,0.5);
\draw (0.95,0.5) arc (5:355:2mm);
\draw[->] (0.95,0.46) to [out=75, in=270] (1,1);
\fill (0.8,0.67) circle (1pt);
\node at (0.75,0.8) {\small $2n$-$1$};
\end{tikzpicture}\hspace{6pt}
-
\hspace{6pt}
\begin{tikzpicture}[baseline=(current bounding box.center),scale=1.60]
\draw (1,0) to [out=90, in=-75](0.95,0.5);
\draw (0.95,0.5) arc (5:355:2mm);
\draw[->] (0.95,0.46) to [out=75, in=270] (1,1);
\fill (0.7,0.67) circle (1pt);
\draw (0.85,0.65) circle [radius=1pt];
\draw (0.98,0.57) circle (1pt);
\node at (0.65,0.8) {\small $2n$-$1$};
\end{tikzpicture}\hspace{6pt}
\\ &=
\hspace{6pt}
\begin{tikzpicture}[baseline=(current bounding box.center),scale=0.75]
\draw[->] (2.5,1) arc (-180:180:5mm);
\draw[->] (3.3,0) to (3.3,2);
\fill (3,1.5) circle (2pt);
\fill (3.5,1) circle (2pt);
\node at (2.7,1.8) {\small $2n$-$1$};
\node at (3.6,1) {};
\end{tikzpicture}\hspace{6pt}
+
\hspace{6pt}2
\begin{tikzpicture}[baseline=(current bounding box.center),scale=1.60]
\draw (1,0) to [out=90, in=-75](0.95,0.5);
\draw (0.95,0.5) arc (5:355:2mm);
\draw[->] (0.95,0.46) to [out=75, in=270] (1,1);
\fill (0.8,0.67) circle (1pt);
\node at (0.75,0.8) {\small $2n$-$1$};
\end{tikzpicture}\hspace{6pt}
\\ &=
\hspace{6pt}
\begin{tikzpicture}[baseline=(current bounding box.center),scale=0.75]
\draw[->] (2.5,1) arc (-180:180:5mm);
\draw[->] (3.3,0) to (3.3,2);
\fill (3,1.5) circle (2pt);
\fill (3.5,1) circle (2pt);
\node at (2.7,1.8) {\small $2n$-$2$};
\node at (3.7,1) {\small $2$};
\end{tikzpicture}\hspace{6pt}
+
\hspace{6pt}
\begin{tikzpicture}[baseline=(current bounding box.center),scale=1.60]
\draw (1,0) to [out=90, in=-75](0.95,0.5);
\draw (0.95,0.5) arc (5:355:2mm);
\draw[->] (0.95,0.46) to [out=75, in=270] (1,1);
\fill (0.8,0.67) circle (1pt);
\node at (0.65,0.8) {\small $2n$-$2$};
\fill (1,0.75) circle (1pt);
\node at (1.1,0.75) {};
\end{tikzpicture}\hspace{6pt}
-
\hspace{6pt}
\begin{tikzpicture}[baseline=(current bounding box.center),scale=1.60]
\draw (1,0) to [out=90, in=-75](0.95,0.5);
\draw (0.95,0.5) arc (5:355:2mm);
\draw[->] (0.95,0.46) to [out=75, in=270] (1,1);
\fill (0.7,0.67) circle (1pt);
\draw (0.85,0.65) circle (1pt);
\draw (0.98,0.57) circle (1pt);
\node at (0.65,0.8) {\small $2n$-$2$};
\fill (1,0.75) circle (1pt);
\node at (1.1,0.75) {};
\end{tikzpicture}\hspace{6pt}
+
\hspace{6pt}2
\begin{tikzpicture}[baseline=(current bounding box.center),scale=1.60]
\draw (1,0) to [out=90, in=-75](0.95,0.5);
\draw (0.95,0.5) arc (5:355:2mm);
\draw[->] (0.95,0.46) to [out=75, in=270] (1,1);
\fill (0.8,0.67) circle (1pt);
\node at (0.75,0.8) {\small $2n$-$1$};
\end{tikzpicture}\hspace{6pt}
\\&=
\hspace{6pt}
\begin{tikzpicture}[baseline=(current bounding box.center),scale=0.75]
\draw[->] (2.5,1) arc (-180:180:5mm);
\draw[->] (3.3,0) to (3.3,2);
\fill (3,1.5) circle (2pt);
\fill (3.5,1) circle (2pt);
\node at (2.7,1.8) {\small $2n$-$2$};
\node at (3.7,1) {\small $2$};
\end{tikzpicture}\hspace{6pt}
+
\hspace{6pt}2
\begin{tikzpicture}[baseline=(current bounding box.center),scale=1.60]
\draw (1,0) to [out=90, in=-75](0.95,0.5);
\draw (0.95,0.5) arc (5:355:2mm);
\draw[->] (0.95,0.46) to [out=75, in=270] (1,1);
\fill (0.8,0.67) circle (1pt);
\node at (0.75,0.8) {\small $2n$-$1$};
\end{tikzpicture}\hspace{6pt}
\\&=
\hspace{6pt}
\begin{tikzpicture}[baseline=(current bounding box.center),scale=0.75]
\draw[->] (2.5,1) arc (-180:180:5mm);
\draw[->] (3.3,0) to (3.3,2);
\fill (3,1.5) circle (2pt);
\fill (3.5,1) circle (2pt);
\node at (2.7,1.8) {\small $2n$-$4$};
\node at (3.7,1) {\small $4$};
\end{tikzpicture}\hspace{6pt}
+
\hspace{6pt}2
\begin{tikzpicture}[baseline=(current bounding box.center),scale=1.60]
\draw (1,0) to [out=90, in=-75](0.95,0.5);
\draw (0.95,0.5) arc (5:355:2mm);
\draw[->] (0.95,0.46) to [out=75, in=270] (1,1);
\fill (0.8,0.67) circle (1pt);
\node at (0.75,0.8) {\small $2n$-$3$};
\end{tikzpicture}\hspace{6pt}
+
\hspace{6pt}2
\begin{tikzpicture}[baseline=(current bounding box.center),scale=1.60]
\draw (1,0) to [out=90, in=-75](0.95,0.5);
\draw (0.95,0.5) arc (5:355:2mm);
\draw[->] (0.95,0.46) to [out=75, in=270] (1,1);
\fill (0.8,0.67) circle (1pt);
\node at (0.75,0.8) {\small $2n$-$1$};
\end{tikzpicture}
\end{align*}
Continuing to slide dots in the first term in this way, we obtain:

\begin{align*}
\begin{tikzpicture}[baseline=(current bounding box.center),scale=0.75]
\draw[->] (2.5,1) arc (-180:180:5mm);
\draw[->] (4,0) to (4,2);
\fill (3,1.5) circle (2pt);
\node at (3,1.8) {\small $2n$};
\end{tikzpicture}\hspace{6pt}
&=
\hspace{6pt}
\begin{tikzpicture}[baseline=(current bounding box.center),scale=0.75]
\draw[->] (2.5,1) arc (-180:180:5mm);
\draw[->] (3.3,0) to (3.3,2);
\fill (3.5,1) circle (2pt);
\node at (3.8,1) {\small $2n$};
\end{tikzpicture}\hspace{6pt}
+
\hspace{6pt}2\sum_{k=1}^n\begin{tikzpicture}[baseline=(current bounding box.center),scale=1.75]
\draw (1,0) to [out=90, in=-75](0.95,0.5);
\draw (0.95,0.5) arc (5:355:2mm);
\draw[->] (0.95,0.46) to [out=75, in=270] (1,1);
\fill (0.8,0.67) circle (1pt);
\node at (0.8,0.8) {\small $2k$-$1$};
\end{tikzpicture}.
\end{align*}
\end{proof}

\begin{lemma} \label{curl split}
We have that $$\ds \left[\; \begin{tikzpicture}[baseline=(current bounding box.center),scale=1.75]
\draw (1,0) to [out=90, in=-75](0.95,0.5);
\draw (0.95,0.5) arc (5:355:2mm);
\draw[->] (0.95,0.46) to [out=75, in=270] (1,1);
\fill (0.8,0.67) circle (1pt);
\node at (0.75,0.8) {2n+1};
\end{tikzpicture}\;\right]
=
\sum_{a+b=n}
\left[\; \begin{tikzpicture}[baseline=(current bounding box.center),scale=0.75]
\draw[->] (2.5,1) arc (-180:180:5mm);
\draw[->] (4,0) to (4,2);
\fill (3,1.5) circle (2pt);
\fill (4,1) circle (2pt);
\node at (3,1.8) {2a};
\node at (4.3,1) {2b};
\end{tikzpicture}\;\right]$$
in $\TrHEv$ for any non-negative integer $n$.
\end{lemma}

\begin{proof}
This is an easy computation using induction on $n$. The base case is 
\begin{equation*}
\begin{tikzpicture}[baseline=(current bounding box.center),scale=1.75]
\draw (1,0) to [out=90, in=-75](0.95,0.5);
\draw (0.95,0.5) arc (5:355:2mm);
\draw[->] (0.95,0.46) to [out=75, in=270] (1,1);
\fill (0.8,0.67) circle (1pt);
\end{tikzpicture}\hspace{6pt}
=
\hspace{6pt}
\begin{tikzpicture}[baseline=(current bounding box.center),scale=1.75]
\draw (1,0) to [out=90, in=-75](0.95,0.5);
\draw (0.95,0.5) arc (5:355:2mm);
\draw[->] (0.95,0.46) to [out=75, in=270] (1,1);
\fill (1,0.27) circle (1pt);
\end{tikzpicture}\hspace{6pt}
+
\hspace{6pt}
\begin{tikzpicture}[baseline=(current bounding box.center),scale=0.75]
\draw[->] (2.5,1) arc (-180:180:5mm);
\draw[->] (4,0) to (4,2);
\end{tikzpicture}\hspace{6pt}
-
\hspace{6pt}
\begin{tikzpicture}[baseline=(current bounding box.center),scale=0.75]
\draw[->] (2.5,1) arc (-180:180:5mm);
\draw[->] (4,0) to (4,2);
\draw (3.5,1) circle (2pt);
\draw (4,0.8) circle (2pt);
\end{tikzpicture}\hspace{6pt}
=
\hspace{6pt}
\begin{tikzpicture}[baseline=(current bounding box.center),scale=0.75]
\draw[->] (2.5,1) arc (-180:180:5mm);
\draw[->] (4,0) to (4,2);
\end{tikzpicture}
\end{equation*}
where the first term after the first equality contains a left twist curl, and the last term is zero since a bubble with a hollow dot is zero.

For the induction step, suppose the statement holds for $n\geq 1$. Then
\begin{align*}
\begin{tikzpicture}[baseline=(current bounding box.center),scale=1.75]
\draw (1,0) to [out=90, in=-75](0.95,0.5);
\draw (0.95,0.5) arc (5:355:2mm);
\draw[->] (0.95,0.46) to [out=75, in=270] (1,1);
\fill (0.8,0.67) circle (1pt);
\node at (0.7,0.78) {\small $2n$+$3$};
\end{tikzpicture}\hspace{6pt}
&=
\hspace{6pt}
\begin{tikzpicture}[baseline=(current bounding box.center),scale=1.75]
\draw (1,0) to [out=90, in=-75](0.95,0.5);
\draw (0.95,0.5) arc (5:355:2mm);
\draw[->] (0.95,0.46) to [out=75, in=270] (1,1);
\fill (0.8,0.67) circle (1pt);
\node at (0.7,0.78) {\small $2n$+$2$};
\fill (1,0.27) circle (1pt);
\node at (1.1,0.27) {};
\end{tikzpicture}\hspace{6pt}
+
\hspace{6pt}
\begin{tikzpicture}[baseline=(current bounding box.center),scale=0.75]
\draw[->] (2.5,1) arc (-180:180:5mm);
\draw[->] (4,0) to (4,2);
\fill (3,1.5) circle (2pt);
%\fill (4,1) circle (2pt);
\node at (3,1.8) {\small $2n$+$2$};
\node at (4.3,1) {};
\end{tikzpicture}\hspace{6pt}
-
\hspace{6pt}
\begin{tikzpicture}[baseline=(current bounding box.center),scale=0.75]
\draw[->] (2.5,1) arc (-180:180:5mm);
\draw[->] (4,0) to (4,2);
\fill (3,1.5) circle (2pt);
%\fill (4,1) circle (2pt);
\draw (3.5,1) circle (2pt);
\draw (4,0.8) circle (2pt);
\node at (3,1.8) {\small $2n$+$2$};
\node at (4.3,1) {};
\end{tikzpicture}\hspace{6pt}
=
\hspace{6pt}
\begin{tikzpicture}[baseline=(current bounding box.center),scale=1.75]
\draw (1,0) to [out=90, in=-75](0.95,0.5);
\draw (0.95,0.5) arc (5:355:2mm);
\draw[->] (0.95,0.46) to [out=75, in=270] (1,1);
\fill (0.8,0.67) circle (1pt);
\node at (0.7,0.78) {\small $2n$+$2$};
\fill (1,0.27) circle (1pt);
\node at (1.1,0.27) {};
\end{tikzpicture}\hspace{6pt}
+
\hspace{6pt}
\begin{tikzpicture}[baseline=(current bounding box.center),scale=0.75]
\draw[->] (2.5,1) arc (-180:180:5mm);
\draw[->] (4,0) to (4,2);
\fill (3,1.5) circle (2pt);
%\fill (4,1) circle (2pt);
\node at (3,1.8) {\small $2n$+$2$};
\node at (4.3,1) {};
\end{tikzpicture}\hspace{6pt}
\\&=
\hspace{6pt}
\begin{tikzpicture}[baseline=(current bounding box.center),scale=1.75]
\draw (1,0) to [out=90, in=-75](0.95,0.5);
\draw (0.95,0.5) arc (5:355:2mm);
\draw[->] (0.95,0.46) to [out=75, in=270] (1,1);
\fill (0.8,0.67) circle (1pt);
\node at (0.7,0.78) {\small $2n$+$1$};
\fill (1,0.27) circle (1pt);
\node at (1.1,0.27) {\small $2$};
\end{tikzpicture}\hspace{6pt}
+
\hspace{6pt}
\begin{tikzpicture}[baseline=(current bounding box.center),scale=0.75]
\draw[->] (2.5,1) arc (-180:180:5mm);
\draw[->] (4,0) to (4,2);
\fill (3,1.5) circle (2pt);
\fill (4,1) circle (2pt);
\node at (3,1.8) {\small $2n$+$1$};
\node at (4.3,1) {};
\end{tikzpicture}\hspace{6pt}
-
\hspace{6pt}
\begin{tikzpicture}[baseline=(current bounding box.center),scale=0.75]
\draw[->] (2.5,1) arc (-180:180:5mm);
\draw[->] (4,0) to (4,2);
\fill (3,1.5) circle (2pt);
\fill (4,1) circle (2pt);
\draw (3.5,1) circle (2pt);
\draw (4,0.8) circle (2pt);
\node at (3,1.8) {\small $2n$+$1$};
\node at (4.3,1) {};
\end{tikzpicture}\hspace{6pt}
+
\hspace{6pt}
\begin{tikzpicture}[baseline=(current bounding box.center),scale=0.75]
\draw[->] (2.5,1) arc (-180:180:5mm);
\draw[->] (4,0) to (4,2);
\fill (3,1.5) circle (2pt);
%\fill (4,1) circle (2pt);
\node at (3,1.8) {\small $2n$+$2$};
\node at (4.3,1) {};
\end{tikzpicture}\hspace{6pt}
\\&=
\hspace{6pt}
\begin{tikzpicture}[baseline=(current bounding box.center),scale=1.75]
\draw (1,0) to [out=90, in=-75](0.95,0.5);
\draw (0.95,0.5) arc (5:355:2mm);
\draw[->] (0.95,0.46) to [out=75, in=270] (1,1);
\fill (0.8,0.67) circle (1pt);
\node at (0.7,0.78) {\small $2n$+$1$};
\fill (1,0.27) circle (1pt);
\node at (1.1,0.27) {\small $2$};
\end{tikzpicture}\hspace{6pt}
+
\hspace{6pt}
\begin{tikzpicture}[baseline=(current bounding box.center),scale=0.75]
\draw[->] (2.5,1) arc (-180:180:5mm);
\draw[->] (4,0) to (4,2);
\fill (3,1.5) circle (2pt);
%\fill (4,1) circle (2pt);
\node at (3,1.8) {\small $2n$+$2$};
\node at (4.3,1) {};
\end{tikzpicture},
\end{align*}
where on the second line, we know that counter-clockwise bubbles with odd number of hollow dots are zero by Proposition \ref{odd dots}, and the terms with hollow dots are zero by Proposition \ref{hollow dots}.

Now we can apply our induction hypothesis to the upper part of $\begin{tikzpicture}[baseline=(current bounding box.center),scale=1.75]
\draw (1,0) to [out=90, in=-75](0.95,0.5);
\draw (0.95,0.5) arc (5:355:2mm);
\draw[->] (0.95,0.46) to [out=75, in=270] (1,1);
\fill (0.8,0.67) circle (1pt);
\node at (0.7,0.78) {\small $2n$+$1$};
\fill (1,0.27) circle (1pt);
\node at (1.1,0.27) {\small $2$};
\end{tikzpicture}$ to get that

\begin{align*}
\begin{tikzpicture}[baseline=(current bounding box.center),scale=1.75]
\draw (1,0) to [out=90, in=-75](0.95,0.5);
\draw (0.95,0.5) arc (5:355:2mm);
\draw[->] (0.95,0.46) to [out=75, in=270] (1,1);
\fill (0.8,0.67) circle (1pt);
\node at (0.7,0.78) {\small $2n$+$3$};
\end{tikzpicture}\hspace{6pt}
=
\hspace{6pt}\sum_{a+b=n}
\begin{tikzpicture}[baseline=(current bounding box.center),scale=0.75]
\draw[->] (2.5,1) arc (-180:180:5mm);
\draw[->] (4,0) to (4,2);
\fill (3,1.5) circle (2pt);
\fill (4,1) circle (2pt);
\node at (3,1.8) {2a};
\node at (4.3,1) {2b};
\fill (4,0.5) circle (2pt);
\node at (4.2,0.5) {\small $2$};
\end{tikzpicture}\hspace{6pt}
+
\hspace{6pt}
\begin{tikzpicture}[baseline=(current bounding box.center),scale=0.75]
\draw[->] (2.5,1) arc (-180:180:5mm);
\draw[->] (4,0) to (4,2);
\fill (3,1.5) circle (2pt);
%\fill (4,1) circle (2pt);
\node at (3,1.8) {\small $2n$+$2$};
\node at (4.3,1) {};
\end{tikzpicture}\hspace{6pt}
=
\hspace{6pt}\sum_{a+b=n+1}
\begin{tikzpicture}[baseline=(current bounding box.center),scale=0.75]
\draw[->] (2.5,1) arc (-180:180:5mm);
\draw[->] (4,0) to (4,2);
\fill (3,1.5) circle (2pt);
\fill (4,1) circle (2pt);
\node at (3,1.8) {2a};
\node at (4.3,1) {2b};
\end{tikzpicture},
\end{align*}
 as desired.
\end{proof}

Obtaining an explicit formula for sliding counter-clockwise bubbles is difficult since we express their commutators in terms of left twist curls with some dots on the curl, whose resolution terms still leave us with counter-clockwise bubbles on the left side of $\h{1}{x_i^a}$. However, the situation is better with clockwise oriented bubbles.

\subsection{Clockwise bubble slide moves}

We can compute an explicit formula for clockwise bubble slides.

\begin{lemma}{\label{clockwise bubbles 1}}
We have $$\ds[d_{2n},\h{1}{}]=2\hspace{2mm}
\left[\; \begin{tikzpicture}[baseline=(current bounding box.center),scale=1.5]
\draw[->] (0,0) to (0,1);
\fill (0,0.75) circle (1pt);
\node at (0.2,0.75) {2n};
\end{tikzpicture} \;\right]
+2\sum_{a+b=2n-1}
\left[\; \begin{tikzpicture}[baseline=(current bounding box.center),scale=1.5]
\draw (1,0) to [out=90, in=85](1.05,0.5);
\draw (1.05,0.5) arc (-175:175:2mm);
\draw[->] (1.05,0.46) to [out=95, in=270] (1,1);
\fill (1.02,0.75) circle (1pt);
\node at (0.8,0.75) {a};
\fill (1.45,0.5) circle (1pt);
\node at (1.55,0.5) {b};
\end{tikzpicture} \;\right]$$
in $\TrHEv$ for all $n\geq 0$.
\end{lemma}

\begin{proof}
This is a direct computation, given below:
\begin{align*}
\begin{tikzpicture}[baseline=(current bounding box.center),scale=0.75]
\draw[<-] (2.5,1) arc (-180:180:5mm);
\draw[->] (4,0) to (4,2);
\fill (3,1.5) circle (2pt);
\node at (3,1.8) {\small $2n$};
\end{tikzpicture}\hspace{6pt}
&=
\hspace{6pt}
\begin{tikzpicture}[baseline=(current bounding box.center),scale=0.75]
\draw[<-] (2.5,1) arc (-180:180:5mm);
\draw[->] (3.3,0) to (3.3,2);
\fill (3,1.5) circle (2pt);
\node at (2.5,1.8) {\small $2n$};
\end{tikzpicture}\hspace{6pt}
+2\hspace{2mm}
\begin{tikzpicture}[baseline=(current bounding box.center),scale=0.75]
\draw[->] (0,0) to (0,2);
\fill (0,1.6) circle (2pt);
\node at (0.34,1.6) {\small $2n$};
\end{tikzpicture}\hspace{6pt}
\\&=
\hspace{6pt}
\begin{tikzpicture}[baseline=(current bounding box.center),scale=0.75]
\draw[<-] (2.5,1) arc (-180:180:5mm);
\draw[->] (3.3,0) to (3.3,2);
\fill (3,1.5) circle (2pt);
\node at (2.7,1.8) {\small $2n$-$1$};
\fill (3.5, 1) circle (2pt);
\node at (3.8,1) {};
\end{tikzpicture}\hspace{6pt}
+
\hspace{6pt}
\begin{tikzpicture}[baseline=(current bounding box.center),scale=1.5]
\draw (1,0) to [out=90, in=85](1.05,0.5);
\draw (1.05,0.5) arc (-175:175:2mm);
\draw[->] (1.05,0.46) to [out=95, in=270] (1,1);
\fill (1.01,0.8) circle (1pt);
\node at (0.7,0.8) {\small $2n$-$1$};
\end{tikzpicture}\hspace{6pt}
+
\hspace{6pt}
\begin{tikzpicture}[baseline=(current bounding box.center),scale=1.5]
\draw (1,0) to [out=90, in=85](1.05,0.5);
\draw (1.05,0.5) arc (-175:175:2mm);
\draw[->] (1.05,0.46) to [out=95, in=270] (1,1);
\fill (1.01,0.8) circle (1pt);
\node at (0.7,0.8) {\small $2n$-$1$};
\draw (1,0.9) circle (1pt);
\draw (1.12,0.67) circle (1pt);
\end{tikzpicture}\hspace{6pt}
+2\hspace{2mm}
\begin{tikzpicture}[baseline=(current bounding box.center),scale=0.75]
\draw[->] (0,0) to (0,2);
\fill (0,1.6) circle (2pt);
\node at (0.34,1.6) {\small $2n$};
\end{tikzpicture}\hspace{6pt}
\\&=
\hspace{6pt}
\begin{tikzpicture}[baseline=(current bounding box.center),scale=0.75]
\draw[<-] (2.5,1) arc (-180:180:5mm);
\draw[->] (3.3,0) to (3.3,2);
\fill (3,1.5) circle (2pt);
\node at (2.7,1.8) {\small $2n$-$1$};
\fill (3.5, 1) circle (2pt);
\node at (3.8,1) {};
\end{tikzpicture}\hspace{6pt}
+2
\hspace{6pt}
\begin{tikzpicture}[baseline=(current bounding box.center),scale=1.5]
\draw (1,0) to [out=90, in=85](1.05,0.5);
\draw (1.05,0.5) arc (-175:175:2mm);
\draw[->] (1.05,0.46) to [out=95, in=270] (1,1);
\fill (1.01,0.8) circle (1pt);
\node at (0.7,0.8) {\small $2n$-$1$};
\end{tikzpicture}\hspace{6pt}
+2\hspace{2mm}
\begin{tikzpicture}[baseline=(current bounding box.center),scale=0.75]
\draw[->] (0,0) to (0,2);
\fill (0,1.6) circle (2pt);
\node at (0.34,1.6) {\small $2n$};
\end{tikzpicture}\hspace{6pt}
\\&=
\hspace{6pt}
\begin{tikzpicture}[baseline=(current bounding box.center),scale=0.75]
\draw[<-] (2.5,1) arc (-180:180:5mm);
\draw[->] (3.3,0) to (3.3,2);
\fill (3,1.5) circle (2pt);
\node at (2.7,1.8) {\small $2n$-$2$};
\fill (3.5, 1) circle (2pt);
\node at (3.8,1) {\small $2$};
\end{tikzpicture}\hspace{6pt}
+
\hspace{6pt}
\begin{tikzpicture}[baseline=(current bounding box.center),scale=1.5]
\draw (1,0) to [out=90, in=85](1.05,0.5);
\draw (1.05,0.5) arc (-175:175:2mm);
\draw[->] (1.05,0.46) to [out=95, in=270] (1,1);
\fill (1.01,0.8) circle (1pt);
\node at (0.7,0.8) {\small $2n$-$2$};
\fill (1.24,0.72) circle (1pt);
\end{tikzpicture}\hspace{6pt}
+
\hspace{6pt}
\begin{tikzpicture}[baseline=(current bounding box.center),scale=1.5]
\draw (1,0) to [out=90, in=85](1.05,0.5);
\draw (1.05,0.5) arc (-175:175:2mm);
\draw[->] (1.05,0.46) to [out=95, in=270] (1,1);
\fill (1.01,0.8) circle (1pt);
\node at (0.7,0.8) {\small $2n$-$2$};
\fill (1.24,0.72) circle (1pt);
\draw (1,0.9) circle (1pt);
\draw (1.12,0.67) circle (1pt);
\end{tikzpicture}\hspace{6pt}
+
\hspace{6pt}2
\begin{tikzpicture}[baseline=(current bounding box.center),scale=1.5]
\draw (1,0) to [out=90, in=85](1.05,0.5);
\draw (1.05,0.5) arc (-175:175:2mm);
\draw[->] (1.05,0.46) to [out=95, in=270] (1,1);
\fill (1.01,0.8) circle (1pt);
\node at (0.7,0.8) {\small $2n$-$1$};
\end{tikzpicture}\hspace{6pt}
+2\hspace{2mm}
\begin{tikzpicture}[baseline=(current bounding box.center),scale=0.75]
\draw[->] (0,0) to (0,2);
\fill (0,1.6) circle (2pt);
\node at (0.34,1.6) {\small $2n$};
\end{tikzpicture}\hspace{6pt}
\\&=
\hspace{6pt}
\begin{tikzpicture}[baseline=(current bounding box.center),scale=0.75]
\draw[<-] (2.5,1) arc (-180:180:5mm);
\draw[->] (3.3,0) to (3.3,2);
\fill (3,1.5) circle (2pt);
\node at (2.7,1.8) {\small $2n$-$2$};
\fill (3.5, 1) circle (2pt);
\node at (3.8,1) {\small $2$};
\end{tikzpicture}\hspace{6pt}
+
2\hspace{6pt}
\begin{tikzpicture}[baseline=(current bounding box.center),scale=1.5]
\draw (1,0) to [out=90, in=85](1.05,0.5);
\draw (1.05,0.5) arc (-175:175:2mm);
\draw[->] (1.05,0.46) to [out=95, in=270] (1,1);
\fill (1.01,0.8) circle (1pt);
\node at (0.7,0.8) {\small $2n$-$2$};
\fill (1.24,0.72) circle (1pt);
\end{tikzpicture}\hspace{6pt}
+
\hspace{6pt}2
\begin{tikzpicture}[baseline=(current bounding box.center),scale=1.5]
\draw (1,0) to [out=90, in=85](1.05,0.5);
\draw (1.05,0.5) arc (-175:175:2mm);
\draw[->] (1.05,0.46) to [out=95, in=270] (1,1);
\fill (1.01,0.8) circle (1pt);
\node at (0.7,0.8) {\small $2n$-$1$};
\end{tikzpicture}\hspace{6pt}
+
2\hspace{2mm}
\begin{tikzpicture}[baseline=(current bounding box.center),scale=0.75]
\draw[->] (0,0) to (0,2);
\fill (0,1.6) circle (2pt);
\node at (0.34,1.6) {\small $2n$};
\end{tikzpicture}
\end{align*}
Continuing to slide dots in the first term in this way, we obtain:
\begin{align*}&
\begin{tikzpicture}[baseline=(current bounding box.center),scale=0.75]
\draw[<-] (2.5,1) arc (-180:180:5mm);
\draw[->] (4,0) to (4,2);
\fill (3,1.5) circle (2pt);
\node at (3,1.8) {\small $2n$};
\end{tikzpicture}\hspace{6pt}
=
\hspace{6pt}
\begin{tikzpicture}[baseline=(current bounding box.center),scale=0.75]
\draw[<-] (2.5,1) arc (-180:180:5mm);
\draw[->] (2,0) to (2,2);
\fill (3,1.5) circle (2pt);
\node at (3,1.8) {\small $2n$};
\end{tikzpicture}\hspace{6pt}
+2\hspace{6pt}
\begin{tikzpicture}[baseline=(current bounding box.center),scale=0.75]
\draw[->] (0,0) to (0,2);
\fill (0,1.6) circle (2pt);
\node at (0.34,1.6) {\small $2n$};
\end{tikzpicture}\hspace{6pt}
+
\hspace{6pt}2\sum_{a+b=2n-1}
\begin{tikzpicture}[baseline=(current bounding box.center),scale=1.5]
\draw (1,0) to [out=90, in=85](1.05,0.5);
\draw (1.05,0.5) arc (-175:175:2mm);
\draw[->] (1.05,0.46) to [out=95, in=270] (1,1);
\fill (1.02,0.75) circle (1pt);
\node at (0.8,0.75) {\small $a$};
\fill (1.45,0.5) circle (1pt);
\node at (1.55,0.5) {\small $b$};
\end{tikzpicture}.\end{align*}
\end{proof}

In particular, we can refine this statement to obtain the following recursive formula for computing $[d_{2n},\h{1}{}]$.
\begin{lemma}{\label{clockwise bubbles recursive}}
We have$$[d_{2n},\h{1}{}]=[d_{2n-2},\h{1}{}]\circ  x_1^2 + 4\hspace{2mm}
\left[\; \begin{tikzpicture}[baseline=(current bounding box.center),scale=1.5]
\draw (1,0) to [out=90, in=85](1.05,0.5);
\draw (1.05,0.5) arc (-175:175:2mm);
\draw[->] (1.05,0.46) to [out=95, in=270] (1,1);
\fill (1.01,0.85) circle (1pt);
\node at (1.1,0.85) {\small $2$};
\fill (1.45,0.5) circle (1pt);
\node at (1.7,0.5) {\small $2n$-$3$};
\end{tikzpicture} \;\right]
-2\hspace{2mm}
\left[\; \begin{tikzpicture}[baseline=(current bounding box.center),scale=0.75]
\draw[<-] (0.5,1) arc (-180:180:5mm);
\draw[->] (0,0) to (0,2);
\fill (1.5,1) circle (2pt);
\node at (1.8,1) {\small $2n-2$};
\end{tikzpicture} \;\right]$$ in $\TrHEv$ for all $n\geq 0$.
\end{lemma}

\begin{proof}
This lemma follows from the observation that 
\begin{equation*}
\begin{tikzpicture}[baseline=(current bounding box.center),scale=1.5]
\draw (1,0) to [out=90, in=85](1.05,0.5);
\draw (1.05,0.5) arc (-175:175:2mm);
\draw[->] (1.05,0.46) to [out=95, in=270] (1,1);
\fill (1.02,0.75) circle (1pt);
\node at (0.8,0.75) {\small $a$};
\fill (1.45,0.5) circle (1pt);
\node at (1.65,0.5) {\small $2k$};
\end{tikzpicture}\hspace{6pt}
=
\hspace{6pt}
\begin{tikzpicture}[baseline=(current bounding box.center),scale=1.5]
\draw (1,0) to [out=90, in=85](1.05,0.5);
\draw (1.05,0.5) arc (-175:175:2mm);
\draw[->] (1.05,0.46) to [out=95, in=270] (1,1);
\fill (1.02,0.75) circle (1pt);
\node at (0.8,0.75) {\small $a$+$1$};
\fill (1.45,0.5) circle (1pt);
\node at (1.70,0.5) {\small $2k$-$1$};
\end{tikzpicture}\hspace{6pt}
-
\hspace{6pt}
\begin{tikzpicture}[baseline=(current bounding box.center),scale=0.75]
\draw[<-] (2.5,1) arc (-180:180:5mm);
\draw[->] (2,0) to (2,2);
\fill (3,1.5) circle (2pt);
\node at (3,1.8) {\small $2k$-$1$};
\fill (2,1.4) circle (2pt);
\node at (1.8,1.4) {\small $a$};
\end{tikzpicture}\hspace{6pt}
+
\hspace{6pt}
\begin{tikzpicture}[baseline=(current bounding box.center),scale=0.75]
\draw[<-] (2.5,1) arc (-180:180:5mm);
\draw[->] (2,0) to (2,2);
\fill (3,1.5) circle (2pt);
\node at (3,1.8) {\small $2k$-$1$};
\fill (2,1.4) circle (2pt);
\node at (1.8,1.4) {\small $a$};
\draw (2,1.2) circle (2pt);
\draw (2.5,1.1) circle (2pt);
\end{tikzpicture}\hspace{6pt}
=
\hspace{6pt}
\begin{tikzpicture}[baseline=(current bounding box.center),scale=1.5]
\draw (1,0) to [out=90, in=85](1.05,0.5);
\draw (1.05,0.5) arc (-175:175:2mm);
\draw[->] (1.05,0.46) to [out=95, in=270] (1,1);
\fill (1.02,0.75) circle (1pt);
\node at (0.8,0.75) {\small $a$+$1$};
\fill (1.45,0.5) circle (1pt);
\node at (1.70,0.5) {\small $2k$-$1$};
\end{tikzpicture},
\end{equation*}
where the second term after the first equality is zero by Lemma \ref{odd dots}, and the third term is zero by Lemma \ref{hollow dots}. Applying this result to the summands in the statement of Lemma \ref{clockwise bubbles 1} yields the result.
\end{proof}

Finally, we obtain an explicit formula for computing $[d_{2n},\h{1}{}]$.
\begin{proposition}\label{clockwise bubbles explicit}
We have $$\ds[d_{2n},\h{1}{}]=(2+4n)\hspace{2mm}
\left[\; \begin{tikzpicture}[baseline=(current bounding box.center),scale=1.5]
\draw[->] (0,0) to (0,1);
\fill (0,0.75) circle (1pt);
\node at (0.2,0.75) {2n};
\end{tikzpicture} \;\right]
-\sum_{a+b=n-1}(2+4a)
\left[\; \begin{tikzpicture}[baseline=(current bounding box.center),scale=0.75]
\draw[<-] (0.5,1) arc (-180:180:5mm);
\draw[->] (0,0) to (0,2);
\fill (1.5,1) circle (2pt);
\fill (0,1.5) circle (2pt);
\node at (0.3,1.5) {2a};
\node at (1.8,1) {2b};
\end{tikzpicture} \;\right]$$ in $\TrHEv$ for all $n\geq 0$.
\end{proposition}

\begin{proof}
We claim that $$\begin{tikzpicture}[baseline=(current bounding box.center),scale=1.5]
\draw (1,0) to [out=90, in=85](1.05,0.5);
\draw (1.05,0.5) arc (-175:175:2mm);
\draw[->] (1.05,0.46) to [out=95, in=270] (1,1);
\fill (1.01,0.85) circle (1pt);
\node at (1.1,0.85) {\small $2$};
\fill (1.45,0.5) circle (1pt);
\node at (1.7,0.5) {\small $2n$-$3$};
\end{tikzpicture}\hspace{6pt}
=
\hspace{6pt}
\begin{tikzpicture}[baseline=(current bounding box.center),scale=1.5]
\draw[->] (0,0) to (0,1);
\fill (0,0.75) circle (1pt);
\node at (0.2,0.75) {$2n$};
\end{tikzpicture}\hspace{6pt}
-
\hspace{6pt}
\ds \sum_{\substack{a+b=n-1\\a\neq0}}\begin{tikzpicture}[baseline=(current bounding box.center),scale=0.75]
\draw[<-] (2.5,1) arc (-180:180:5mm);
\draw[->] (2,0) to (2,2);
\fill (3,1.5) circle (2pt);
\node at (3,1.8) {\small $2b$};
\fill (2,1.4) circle (2pt);
\node at (1.7,1.4) {\small $2a$};
\end{tikzpicture}\hspace{6pt}
$$ for $n\geq2$.  We proceed via induction on $n$. The base case $n=2$ is a direct computation. Now suppose the formula holds for some $n\geq 2$. Then 
\begin{align*}
    \begin{tikzpicture}[baseline=(current bounding box.center),scale=1.5]
\draw (1,0) to [out=90, in=85](1.05,0.5);
\draw (1.05,0.5) arc (-175:175:2mm);
\draw[->] (1.05,0.46) to [out=95, in=270] (1,1);
\fill (1.01,0.85) circle (1pt);
\node at (1.1,0.85) {\small $2$};
\fill (1.45,0.5) circle (1pt);
\node at (1.8,0.5) {\small $2n-3$};
\end{tikzpicture}\hspace{6pt}
=
\hspace{6pt}
\begin{tikzpicture}[baseline=(current bounding box.center),scale=1.5]
\draw (1,0) to [out=90, in=85](1.05,0.5);
\draw (1.05,0.5) arc (-175:175:2mm);
\draw[->] (1.05,0.46) to [out=95, in=270] (1,1);
\fill (1.01,0.85) circle (1pt);
\node at (1.1,0.85) {\small $3$};
\fill (1.45,0.5) circle (1pt);
\node at (1.8,0.5) {\small $2n-4$};
\end{tikzpicture}\hspace{6pt}
-
\hspace{6pt}
\begin{tikzpicture}[baseline=(current bounding box.center),scale=0.75]
\draw[<-] (2.5,1) arc (-180:180:5mm);
\draw[->] (2,0) to (2,2);
\fill (3,1.5) circle (2pt);
\node at (3,1.8) {\small $2n$-$4$};
\fill (2,1.4) circle (2pt);
\node at (1.7,1.4) {\small $2$};
\end{tikzpicture}\hspace{6pt}
&=
\hspace{6pt}
\begin{tikzpicture}[baseline=(current bounding box.center),scale=1.5]
\draw (1,0) to [out=90, in=85](1.05,0.5);
\draw (1.05,0.5) arc (-175:175:2mm);
\draw[->] (1.05,0.46) to [out=95, in=270] (1,1);
\fill (1.01,0.85) circle (1pt);
\node at (1.1,0.85) {\small $4$};
\fill (1.45,0.5) circle (1pt);
\node at (1.8,0.5) {\small $2n-5$};
\end{tikzpicture}\hspace{6pt}
-
\hspace{6pt}
\begin{tikzpicture}[baseline=(current bounding box.center),scale=0.75]
\draw[<-] (2.5,1) arc (-180:180:5mm);
\draw[->] (2,0) to (2,2);
\fill (3,1.5) circle (2pt);
\node at (3,1.8) {\small $2n$-$4$};
\fill (2,1.4) circle (2pt);
\node at (1.7,1.4) {\small $2$};
\end{tikzpicture}\hspace{6pt}
\\&=
\hspace{6pt}
\begin{tikzpicture}[baseline=(current bounding box.center),scale=1.5]
\draw (1,0) to [out=90, in=85](1.05,0.5);
\draw (1.05,0.5) arc (-175:175:2mm);
\draw[->] (1.05,0.46) to [out=95, in=270] (1,1);
\fill (1.02,0.75) circle (1pt);
\node at (1.12,0.75) {\small $2$};
\fill (1,0.92) circle (1pt);
\node at (1.1,0.92) {\small $2$};
\fill (1.45,0.5) circle (1pt);
\node at (1.8,0.5) {\small $2n-5$};
\end{tikzpicture}\hspace{6pt}
-
\hspace{6pt}
\begin{tikzpicture}[baseline=(current bounding box.center),scale=0.75]
\draw[<-] (2.5,1) arc (-180:180:5mm);
\draw[->] (2,0) to (2,2);
\fill (3,1.5) circle (2pt);
\node at (3,1.8) {\small $2n$-$4$};
\fill (2,1.4) circle (2pt);
\node at (1.7,1.4) {\small $2$};
\end{tikzpicture}.
\end{align*}
 Now we can apply the induction hypothesis to the lower part of the first term in the last expression. This gives us:

\begin{align*}
    \begin{tikzpicture}[baseline=(current bounding box.center),scale=1.5]
\draw (1,0) to [out=90, in=85](1.05,0.5);
\draw (1.05,0.5) arc (-175:175:2mm);
\draw[->] (1.05,0.46) to [out=95, in=270] (1,1);
\fill (1.01,0.85) circle (1pt);
\node at (1.1,0.85) {\small $2$};
\fill (1.45,0.5) circle (1pt);
\node at (1.8,0.5) {\small $2n$-$3$};
\end{tikzpicture}\hspace{6pt}
&=
\hspace{6pt}
\Bigg(\;\begin{tikzpicture}[baseline=(current bounding box.center),scale=1.5]
\draw[->] (0,0) to (0,1);
\fill (0,0.65) circle (1pt);
\node at (0.3,0.65) {\small $2n$-$2$};
\fill (0,0.85) circle (1pt);
\node at (0.15,0.85) {\small $2$};
\end{tikzpicture}\hspace{6pt}
-
\hspace{6pt}
\ds \sum_{\substack{a+b=n-2\\a\neq0}}\begin{tikzpicture}[baseline=(current bounding box.center),scale=0.75]
\draw[<-] (2.5,1) arc (-180:180:5mm);
\draw[->] (2,0) to (2,2);
\fill (3,1.5) circle (2pt);
\node at (3,1.8) {\small $2b$};
\fill (2,1.4) circle (2pt);
\node at (1.7,1.4) {\small $2a$};
\fill (2,1.7) circle (2pt);
\node at (1.7,1.7) {\small $2$};
\end{tikzpicture}\hspace{6pt}
\Bigg)-
\hspace{6pt}
\begin{tikzpicture}[baseline=(current bounding box.center),scale=0.75]
\draw[<-] (2.5,1) arc (-180:180:5mm);
\draw[->] (2,0) to (2,2);
\fill (3,1.5) circle (2pt);
\node at (3,1.8) {\small $2n$-$4$};
\fill (2,1.4) circle (2pt);
\node at (1.7,1.4) {\small $2$};
\end{tikzpicture}\hspace{6pt}
\\&=
\hspace{6pt}
\Bigg(\;\begin{tikzpicture}[baseline=(current bounding box.center),scale=1.5]
\draw[->] (0,0) to (0,1);
\fill (0,0.65) circle (1pt);
\node at (0.2,0.65) {\small $2n$};
\end{tikzpicture}\hspace{6pt}
-
\hspace{6pt}
\ds \sum_{\substack{a+b=n-2\\a\neq0}}\begin{tikzpicture}[baseline=(current bounding box.center),scale=0.75]
\draw[<-] (2.5,1) arc (-180:180:5mm);
\draw[->] (2,0) to (2,2);
\fill (3,1.5) circle (2pt);
\node at (3,1.8) {\small $2b$};
\fill (2,1.4) circle (2pt);
\node at (1.4,1.4) {\small $2a$+$2$};
\end{tikzpicture}\hspace{6pt}
\Bigg)-
\hspace{6pt}
\begin{tikzpicture}[baseline=(current bounding box.center),scale=0.75]
\draw[<-] (2.5,1) arc (-180:180:5mm);
\draw[->] (2,0) to (2,2);
\fill (3,1.5) circle (2pt);
\node at (3,1.8) {\small $2n$-$4$};
\fill (2,1.4) circle (2pt);
\node at (1.7,1.4) {\small $2$};
\end{tikzpicture}\hspace{6pt}
\\&=
\hspace{6pt}
\begin{tikzpicture}[baseline=(current bounding box.center),scale=1.5]
\draw[->] (0,0) to (0,1);
\fill (0,0.65) circle (1pt);
\node at (0.2,0.65) {\small $2n$};
\end{tikzpicture}\hspace{6pt}
-
\hspace{6pt}
\ds \sum_{\substack{a+b=n-1\\a\neq0}}\begin{tikzpicture}[baseline=(current bounding box.center),scale=0.75]
\draw[<-] (2.5,1) arc (-180:180:5mm);
\draw[->] (2,0) to (2,2);
\fill (3,1.5) circle (2pt);
\node at (3,1.8) {\small $2b$};
\fill (2,1.4) circle (2pt);
\node at (1.4,1.4) {\small $2a$};
\end{tikzpicture}
\end{align*}

Applying this result to the recursive formula in Lemma \ref{clockwise bubbles recursive} proves the statement.

\end{proof}

Commutators of bubbles with downward strands are similar to those of bubbles with upward strands. 

\begin{lemma}\label{clockwise bubbles with down recursive} We have
$$\ds[d_{2n},\h{-1}{}]=-2\hspace{2mm}
\left[\; \begin{tikzpicture}[baseline=(current bounding box.center),scale=1.5]
\draw[<-] (0,0) to (0,1);
\fill (0,0.75) circle (1pt);
\node at (0.2,0.75) {\small $2n$};
\end{tikzpicture} \;\right]
-2\sum_{a+b=2n-1}
\left[\; \begin{tikzpicture}[baseline=(current bounding box.center),scale=1.5]
\draw[<-] (1,0) to [out=90, in=-75](0.95,0.5);
\draw (0.95,0.5) arc (5:355:2mm);
\draw (0.95,0.46) to [out=75, in=270] (1,1);
\fill (1.00,0.75) circle (1pt);
\node at (0.85,0.75) {\small $a$};
\fill (0.55,0.5) circle (1pt);
\node at (0.40,0.5) {\small $b$};
\end{tikzpicture} \;\right]$$
in $\TrHEv$ for all $n\geq 0$.
\end{lemma}
\begin{proof} This follows from a computation similar to those in the proofs of Lemmas \ref{clockwise bubbles 1} and \ref{clockwise bubbles recursive}. \end{proof}

Finally we have an explicit formula for commutators of clockwise oriented bubbles and a single downward strand.
\begin{proposition} \label{clockwise bubbles with down explicit} We have
$$\ds[d_{2n},\h{-1}{}]=-(2+4n)\hspace{2mm}
\left[\; \begin{tikzpicture}[baseline=(current bounding box.center),scale=1.5]
\draw[<-] (0,0) to (0,1);
\fill (0,0.75) circle (1pt);
\node at (0.2,0.75) {\small $2n$};
\end{tikzpicture} \;\right]
+\sum_{a+b=n-1}(2+4a)\hspace{2mm}
\left[\; \begin{tikzpicture}[baseline=(current bounding box.center),scale=0.75]
\draw[<-] (0.5,1) arc (-180:180:5mm);
\draw[<-] (2.3,0) to (2.3,2);
\fill (1.5,1) circle (2pt);
\fill (2.3,1.5) circle (2pt);
\node at (2.6,1.5) {\small $2a$};
\node at (1.8,1) {\small $2b$};
\end{tikzpicture}\;\right]$$ in $\TrHEv$ for $n\geq0$.
\end{proposition}
\begin{proof} This follows from Lemma \ref{clockwise bubbles with down recursive}, using a similar argument as in the proof of Proposition \ref{clockwise bubbles explicit}.\end{proof}

Note that in this formula, we are still left with clockwise bubbles on the left side of a downward strand, but with fewer dots on it. Hence the formula may be applied inductively in order to move all the bubbles to the rightmost part of the diagram.

\section{Diagrammatic lemmas}

This section contains some technical computations to derive relations between diagrams consisting of up and down strands. These relations allow us to find a generating set of $\TrH$ in Section 6.
\subsection{Differential degree zero part of $\TrHEv$}

The differential degree zero part of $\TrHEv$ consists of elements $\{\h{n}{}\}_{n\in \mathbb{Z}}$. First, we have the following basic fact.

\begin{proposition} \cite[Proposition 3.9]{Mike} \label{evenCyclesZero} We have $$\h{2n}{} \cong 0$$ for any $n\in \mathbb{Z}$. \end{proposition} \begin{proof} By Proposition \ref{triangularDecomposition}, the proof in the Hecke-Clifford algebra applies here, as well. \end{proof} 

The elements of $\TrHEv$ satisfy the following relations. 
\begin{lemma}\label{basic hn commutators} The following commutators are zero for all non-negative integers $n,m$:
\begin{enumerate}
    \item $[\h{n}{},\h{m}{}]=0$,
    \item  $[\h{-n}{},\h{-m}{}]=0$,
    \item  $[\h{2n}{},\h{-2n}{}]=0$.
\end{enumerate}
\end{lemma}

\begin{proof}
Parts (1) and (2) follow from the fact that similarly oriented strands can be split apart when they cross twice.
Part (3) follows immediately from Proposition \ref{evenCyclesZero}.
\end{proof}

To obtain a copy of the twisted Heisenberg algebra in the $\TrHEv$, we need to look at commutators between elements with odd numbers of oppositely oriented strands.

\begin{lemma}\label{twistedHeisRels}
We have, for any $n,m \in \mathbb{Z}^{\geq 0}$, $$[\h{2n+1}{},\h{-2m+1}{}]=(\delta_{n,-m})(-2(2n+1)).$$
\end{lemma}
\begin{proof}
First note that  \cite[Lemma 19]{CLLS} and \cite[Lemma 20]{CLLS} holds in our twisted case with a small modification, since all the arguments in their proofs use the fact that the resolution terms contain left twist curls, hence are zero. There are extra resolution terms with hollow dots due to relation \eqref{R3}, but two hollow dots on a diagram containing a left twist curl still gives zero. The only modification comes in the case $m=n$ where we get two copies of counter-clockwise bubbles instead of one, since a two hollow dots on a counter-clockwise bubble end up canceling each other without changing the sign of the diagram. We immediately get that when $m\neq n$, our commutator is zero since we have no solid dots. Therefore we have 

\begin{align*}
    h_{2n+1}h_{-2m+1}&=\left[ \;\begin{tikzpicture}[baseline=(current bounding box.center),scale=0.8]
       \draw[->] (3.2,0) .. controls (3.2,1.25) and (0,.25) .. (0,2)
     node[pos=0.85, shape=coordinate](X){};
  \draw[->] (0,0) .. controls (0,1) and (.8,.8) .. (.8,2);
  \draw[->] (.8,0) .. controls (.8,1) and (1.6,.8) .. (1.6,2);
  \draw[->] (2.4,0) .. controls (2.4,1) and (3.2,.8) .. (3.2,2);
  \node at (1.6,.35) {$\dots$};
  \node at (2.4,1.65) {$\dots$};
\end{tikzpicture}\hspace{2mm}
\begin{tikzpicture}[baseline=(current bounding box.center),scale=0.8]
 \draw[<-] (3.2,0) .. controls (3.2,1.25) and (0,.25) .. (0,2)
     node[pos=0.85, shape=coordinate](X){};
  \draw[<-] (0,0) .. controls (0,1) and (.8,.8) .. (.8,2);
  \draw[<-] (.8,0) .. controls (.8,1) and (1.6,.8) .. (1.6,2);
  \draw[<-] (2.4,0) .. controls (2.4,1) and (3.2,.8) .. (3.2,2);
  \node at (1.6,.35) {$\dots$};
  \node at (2.4,1.65) {$\dots$};
\end{tikzpicture}\; \right]
\\&=
    \left[ \;\begin{tikzpicture}[baseline=(current bounding box.center),scale=0.8]
 %% Separate lines by 0.6
 %% UPWARD ORIENTED LINES
  \draw (3,0) .. controls ++(0,1.25) and ++(0,-1.75) .. (0,1.5);
  \draw (0,0) .. controls ++(0,1) and ++(0,-.7) .. (.6,1.5);
  \draw (.6,0) .. controls ++(0,1) and ++(0,-.7) .. (1.2,1.5);
  \draw (1.8,0) .. controls ++(0,1) and ++(0,-.7) .. (2.4,1.5);
  \draw (2.4,0) .. controls ++(0,1) and ++(0,-.7) .. (3,1.5);
  \node at (1.2,.35) {$\dots$};
  \node at (1.8,1.15) {$\dots$};
  %%
  %% DOWNWARD ORIENTED LINES
  \draw (6,0) .. controls ++(0,1.25) and ++(0,-1.75) .. (3.6,1.5) ;
  \draw (3.6,0) .. controls ++(0,1) and ++(0,-.7) .. (4.2,1.5);
  \draw (4.2,0) .. controls ++(0,1) and ++(0,-.7) .. (4.8,1.5);
  \draw (5.4,0) .. controls ++(0,1) and ++(0,-.7) .. (6,1.5);
  \node at (4.8,-.65) {$\dots$};
  \node at (5.4,1.85) {$\dots$};
  %%
  %% BLUE LINES
    \draw[blue, dotted] (-0.4,0) -- (6.4,0);
    \draw[blue, dotted] (-0.4,1.5) -- (6.4,1.5);
  %% TOP LEVEL
   \draw[->]  (3,1.5) .. controls ++(0,0.5) and ++(0,-.5) .. (3.6,2.5);
   \draw (3.6,1.5) .. controls ++(0,1) and ++(0,-1) .. (0,2.5);
   \draw[->]  (2.4,1.5) .. controls ++(0,0.5) and ++(0,-.5) .. (3,2.5);
   \draw[->]  (1.2,1.5) .. controls ++(0,0.5) and ++(0,-.5) .. (1.8,2.5);
   \draw[->]  (.6,1.5) .. controls ++(0,0.5) and ++(0,-.5) .. (1.2,2.5);
   \draw[->]  (0,1.5) .. controls ++(0,0.5) and ++(0,-.5) .. (.6,2.5);
   \draw (4.2,1.5) -- (4.2,2.5);
   \draw (4.8,1.5) -- (4.8,2.5);
   \draw (6,1.5) -- (6,2.5);
  %% BOTTOM LEVEL
   \draw  (2.4,0) .. controls ++(0,-.5) and ++(0,.5) .. (3,-1);
   \draw  (1.8,0) .. controls ++(0,-.5) and ++(0,.5) .. (2.4,-1);
   \draw  (.6,0) .. controls ++(0,-.5) and ++(0,.5) .. (1.2,-1);
   \draw  (0,0) .. controls ++(0,-.5) and ++(0,.5) .. (.6,-1);
   \draw  (3,0) .. controls ++(0,-.5) and ++(0,.5) .. (3.6,-1);
  \draw[,->] (3.6,0) .. controls ++(0,-1) and ++(0,+1) .. (0,-1);
   \draw[->] (4.2,0) -- (4.2,-1);
   \draw[->] (5.4,0) -- (5.4,-1);
   \draw[->] (6,0) -- (6,-1);
\end{tikzpicture}\;\right]
\\&=h_{-(2m+1)}h_{2n+1}(-2\bar{d}_0(2n+1)).
\end{align*}

Hence $[\h{(2n+1)}{},\h{-(2m+1)}{}]=\delta_{n,-m}(-2(2n+1))$.
 \end{proof}

Therefore the subset $A=\{\h{(2n+1)}{}\}_{n\in\mathbb{Z}}$ of the filtration degree zero part of $\TrHEv$ is isomorphic to the twisted Heisenberg algebra via

\begin{eqnarray*}
\phi:&\mathfrak{h}_{tw}&\stackrel{\sim}{\longrightarrow} A \\
&\h{\frac{2n+1}{2}}{}&\mapsto \frac{1}{2}\h{-(2n+1)}{}.\end{eqnarray*}

In the $W$-algebra $W^-$, we have an isomophic copy of the twisted Heisenberg algebra as well, given by $B=\{\omega_{2n+1,0}\}_{n\in\mathbb{Z}}$, with the isomorphism given by

\begin{eqnarray*}
\psi:&\mathfrak{h}_{tw}&\stackrel{\sim}{\longrightarrow} B \\
&\h{\frac{2n+1}{2}}{}&\mapsto \frac{1}{\sqrt{2}}\ \omega_{2n+1,0}.
\end{eqnarray*}

Therefore we have an isomorphism between the degree zero part of $\TrHEv$ and the degree zero part of $W^-$:

\begin{eqnarray*}
\psi\circ\phi^{-1}:&A&\stackrel{\sim}{\longrightarrow} B \\
&\h{-(2n+1)}{}&\mapsto \sqrt{2}\ w_{2n+1,0}.
\end{eqnarray*}

\subsection{Nonzero differential degree part of $\TrHEv$}

We have the following basic facts about diagrams in $\TrHEv^{>}$, which we may copy from the corresponding facts in the trace of the affine Hecke-Clifford algebra because of the triangular decomposition of $\TrHEv$ described in Proposition \ref{triangularDecomposition}.

\begin{proposition} \label{odd with odd dots zero}\cite[Propositions 3.9, 4.2]{Mike} In $Tr(\HC)$ for any $m,n \in \mathbb{Z}$, we have $$\h{2n+1}{x_1^{2m+1}} = 0,$$
$$\h{2n}{x_1^{2m}} =0.$$ \end{proposition} 

Hence any diagram containing an odd cycle with an odd number of dots or an even cycle with an even number of dots is zero. Therefore, the difference of the number of strands and number of solid dots must be odd. This agrees with the fact that in the $W$-algebra $W^-$, $l-k$ has to be an odd number for $w_{l,k}$.

The generators of $\TrHEv^{>}$ satisfy the following relations.

\begin{lemma}\label{commutators} For $m,n\in \mathbb{Z}$ with $mn>0$, we have \begin{enumerate} \item $[\h{2m }{x_1}, \h{2n}{ x_1}] = 2(n-m) \h{2n+2m}{x_1}.$
\item$ [\h{m}{c_1}, \h{n}{c_1}] = -2\h{n}{c_1}.$
\end{enumerate}
\end{lemma}
\begin{proof} Part (1) is a slight modification of \cite[Lemma 23]{CLLS}. By Proposition \ref{odd with odd dots zero}, if at least one of the indices inside the commutator is odd, the commutator will be zero. Hence we will work with the case where both indices are even numbers. The modification we need in \cite[Lemma 23]{CLLS} is a result of us having two resolution terms in our relations \eqref{dotSlide: bottomLeft} and \eqref{dotSlide: topLeft}. As a consequence of having even number of strands in both of our elements, canceling the two empty dots in our resolution terms give rise to the same sign as the other resolution term, hence we have a coefficient of two in front of our result.

Part (2) follows easily the proof of \cite[Lemma 23]{CLLS} since moving an empty dot through a crossing is for free in $\TrHEv$, and we get a negative sign from changing relative heights of hollow dots.
\end{proof}

\begin{lemma}\label{VirasorowithUp}  For $n\geq0$, we have $$[\h{\pm 2n}{(x_1 + \ldots + x_{2n})}, \h{ 1}{}] = \pm 4n \h{\pm(2n+1)}{}.$$
\end{lemma}
\begin{proof} First note that we have: 

\begin{align*}
\left[\hspace{2mm}\begin{tikzpicture}[baseline=(current bounding box.center),scale=0.75]
    \draw[->](0,0) to (1,2);
    \draw[->](1,0) to (0,2);
    \draw[fill](0.25,1.5) circle[radius=3pt];
    \draw [->](1.5,0) to (1.5,2);
\end{tikzpicture}\hspace{2mm}\right]\hspace{6pt}
&=
\hspace{6pt}
\left[\hspace{2mm}\begin{tikzpicture}[baseline=(current bounding box.center),scale=0.75]
\draw [->](0,0) to (1.5,2);
\draw [->](1.5,0) to (0,2);
\draw [fill](.25,1.63) circle[radius=3pt];
\draw[->](0.75,0) to [out=135,in=225] (0.75,2);
\end{tikzpicture}\hspace{2mm}\right] \hspace{6pt}
\\&= 
\hspace{6pt}
\left[\hspace{2mm}\begin{tikzpicture}[baseline=(current bounding box.center),scale=0.75]
\draw [->](0,0) to (1.5,2);
\draw [->](1.5,0) to (0,2);
\draw [fill](.65,1.2) circle[radius=3pt];
\draw[->](0.75,0) to [out=135,in=225] (0.75,2);
\end{tikzpicture}\hspace{2mm}\right]\hspace{6pt}
 +
 \hspace{6pt}
\left[\hspace{2mm}\begin{tikzpicture}[baseline=(current bounding box.center),scale=0.75]
\draw[->] (2,0) .. controls (2,1.25) and (0,.25) .. (0,2);
\draw[->] (0,0) .. controls (0,1) and (.8,.8) .. (1,2);
\draw[->] (1,0) .. controls (1,1) and (1.8,.8) .. (2,2);
\end{tikzpicture}\hspace{2mm}\right]\hspace{6pt}
-
\hspace{6pt}
\left[\hspace{2mm}\begin{tikzpicture}[baseline=(current bounding box.center),scale=0.75]
\draw[->] (2,0) .. controls (2,1.25) and (0,.25) .. (0,2);
\draw[->] (0,0) .. controls (0,1) and (.8,.8) .. (1,2);
\draw[->] (1,0) .. controls (1,1) and (1.8,.8) .. (2,2);
\draw (.08,1.5) circle[radius=3pt];
\draw (.67, 1.2) circle[radius=3pt];
\end{tikzpicture}\hspace{2mm}\right]\hspace{6pt}
\\&\numberthis \label{zeroDotMove} 
\hspace{6pt}=
\hspace{6pt}
\left[\hspace{2mm}\begin{tikzpicture}[baseline=(current bounding box.center),scale=0.75]
    \draw[->](0,0) to (1,2);
    \draw[->](1,0) to (0,2);
    \draw[fill](0.25,1.5) circle[radius=3pt];
    \draw [->](-.5,0) to (-.5,2);
\end{tikzpicture}\hspace{2mm}\right]\hspace{6pt}
+
\hspace{6pt}2
\left[\hspace{2mm}\begin{tikzpicture}[baseline=(current bounding box.center),scale=0.75]
\draw[->] (2,0) .. controls (2,1.25) and (0,.25) .. (0,2);
\draw[->] (0,0) .. controls (0,1) and (.8,.8) .. (1,2);
\draw[->] (1,0) .. controls (1,1) and (1.8,.8) .. (2,2);
\end{tikzpicture}\hspace{2mm}\right].
\end{align*}

Hence $[\h{2}{x_1}, \h{1}{}] = 2 \h{3}{}$. 

Next, moving the solid dot in $\h{2}{x_2}$ around to the bottom of the crossing using the trace relation gives:

\begin{align*}
\left[\hspace{2mm}\begin{tikzpicture}[baseline=(current bounding box.center),scale=0.75]
    \draw[->](0,0) to (1,2);
    \draw[->](1,0) to (0,2);
    \draw[fill](.75,.5) circle[radius=3pt];
    \draw [->](1.5,0) to (1.5,2);
\end{tikzpicture}\hspace{2mm}\right]\hspace{6pt}
&=
\hspace{6pt}
\left[\hspace{2mm}\begin{tikzpicture}[baseline=(current bounding box.center),scale=0.75]
\draw [->](0,0) to (1.5,2);
\draw [->](1.5,0) to (0,2);
\draw[->](0.75,0) to [out=45,in=-45] (0.75,2);
\draw [fill](.9,.75) circle[radius=3pt];
\end{tikzpicture}\hspace{2mm}\right] \hspace{6pt}
\\&= 
\hspace{6pt}
\left[\hspace{2mm}\begin{tikzpicture}[baseline=(current bounding box.center),scale=0.75]
\draw [->](0,0) to (1.5,2);
\draw [->](1.5,0) to (0,2);
\draw[->](0.75,0) to [out=45,in=-45] (0.75,2);
\draw [fill](1.3,.25) circle[radius=3pt];
\end{tikzpicture}\hspace{2mm}\right]\hspace{6pt}
 +
 \hspace{6pt}
\left[\hspace{2mm}\begin{tikzpicture}[baseline=(current bounding box.center),scale=0.75]
\draw[->] (2,0) .. controls (2,1.25) and (0,.25) .. (0,2);
\draw[->] (0,0) .. controls (0,1) and (.8,.8) .. (1,2);
\draw[->] (1,0) .. controls (1,1) and (1.8,.8) .. (2,2);
\end{tikzpicture}\hspace{2mm}\right]\hspace{6pt}
-
\hspace{6pt}
\left[\hspace{2mm}\begin{tikzpicture}[baseline=(current bounding box.center),scale=0.75]
\draw[->] (2,0) .. controls (2,1.25) and (0,.25) .. (0,2);
\draw[->] (0,0) .. controls (0,1) and (.8,.8) .. (1,2);
\draw[->] (1,0) .. controls (1,1) and (1.8,.8) .. (2,2);
\draw (1.77,.5) circle[radius=3pt];
\draw (1.05, .3) circle[radius=3pt];
\end{tikzpicture}\hspace{2mm}\right]\hspace{6pt}
\\&=
\hspace{6pt}
\left[\hspace{2mm}\begin{tikzpicture}[baseline=(current bounding box.center),scale=0.75]
    \draw[->](0,0) to (1,2);
    \draw[->](1,0) to (0,2);
    \draw[fill](.75,.5) circle[radius=3pt];
    \draw [->](-.5,0) to (-.5,2);
\end{tikzpicture}\hspace{2mm}\right]\hspace{6pt}
+ 
\hspace{6pt}2\left[\hspace{2mm}\begin{tikzpicture}[baseline=(current bounding box.center),scale=0.75]
\draw[->] (2,0) .. controls (2,1.25) and (0,.25) .. (0,2);
\draw[->] (0,0) .. controls (0,1) and (.8,.8) .. (1,2);
\draw[->] (1,0) .. controls (1,1) and (1.8,.8) .. (2,2);
\end{tikzpicture}\hspace{2mm}\right].
\end{align*}
So, $[\h{2}{(x_1 + x_2)}, \h{1}{} ] = 4\h{3}{}.$ 

Next, we claim that $[\h{2n}{x_{2n}}, \h{1}{}] = 2 \h{2n+1}{}$ for any $n$. Indeed, we have:
\begin{align*}
\left[\hspace{2mm}\begin{tikzpicture}[baseline=(current bounding box.center),scale=0.75]
  \draw[->] (3.2,0) .. controls (3.2,1.25) and (0,.25) .. (0,2)
     node[pos=0.85, shape=coordinate](X){};
  \draw[->] (0,0) .. controls (0,1) and (.8,.8) .. (.8,2);
  \draw[->] (.8,0) .. controls (.8,1) and (1.6,.8) .. (1.6,2);
  \draw[->] (2.4,0) .. controls (2.4,1) and (3.2,.8) .. (3.2,2);
  \node at (1.6,.35) {$\dots$};
  \node at (2.4,1.65) {$\dots$};
  \filldraw  (3.18,.2) circle (3pt);
\draw[->](3.5,0) to (3.5,2);
\end{tikzpicture}\hspace{2mm}\right]\hspace{6pt}
&=
\hspace{6pt}
\left[\hspace{2mm}\begin{tikzpicture}[baseline=(current bounding box.center),scale=0.75]
  \draw[->] (3.2,0) .. controls (3.2,1.25) and (0,.25) .. (0,2)
     node[pos=0.85, shape=coordinate](X){};
  \draw[->] (0,0) .. controls (0,1) and (.8,.8) .. (.8,2);
  \draw[->] (.8,0) .. controls (.8,1) and (1.6,.8) .. (1.6,2);
  \draw[->] (2.4,0) .. controls (2.4,1) and (3.2,.8) .. (3.2,2);
  \node at (1.6,.35) {$\dots$};
  \node at (2.4,1.65) {$\dots$};
  \filldraw  (2.8,.55) circle (3pt);
\draw[->](3,0) ..controls (3.2,1) and (3,.8)..(2.8,2);
\end{tikzpicture}\hspace{2mm}\right] \hspace{6pt}
\\&=
\hspace{6pt}
\left[\hspace{2mm}\begin{tikzpicture}[baseline=(current bounding box.center),scale=0.75]
  \draw[->] (3.2,0) .. controls (3.2,1.25) and (0,.25) .. (0,2)
     node[pos=0.85, shape=coordinate](X){};
  \draw[->] (0,0) .. controls (0,1) and (.8,.8) .. (.8,2);
  \draw[->] (.8,0) .. controls (.8,1) and (1.6,.8) .. (1.6,2);
  \draw[->] (2.4,0) .. controls (2.4,1) and (3.2,.8) .. (3.2,2);
  \node at (1.6,.35) {$\dots$};
  \node at (2.4,1.65) {$\dots$};
  \filldraw  (3.18,.2) circle (3pt);
\draw[->](2.8,0) ..controls (3,1) and (2.8,.8)..(2.6,2);
\end{tikzpicture}\hspace{2mm}\right]\hspace{6pt}
+
\hspace{6pt}
\left[\hspace{2mm}\begin{tikzpicture}[baseline=(current bounding box.center),scale=0.75]
  \draw[->] (3,0) .. controls (3,1.25) and (0,.25) .. (0,2)
     node[pos=0.85, shape=coordinate](X){};
  \draw[->] (0,0) .. controls (0,1) and (.8,.8) .. (.8,2);
  \draw[->] (.8,0) .. controls (.8,1) and (1.6,.8) .. (1.6,2);
  \draw[->] (2.4,0) .. controls (2.4,1) and (3.2,.8) .. (3.2,2);
  \node at (1.6,.35) {$\dots$};
  \node at (2.4,1.65) {$\dots$};
  %\filldraw  (3.18,.2) circle (3pt);
\draw[->](3.4,0) ..controls (3.6,1) and (3,.8)..(2.8,2);
\end{tikzpicture}\hspace{2mm}\right]\hspace{6pt}
\\&-
\hspace{6pt}
\left[\hspace{2mm}\begin{tikzpicture}[baseline=(current bounding box.center),scale=0.75]
  \draw[->] (3,0) .. controls (3,1.25) and (0,.25) .. (0,2)
     node[pos=0.85, shape=coordinate](X){};
  \draw[->] (0,0) .. controls (0,1) and (.8,.8) .. (.8,2);
  \draw[->] (.8,0) .. controls (.8,1) and (1.6,.8) .. (1.6,2);
  \draw[->] (2.4,0) .. controls (2.4,1) and (3.2,.8) .. (3.2,2);
  \node at (1.6,.35) {$\dots$};
  \node at (2.4,1.65) {$\dots$};
 \draw  (2.85,.4) circle (3pt);
\draw (3.43, .2) circle (3pt);
\draw[->](3.4,0) ..controls (3.6,1) and (3,.8)..(2.8,2);
\end{tikzpicture}\hspace{2mm}\right]\hspace{6pt}
\\&=
\hspace{6pt}
\left[\hspace{2mm}\begin{tikzpicture}[baseline=(current bounding box.center),scale=0.75]
  \draw[->] (3.2,0) .. controls (3.2,1.25) and (0,.25) .. (0,2)
     node[pos=0.85, shape=coordinate](X){};
  \draw[->] (0,0) .. controls (0,1) and (.8,.8) .. (.8,2);
  \draw[->] (.8,0) .. controls (.8,1) and (1.6,.8) .. (1.6,2);
  \draw[->] (2.4,0) .. controls (2.4,1) and (3.2,.8) .. (3.2,2);
  \node at (1.6,.35) {$\dots$};
  \node at (2.4,1.65) {$\dots$};
  \filldraw  (3.18,.2) circle (3pt);
\draw[->](-.3,0) to (-.3,2);
\end{tikzpicture}\hspace{2mm}\right]\hspace{6pt}
+
\hspace{6pt}2
\left[\hspace{2mm}\begin{tikzpicture}[baseline=(current bounding box.center),scale=0.75]
  \draw[->] (3.2,0) .. controls (3.2,1.25) and (0,.25) .. (0,2)
     node[pos=0.85, shape=coordinate](X){};
  \draw[->] (0,0) .. controls (0,1) and (.8,.8) .. (.8,2);
  \draw[->] (.8,0) .. controls (.8,1) and (1.6,.8) .. (1.6,2);
  \draw[->] (2.4,0) .. controls (2.4,1) and (3.2,.8) .. (3.2,2);
  \node at (1.6,.35) {$\dots$};
  \node at (2.4,1.65) {$\dots$};
  %\filldraw  (3.18,.2) circle (3pt);
  \draw[->] (2.8,0) .. controls (2.8,1) and (3.6,.8) .. (3.6,2);
%\draw[->](3.5,0) to (3.5,2);
\end{tikzpicture}\hspace{2mm}\right],
\end{align*}
where the last equality is obtained by pushing the crossings at the bottom of the diagrams without dots to the top. Indeed, diagrammatic calculations similar to the above give that $$[\h{2n}{ x_{a}}, \h{1}{}] = 2\h{2n+1}{} $$ for any $1 < a \leq 2n$.

Finally, note that $$\h{2n}{x_1} \h{1}{}\hspace{6pt} = \hspace{6pt}
\left[\hspace{2mm}\begin{tikzpicture}[baseline=(current bounding box.center),scale=0.75]
  \draw[->] (3.2,0) .. controls (3.2,1.25) and (0,.25) .. (0,2)
     node[pos=0.85, shape=coordinate](X){};
  \draw[->] (0,0) .. controls (0,1) and (.8,.8) .. (.8,2);
  \draw[->] (.8,0) .. controls (.8,1) and (1.6,.8) .. (1.6,2);
  \draw[->] (2.4,0) .. controls (2.4,1) and (3.2,.8) .. (3.2,2);
  \node at (1.6,.35) {$\dots$};
  \node at (2.4,1.65) {$\dots$};
  \filldraw  (.1,1.55) circle (3pt);
\draw[->](.4,0) ..controls (.4,1) and (.2,.8)..(.4,2);
\end{tikzpicture}\hspace{2mm}\right].
$$
The dot will slide over the top-leftmost crossing in the same manner as in Equation \eqref{zeroDotMove}, meaning the correction terms will cancel out. Hence, we have the desired result.
 \end{proof}

\begin{lemma}\label{upVirasorowithUp} Let $m$ be an odd integer. We have $$[ \h{2}{(x_1+ x_2)}, \h{m}{}] = 4m\h{m+2}{}.$$ \end{lemma}
\begin{proof}

We compute directly: 

\begin{align*}\h{m}{}\h{2}{x_1}
\hspace{6pt} &= \hspace{6pt} \left[ \;\;\;
\hackcenter{
\begin{tikzpicture}[scale=0.8]
 %% Separate lines by 0.6
 %% DOWNWARD ORIENTED LINES
  \draw[->] (3,0) .. controls ++(0,1.25) and ++(0,-1.75) .. (0,2);
  \draw[->] (0,0) .. controls ++(0,1) and ++(0,-1.2) .. (.6,2);
  \draw[->] (.6,0) .. controls ++(0,1) and ++(0,-1.2) .. (1.2,2);
  \draw[->] (1.8,0) .. controls ++(0,1) and ++(0,-1.2) .. (2.4,2);
  \draw[->] (2.4,0) .. controls ++(0,1) and ++(0,-1.2) .. (3,2);
  \node at (1.2,.35) {$\dots$};
  \node at (1.8,1.65) {$\dots$};
  %%
  %% UPWARD ORIENTED LINES
  \draw[->] (3.4,0)--(4.5,2);
  \draw[->] (4.5,0)--(3.4,2);
   \filldraw  (3.7,1.4) circle (2pt);
\end{tikzpicture}}
\; \;\;\right]\hspace{6pt}
\\&=
\hspace{6pt}
\left[ \;\;\;
\hackcenter{
\begin{tikzpicture}[scale=0.8]
 %% Separate lines by 0.6
 %% DOWNWARD ORIENTED LINES
  \draw (3,0) .. controls (3,1.25) and (0,.25) .. (0,1.5)
     node[pos=0.85, shape=coordinate](X){};
  \draw (0,0) .. controls (0,1) and (.6,.8) .. (.6,1.5);
  \draw (.6,0) .. controls (.6,1) and (1.2,.8) .. (1.2,1.5);
  \draw (1.8,0) .. controls (1.8,1) and (2.4,.8) .. (2.4,1.5);
  \draw (2.4,0) .. controls (2.4,1) and (3,.8) .. (3,1.5);
  \node at (1.2,.35) {$\dots$};
  \node at (1.8,1.15) {$\dots$};
  %%
  %% UPWARD ORIENTED LINES
  \draw (4.2,0) .. controls (4.2,1.25) and (3.6,.25) .. (3.6,1.5)
     node[pos=0.85, shape=coordinate](X){};
  \draw (3.6,0) .. controls (3.6,1) and (4.2,.8) .. (4.2,1.5);
  %%
  %% BLUE LINES
    \draw[blue, dotted] (-0.4,0) -- (4.4,0);
    \draw[blue, dotted] (-0.4,1.5) -- (4.4,1.5);
  %% TOP LEVEL
   \draw[->]  (3,1.5) .. controls (3,2) and (3.6,2) .. (3.6,2.5);
   \draw[->] (3.6,1.5) .. controls (3.6,2.5) and (0,1.5) .. (0,2.5);
   \draw[->]  (2.4,1.5) .. controls (2.4,2) and (3,2) .. (3,2.5);
   \draw[->]  (1.2,1.5) .. controls (1.2,2) and (1.8,2) .. (1.8,2.5);
   \draw[->]  (.6,1.5) .. controls (.6,2) and (1.2,2) .. (1.2,2.5);
   \draw[->]  (0,1.5) .. controls (0,2) and (.6,2) .. (.6,2.5);
   \draw[->] (4.2,1.5) -- (4.2,2.5);
     \filldraw  (3.6,1.5) circle (2pt);
  %% BOTTOM LEVEL
   \draw  (2.4,0) .. controls (2.4,-.5) and (3,-.5) .. (3,-1);
   \draw  (1.8,0) .. controls (1.8,-.5) and (2.4,-.5) .. (2.4,-1);
   \draw  (.6,0) .. controls (.6,-.5) and (1.2,-.5) .. (1.2,-1);
   \draw  (0,0) .. controls (0,-.5) and (.6,-.5) .. (.6,-1);
   \draw  (3,0) .. controls (3,-.5) and (3.6,-.5) .. (3.6,-1);
  \draw (3.6,0) .. controls (3.6,-1) and (0,0) .. (0,-1);
   \draw (4.2,0) -- (4.2,-1);
\end{tikzpicture}}
\right]
\hspace{6pt}
\\&=
\hspace{6pt}
\left[ \;\;\;
\hackcenter{
\begin{tikzpicture}[scale=0.8]
 %% Separate lines by 0.6
 %% DOWNWARD ORIENTED LINES
  \draw (3,0) .. controls (3,1.25) and (0,.25) .. (0,1.5)
     node[pos=0.85, shape=coordinate](X){};
  \draw (0,0) .. controls (0,1) and (.6,.8) .. (.6,1.5);
  \draw (.6,0) .. controls (.6,1) and (1.2,.8) .. (1.2,1.5);
  \draw (1.8,0) .. controls (1.8,1) and (2.4,.8) .. (2.4,1.5);
  \draw (2.4,0) .. controls (2.4,1) and (3,.8) .. (3,1.5);
  \node at (1.2,.35) {$\dots$};
  \node at (1.8,1.15) {$\dots$};
  %%
  %% UPWARD ORIENTED LINES
  \draw (4.2,0) .. controls (4.2,1.25) and (3.6,.25) .. (3.6,1.5)
     node[pos=0.85, shape=coordinate](X){};
  \draw (3.6,0) .. controls (3.6,1) and (4.2,.8) .. (4.2,1.5);
  %%
  %% BLUE LINES
    \draw[blue, dotted] (-0.4,0) -- (4.4,0);
    \draw[blue, dotted] (-0.4,1.5) -- (4.4,1.5);
  %% TOP LEVEL
   \draw[->]  (3,1.5) .. controls (3,2) and (3.6,2) .. (3.6,2.5);
   \draw[->] (3.6,1.5) .. controls (3.6,2.5) and (0,1.5) .. (0,2.5);
   \draw[->]  (2.4,1.5) .. controls (2.4,2) and (3,2) .. (3,2.5);
   \draw[->]  (1.2,1.5) .. controls (1.2,2) and (1.8,2) .. (1.8,2.5);
   \draw[->]  (.6,1.5) .. controls (.6,2) and (1.2,2) .. (1.2,2.5);
   \draw[->]  (0,1.5) .. controls (0,2) and (.6,2) .. (.6,2.5);
   \draw[->] (4.2,1.5) -- (4.2,2.5);
     \filldraw  (3.0,1.95) circle (2pt);
  %% BOTTOM LEVEL
   \draw  (2.4,0) .. controls (2.4,-.5) and (3,-.5) .. (3,-1);
   \draw  (1.8,0) .. controls (1.8,-.5) and (2.4,-.5) .. (2.4,-1);
   \draw  (.6,0) .. controls (.6,-.5) and (1.2,-.5) .. (1.2,-1);
   \draw  (0,0) .. controls (0,-.5) and (.6,-.5) .. (.6,-1);
   \draw  (3,0) .. controls (3,-.5) and (3.6,-.5) .. (3.6,-1);
  \draw (3.6,0) .. controls (3.6,-1) and (0,0) .. (0,-1);
   \draw (4.2,0) -- (4.2,-1);
\end{tikzpicture}}
\right]
\hspace{6pt} 
-
\hspace{6pt}
 \left[ \;\;\;\hackcenter{
\begin{tikzpicture}[scale=0.8]
 %% Separate lines by 0.6
 %% DOWNWARD ORIENTED LINES
  \draw (3,0) .. controls (3,1.25) and (0,.25) .. (0,1.5)
     node[pos=0.85, shape=coordinate](X){};
  \draw (0,0) .. controls (0,1) and (.6,.8) .. (.6,1.5);
  \draw (.6,0) .. controls (.6,1) and (1.2,.8) .. (1.2,1.5);
  \draw (1.8,0) .. controls (1.8,1) and (2.4,.8) .. (2.4,1.5);
  \draw (2.4,0) .. controls (2.4,1) and (3,.8) .. (3,1.5);
  \node at (1.2,.35) {$\dots$};
  \node at (1.8,1.15) {$\dots$};
  %%
  %% UPWARD ORIENTED LINES
  \draw (4.2,0) .. controls (4.2,1.25) and (3.6,.25) .. (3.6,1.5)
     node[pos=0.85, shape=coordinate](X){};
  \draw (3.6,0) .. controls (3.6,1) and (4.2,.8) .. (4.2,1.5);
  %%
  %% BLUE LINES
    \draw[blue, dotted] (-0.4,0) -- (4.4,0);
    \draw[blue, dotted] (-0.4,1.5) -- (4.4,1.5);
  %% TOP LEVEL
   \draw[->]  (3,1.5) .. controls (3,2) and (0,2) .. (0,2.5);
   \draw[->] (3.6,1.5)--(3.6,2.5);
   \draw[->]  (2.4,1.5) .. controls (2.4,2) and (3,2) .. (3,2.5);
   \draw[->]  (1.2,1.5) .. controls (1.2,2) and (1.8,2) .. (1.8,2.5);
   \draw[->]  (.6,1.5) .. controls (.6,2) and (1.2,2) .. (1.2,2.5);
   \draw[->]  (0,1.5) .. controls (0,2) and (.6,2) .. (.6,2.5);
   \draw[->] (4.2,1.5) -- (4.2,2.5);
    % \filldraw  (3.0,2.0) circle (2pt);
  %% BOTTOM LEVEL
   \draw  (2.4,0) .. controls (2.4,-.5) and (3,-.5) .. (3,-1);
   \draw  (1.8,0) .. controls (1.8,-.5) and (2.4,-.5) .. (2.4,-1);
   \draw  (.6,0) .. controls (.6,-.5) and (1.2,-.5) .. (1.2,-1);
   \draw  (0,0) .. controls (0,-.5) and (.6,-.5) .. (.6,-1);
   \draw  (3,0) .. controls (3,-.5) and (3.6,-.5) .. (3.6,-1);
  \draw (3.6,0) .. controls (3.6,-1) and (0,0) .. (0,-1);
   \draw (4.2,0) -- (4.2,-1);
\end{tikzpicture}}
\right]\hspace{6pt}
 \\&+
 \hspace{6pt}
\left[ \;\;\;\hackcenter{
\begin{tikzpicture}[scale=0.8]
 %% Separate lines by 0.6
 %% DOWNWARD ORIENTED LINES
  \draw (3,0) .. controls (3,1.25) and (0,.25) .. (0,1.5)
     node[pos=0.85, shape=coordinate](X){};
  \draw (0,0) .. controls (0,1) and (.6,.8) .. (.6,1.5);
  \draw (.6,0) .. controls (.6,1) and (1.2,.8) .. (1.2,1.5);
  \draw (1.8,0) .. controls (1.8,1) and (2.4,.8) .. (2.4,1.5);
  \draw (2.4,0) .. controls (2.4,1) and (3,.8) .. (3,1.5);
  \node at (1.2,.35) {$\dots$};
  \node at (1.8,1.15) {$\dots$};
  %%
  %% UPWARD ORIENTED LINES
  \draw (4.2,0) .. controls (4.2,1.25) and (3.6,.25) .. (3.6,1.5)
     node[pos=0.85, shape=coordinate](X){};
  \draw (3.6,0) .. controls (3.6,1) and (4.2,.8) .. (4.2,1.5);
\draw (3.6, 1.2) circle (3pt);
  %%
  %% BLUE LINES
    \draw[blue, dotted] (-0.4,0) -- (4.4,0);
    \draw[blue, dotted] (-0.4,1.5) -- (4.4,1.5);
  %% TOP LEVEL
   \draw[->]  (3,1.5) .. controls (3,2) and (0,2) .. (0,2.5);
   \draw[->] (3.6,1.5)--(3.6,2.5);
   \draw[->]  (2.4,1.5) .. controls (2.4,2) and (3,2) .. (3,2.5);
   \draw[->]  (1.2,1.5) .. controls (1.2,2) and (1.8,2) .. (1.8,2.5);
   \draw[->]  (.6,1.5) .. controls (.6,2) and (1.2,2) .. (1.2,2.5);
   \draw[->]  (0,1.5) .. controls (0,2) and (.6,2) .. (.6,2.5);
   \draw[->] (4.2,1.5) -- (4.2,2.5);
\draw (2.8,1.7) circle (3pt);
    % \filldraw  (3.0,2.0) circle (2pt);
  %% BOTTOM LEVEL
   \draw  (2.4,0) .. controls (2.4,-.5) and (3,-.5) .. (3,-1);
   \draw  (1.8,0) .. controls (1.8,-.5) and (2.4,-.5) .. (2.4,-1);
   \draw  (.6,0) .. controls (.6,-.5) and (1.2,-.5) .. (1.2,-1);
   \draw  (0,0) .. controls (0,-.5) and (.6,-.5) .. (.6,-1);
   \draw  (3,0) .. controls (3,-.5) and (3.6,-.5) .. (3.6,-1);
  \draw (3.6,0) .. controls (3.6,-1) and (0,0) .. (0,-1);
   \draw (4.2,0) -- (4.2,-1);
\end{tikzpicture}}
\right].
\end{align*}

Cancelling the empty dots in the last term results in a change in sign, and both of the latter diagrams are $(m+2)$-cycles. Hence we have:
\[=\hspace{6pt}\left[ \;\;\;
\hackcenter{
\begin{tikzpicture}[scale=0.8]
 %% Separate lines by 0.6
 %% DOWNWARD ORIENTED LINES
  \draw (3,0) .. controls (3,1.25) and (0,.25) .. (0,1.5)
     node[pos=0.85, shape=coordinate](X){};
  \draw (0,0) .. controls (0,1) and (.6,.8) .. (.6,1.5);
  \draw (.6,0) .. controls (.6,1) and (1.2,.8) .. (1.2,1.5);
  \draw (1.8,0) .. controls (1.8,1) and (2.4,.8) .. (2.4,1.5);
  \draw (2.4,0) .. controls (2.4,1) and (3,.8) .. (3,1.5);
  \node at (1.2,.35) {$\dots$};
  \node at (1.8,1.15) {$\dots$};
  %%
  %% UPWARD ORIENTED LINES
  \draw (4.2,0) .. controls (4.2,1.25) and (3.6,.25) .. (3.6,1.5)
     node[pos=0.85, shape=coordinate](X){};
  \draw (3.6,0) .. controls (3.6,1) and (4.2,.8) .. (4.2,1.5);
  %%
  %% BLUE LINES
    \draw[blue, dotted] (-0.4,0) -- (4.4,0);
    \draw[blue, dotted] (-0.4,1.5) -- (4.4,1.5);
  %% TOP LEVEL
   \draw[->]  (3,1.5) .. controls (3,2) and (3.6,2) .. (3.6,2.5);
   \draw[->] (3.6,1.5) .. controls (3.6,2.5) and (0,1.5) .. (0,2.5);
   \draw[->]  (2.4,1.5) .. controls (2.4,2) and (3,2) .. (3,2.5);
   \draw[->]  (1.2,1.5) .. controls (1.2,2) and (1.8,2) .. (1.8,2.5);
   \draw[->]  (.6,1.5) .. controls (.6,2) and (1.2,2) .. (1.2,2.5);
   \draw[->]  (0,1.5) .. controls (0,2) and (.6,2) .. (.6,2.5);
   \draw[->] (4.2,1.5) -- (4.2,2.5);
     \filldraw  (3.0,1.95) circle (2pt);
  %% BOTTOM LEVEL
   \draw  (2.4,0) .. controls (2.4,-.5) and (3,-.5) .. (3,-1);
   \draw  (1.8,0) .. controls (1.8,-.5) and (2.4,-.5) .. (2.4,-1);
   \draw  (.6,0) .. controls (.6,-.5) and (1.2,-.5) .. (1.2,-1);
   \draw  (0,0) .. controls (0,-.5) and (.6,-.5) .. (.6,-1);
   \draw  (3,0) .. controls (3,-.5) and (3.6,-.5) .. (3.6,-1);
  \draw (3.6,0) .. controls (3.6,-1) and (0,0) .. (0,-1);
   \draw (4.2,0) -- (4.2,-1);
\end{tikzpicture}}
\right]
\hspace{6pt} -\hspace{6pt} 2\h{m+2}{}
\]

Sliding the solid dot in the first diagram all the way to the left results in $m$ total crossing resolutions, each of which yieds a term of $-2 \h{m+2}{}$. So,
\begin{align*}&=\hspace{6pt}\left[ \;\;\;
\hackcenter{
\begin{tikzpicture}[scale=0.8]
 %% Separate lines by 0.6
 %% DOWNWARD ORIENTED LINES
  \draw (3,0) .. controls (3,1.25) and (0,.25) .. (0,1.5)
     node[pos=0.85, shape=coordinate](X){};
  \draw (0,0) .. controls (0,1) and (.6,.8) .. (.6,1.5);
  \draw (.6,0) .. controls (.6,1) and (1.2,.8) .. (1.2,1.5);
  \draw (1.8,0) .. controls (1.8,1) and (2.4,.8) .. (2.4,1.5);
  \draw (2.4,0) .. controls (2.4,1) and (3,.8) .. (3,1.5);
  \node at (1.2,.35) {$\dots$};
  \node at (1.8,1.15) {$\dots$};
  %%
  %% UPWARD ORIENTED LINES
  \draw (4.2,0) .. controls (4.2,1.25) and (3.6,.25) .. (3.6,1.5)
     node[pos=0.85, shape=coordinate](X){};
  \draw (3.6,0) .. controls (3.6,1) and (4.2,.8) .. (4.2,1.5);
  %%
  %% BLUE LINES
    \draw[blue, dotted] (-0.4,0) -- (4.4,0);
    \draw[blue, dotted] (-0.4,1.5) -- (4.4,1.5);
  %% TOP LEVEL
   \draw[->]  (3,1.5) .. controls (3,2) and (3.6,2) .. (3.6,2.5);
   \draw[->] (3.6,1.5) .. controls (3.6,2.5) and (0,1.5) .. (0,2.5);
   \draw[->]  (2.4,1.5) .. controls (2.4,2) and (3,2) .. (3,2.5);
   \draw[->]  (1.2,1.5) .. controls (1.2,2) and (1.8,2) .. (1.8,2.5);
   \draw[->]  (.6,1.5) .. controls (.6,2) and (1.2,2) .. (1.2,2.5);
   \draw[->]  (0,1.5) .. controls (0,2) and (.6,2) .. (.6,2.5);
   \draw[->] (4.2,1.5) -- (4.2,2.5);
     \filldraw  (.15,2.2) circle (2pt);
  %% BOTTOM LEVEL
   \draw  (2.4,0) .. controls (2.4,-.5) and (3,-.5) .. (3,-1);
   \draw  (1.8,0) .. controls (1.8,-.5) and (2.4,-.5) .. (2.4,-1);
   \draw  (.6,0) .. controls (.6,-.5) and (1.2,-.5) .. (1.2,-1);
   \draw  (0,0) .. controls (0,-.5) and (.6,-.5) .. (.6,-1);
   \draw  (3,0) .. controls (3,-.5) and (3.6,-.5) .. (3.6,-1);
  \draw (3.6,0) .. controls (3.6,-1) and (0,0) .. (0,-1);
   \draw (4.2,0) -- (4.2,-1);
\end{tikzpicture}}
\right]
\;\; \;\; - \hspace{6pt}2m\h{m+2}{}
\\& 
=\;\; \left[ \;\;\;
\hackcenter{
\begin{tikzpicture}[baseline=(current bounding box.center),scale=.8]
  %%
  %% UPWARD ORIENTED LINES
  \draw[->] (0,0)--(1.2,2);
  \draw[->] (1.2,0)--(0,2);
   \filldraw  (.35,1.4) circle (2pt);
\end{tikzpicture}
\;\;\;
\begin{tikzpicture}[baseline=(current bounding box.center),scale=0.8]
 %% Separate lines by 0.6
 %% DOWNWARD ORIENTED LINES
  \draw[->] (3,0) .. controls ++(0,1.25) and ++(0,-1.75) .. (0,2);
  \draw[->] (0,0) .. controls ++(0,1) and ++(0,-1.2) .. (.6,2);
  \draw[->] (.6,0) .. controls ++(0,1) and ++(0,-1.2) .. (1.2,2);
  \draw[->] (1.8,0) .. controls ++(0,1) and ++(0,-1.2) .. (2.4,2);
  \draw[->] (2.4,0) .. controls ++(0,1) and ++(0,-1.2) .. (3,2);
  \node at (1.2,.35) {$\dots$};
  \node at (1.8,1.65) {$\dots$};
\end{tikzpicture}}\;\;\;
\right]
 \;\; - \hspace{6pt}2m\h{m+2}{}\end{align*}
Hence we have $$[\h{2}{x_1}, \h{m}{}] = 2m\h{m+2}{}.$$ A similar computation gives that $$[\h{2}{x_2}, \h{m}{}] = 2m\h{m+2}{},$$ giving the desired result.
\end{proof}

\begin{lemma}\label{upVirasorowithDown} We have $$[\h{2n}{(x_1+x_2 +\ldots+ x_{2n})}, \h{-(2m+1)}{}] = \left\{\begin{array}{lr} -4(2m+1) \h{2n-2m-1}{} & \qquad \text{if } n>m\geq 1 \\ 0 & \qquad \text{if } n=m\geq 1 \\ -2(2m+1) \h{2n-2m-1}{}& \qquad \text{if } 1\leq n < m.\end{array}\right.$$\end{lemma}
\begin{proof} We follow the methods of \cite[Lemma 26]{CLLS}, substituting our new relations as necessary.

As in that case, let $\beta_n = \h{2n}{x_1}$ and $\alpha_m = \h{2m+1}{x_1}$, and proceed by induction on $m$. When $m=1$, we can compute directly:
\begin{align}\label{Lm26 setup}
 \left[ \;\;\;
\hackcenter{
\begin{tikzpicture}[scale=0.8]
  \draw[->] (3,0) .. controls ++(0,1.25) and (0,.25) .. (0,2);
  \draw[->] (0.0,0) .. controls ++(0,1) and ++(0,-1.2) .. (0.6,2);
  \draw[->] (0.6,0) .. controls ++(0,1) and ++(0,-1.2) .. (1.2,2);
  \draw[->] (1.2,0) .. controls ++(0,1) and ++(0,-1.2) .. (1.8,2);
  \draw[->] (2.4,0) .. controls ++(0,1) and ++(0,-1.2) .. (3,2);
  \node at (1.8,.35) {$\dots$};
  \node at (2.4,1.65) {$\dots$};
  \draw[<-] (-0.6,0) -- (-0.6,2);
  \filldraw  (.05,1.6) circle (2pt);
\end{tikzpicture}}
 \;\;\right]
\;\;  &\refequal{\eqref{R2}} \;\;
 \left[ \;\;\;
\hackcenter{
\begin{tikzpicture}[scale=0.8]
  \draw (3.0,0) .. controls ++(0,1.25) and ++(0,-1.1) .. (0,1.5);
  \draw (0.0,0) .. controls ++(0,.5) and ++(0,-.6) .. (0.6,1.5);
  \draw (0.6,0) .. controls ++(0,.5) and ++(0,-.6) .. (1.2,1.5);
  \draw (1.2,0) .. controls ++(0,.5) and ++(0,-.6) .. (1.8,1.5);
  \draw (2.4,0) .. controls ++(0,.5) and ++(0,-.6) .. (3,1.5);
  \node at (1.8,.35) {$\dots$};
  \node at (2.4,1.35) {$\dots$};
  \draw[<-] (-0.6,0) -- (-0.6,1.5);
  \filldraw  (.1,1.2) circle (2pt);
    %%
  %% BLUE LINES
    \draw[blue, dotted] (-1,0) -- (3.4,0);
    \draw[blue, dotted] (-1,1.5) -- (3.4,1.5);
  %% TOP LINES
   \draw[->] (0.0,1.5) .. controls ++(0,.45) and ++(0,-.6) .. (-0.6,2.5);
   \draw (-0.6,1.5) .. controls ++(0,.45) and ++(0,-.6) .. (0.0,2.5);
   \draw[->] (0.6,1.5) --  (0.6,2.5);
   \draw[->] (1.2,1.5) --  (1.2,2.5);
   \draw[->] (1.8,1.5) --  (1.8,2.5);
   \draw[->] (3,1.5) --  (3,2.5);
   %% bottom LINES
   \draw[<-] (0.0,0) .. controls ++(0,-.45) and ++(0,.6) .. (-0.6,-1);
   \draw[->] (-0.6,0) .. controls ++(0,-.45) and ++(0,.6) .. (0.0,-1);
   \draw[<-] (0.6,0) --  (0.6,-1);
   \draw[<-] (1.2,0) --  (1.2,-1);
   \draw[<-] (2.4,0) --  (2.4,-1);
   \draw[<-] (3.0,0) --  (3,-1);
\end{tikzpicture}}
\;\;\right]
\\ &+\hspace{6pt}2 \;\;
 \left[ \;\;\;
\hackcenter{
\begin{tikzpicture}[scale=0.8]
  \draw (3.0,0) .. controls ++(0,1.25) and ++(0,-1.1) .. (0,1.5);
  \draw (0.0,0) .. controls ++(0,.5) and ++(0,-.6) .. (0.6,1.5);
  \draw (0.6,0) .. controls ++(0,.5) and ++(0,-.6) .. (1.2,1.5);
  \draw (1.2,0) .. controls ++(0,.5) and ++(0,-.6) .. (1.8,1.5);
  \draw (2.4,0) .. controls ++(0,.5) and ++(0,-.6) .. (3,1.5);
  \node at (1.8,.35) {$\dots$};
  \node at (2.4,1.35) {$\dots$};
  \draw (-0.6,0) -- (-0.6,1.5);
  \filldraw  (.1,1.2) circle (2pt);
    %%
  %% BLUE LINES
    \draw[blue, dotted] (-1,0) -- (3.4,0);
    \draw[blue, dotted] (-1,1.5) -- (3.4,1.5);
  %% TOP LINES
  % \draw[thick,->] (0.0,1.5) .. controls ++(0,.45) and ++(0,-.6) .. (-0.6,2.5);
   \draw (-0.6,1.5) .. controls ++(0,.35) and ++(0,.35) .. (0.0,1.5);
   \draw[->] (0.6,1.5) --  (0.6,2.5);
   \draw[->] (1.2,1.5) --  (1.2,2.5);
   \draw[->] (1.8,1.5) --  (1.8,2.5);
   \draw[->] (3,1.5) --  (3,2.5);
   %% bottom LINES
   %\draw[thick,<-] (0.0,0) .. controls ++(0,-.45) and ++(0,.6) .. (-0.6,-1);
   \draw[->] (-0.6,0) .. controls ++(0,-.35) and ++(0,-.35) .. (0.0,0);
   \draw[<-] (0.6,0) --  (0.6,-1);
   \draw[<-] (1.2,0) --  (1.2,-1);
   \draw[<-] (2.4,0) --  (2.4,-1);
   \draw[<-] (3.0,0) --  (3,-1);
\end{tikzpicture}}
\;\;\right]
\end{align}
where the trailing terms arising from relation \eqref{R2} have the same sign after cancelling the empty dots, and thus add together. We claim that the diagram in the second term is $\h{2n-1}{}$. Indeed, sliding the dot gives:
\begin{align*}
 \;\;
 \left[ \;\;\;
\hackcenter{
\begin{tikzpicture}[scale=0.8]
  \draw (3.0,0) .. controls ++(0,1.25) and ++(0,-1.1) .. (0,1.5);
  \draw (0.0,0) .. controls ++(0,.5) and ++(0,-.6) .. (0.6,1.5);
  \draw (0.6,0) .. controls ++(0,.5) and ++(0,-.6) .. (1.2,1.5);
  \draw (1.2,0) .. controls ++(0,.5) and ++(0,-.6) .. (1.8,1.5);
  \draw (2.4,0) .. controls ++(0,.5) and ++(0,-.6) .. (3,1.5);
  \node at (1.8,.35) {$\dots$};
  \node at (2.4,1.35) {$\dots$};
  \draw (-0.6,0) -- (-0.6,1.5);
  \filldraw  (.3,.65) circle (2pt);
    %%
  %% BLUE LINES
    \draw[blue, dotted] (-1,0) -- (3.4,0);
    \draw[blue, dotted] (-1,1.5) -- (3.4,1.5);
  %% TOP LINES
  % \draw[thick,->] (0.0,1.5) .. controls ++(0,.45) and ++(0,-.6) .. (-0.6,2.5);
   \draw(-0.6,1.5) .. controls ++(0,.35) and ++(0,.35) .. (0.0,1.5);
   \draw[->] (0.6,1.5) --  (0.6,2.5);
   \draw[->] (1.2,1.5) --  (1.2,2.5);
   \draw[->] (1.8,1.5) --  (1.8,2.5);
   \draw[->] (3,1.5) --  (3,2.5);
   %% bottom LINES
   %\draw[thick,<-] (0.0,0) .. controls ++(0,-.45) and ++(0,.6) .. (-0.6,-1);
   \draw[->] (-0.6,0) .. controls ++(0,-.35) and ++(0,-.35) .. (0.0,0);
   \draw[<-] (0.6,0) --  (0.6,-1);
   \draw[<-] (1.2,0) --  (1.2,-1);
   \draw[<-] (2.4,0) --  (2.4,-1);
   \draw[<-] (3.0,0) --  (3,-1);
\end{tikzpicture}}
\;\;\right] \;\; &\refequal{\eqref{dotSlide: bottomLeft}} \;\;
\overline{d}_{0,0}\h{2n-1}{} + \overline{d}_{0,1}\h{2n-1}{}
= \h{2n-1}{}
\end{align*}
by relations \eqref{d00=1} and \eqref{d01=0}.

 Now, sliding the solid dot over the crossing on the right hand side of Equation \eqref{Lm26 setup} gives:

\[
\;\; \refequal{\eqref{dotSlide: bottomLeft}} \;\;
 \left[ \;\;\;
\hackcenter{
\begin{tikzpicture}[scale=0.8]
  \draw (3.0,0) .. controls ++(0,1.25) and ++(0,-1.1) .. (0,1.5);
  \draw (0.0,0) .. controls ++(0,.5) and ++(0,-.6) .. (0.6,1.5);
  \draw (0.6,0) .. controls ++(0,.5) and ++(0,-.6) .. (1.2,1.5);
  \draw (1.2,0) .. controls ++(0,.5) and ++(0,-.6) .. (1.8,1.5);
  \draw (2.4,0) .. controls ++(0,.5) and ++(0,-.6) .. (3,1.5);
  \node at (1.8,.35) {$\dots$};
  \node at (2.4,1.35) {$\dots$};
  \draw[<-] (-0.6,0) -- (-0.6,1.5);
  \filldraw  (-0.55,2.2) circle (2pt);
    %%
  %% BLUE LINES
    \draw[blue, dotted] (-1,0) -- (3.4,0);
    \draw[blue, dotted] (-1,1.5) -- (3.4,1.5);
  %% TOP LINES
   \draw[->] (0.0,1.5) .. controls ++(0,.45) and ++(0,-.6) .. (-0.6,2.5);
   \draw (-0.6,1.5) .. controls ++(0,.45) and ++(0,-.6) .. (0.0,2.5);
   \draw[->] (0.6,1.5) --  (0.6,2.5);
   \draw[->] (1.2,1.5) --  (1.2,2.5);
   \draw[->] (1.8,1.5) --  (1.8,2.5);
   \draw[->] (3,1.5) --  (3,2.5);
   %% bottom LINES
   \draw[<-] (0.0,0) .. controls ++(0,-.45) and ++(0,.6) .. (-0.6,-1);
   \draw[->] (-0.6,0) .. controls ++(0,-.45) and ++(0,.6) .. (0.0,-1);
   \draw[<-] (0.6,0) --  (0.6,-1);
   \draw[<-] (1.2,0) --  (1.2,-1);
   \draw[<-] (2.4,0) --  (2.4,-1);
   \draw[<-] (3.0,0) --  (3,-1);
\end{tikzpicture}}
\;\;\right]
+\hspace{6pt}2 \;\;
 \left[ \;\;\;
\hackcenter{
\begin{tikzpicture}[scale=0.8]
  \draw (3.0,0) .. controls ++(0,1.25) and ++(0,-1.1) .. (0,1.5);
  \draw (0.0,0) .. controls ++(0,.5) and ++(0,-.6) .. (0.6,1.5);
  \draw (0.6,0) .. controls ++(0,.5) and ++(0,-.6) .. (1.2,1.5);
  \draw (1.2,0) .. controls ++(0,.5) and ++(0,-.6) .. (1.8,1.5);
  \draw (2.4,0) .. controls ++(0,.5) and ++(0,-.6) .. (3,1.5);
  \node at (1.8,.35) {$\dots$};
  \node at (2.4,1.35) {$\dots$};
  \draw[<-] (-0.6,0) -- (-0.6,1.5);
    %%
  %% BLUE LINES
    \draw[blue, dotted] (-1,0) -- (3.4,0);
    \draw[blue, dotted] (-1,1.5) -- (3.4,1.5);
  %% TOP LINES
%\draw[->] (0.0,1.5) -- (0.0,2.5);
%   \draw (-0.6,1.5) -- (-0.6,2.5);
   \draw[->] (0.0,2.5) arc (360:180:3mm);
   \draw[->] (0.0,1.5) arc (0:180:3mm);
   \draw[->] (0.6,1.5) --  (0.6,2.5);
   \draw[->] (1.2,1.5) --  (1.2,2.5);
   \draw[->] (1.8,1.5) --  (1.8,2.5);
   \draw[->] (3,1.5) --  (3,2.5);
   %% bottom LINES
   \draw[<-] (0.0,0) .. controls ++(0,-.45) and ++(0,.6) .. (-0.6,-1);
   \draw[->] (-0.6,0) .. controls ++(0,-.45) and ++(0,.6) .. (0.0,-1);
   \draw[<-] (0.6,0) --  (0.6,-1);
   \draw[<-] (1.2,0) --  (1.2,-1);
   \draw[<-] (2.4,0) --  (2.4,-1);
   \draw[<-] (3.0,0) --  (3,-1);
\end{tikzpicture}}
\;\;\right]
\;\; 
\]
where the trailing terms arising from relation \eqref{dotSlide: bottomLeft} have the same sign after canceling the empty dots, and thus add together.  We can use the trace relation to slide the top cup in the second term to the bottom; after simplication, this term is therefore equal to $\h{n-1}{}$.  The first term is equal to $\beta_n\alpha_{-1}$ as in \cite[Lemma 26]{CLLS}. Thus, 

\begin{center}
$ \left[ \;\;\;
\hackcenter{
\begin{tikzpicture}[scale=0.8]
  \draw[->] (3,0) .. controls ++(0,1.25) and (0,.25) .. (0,2);
  \draw[->] (0.0,0) .. controls ++(0,1) and ++(0,-1.2) .. (0.6,2);
  \draw[->] (0.6,0) .. controls ++(0,1) and ++(0,-1.2) .. (1.2,2);
  \draw[->] (1.2,0) .. controls ++(0,1) and ++(0,-1.2) .. (1.8,2);
  \draw[->] (2.4,0) .. controls ++(0,1) and ++(0,-1.2) .. (3,2);
  \node at (1.8,.35) {$\dots$};
  \node at (2.4,1.65) {$\dots$};
  \draw[<-] (-0.6,0) -- (-0.6,2);
  \filldraw  (.05,1.6) circle (2pt);
\end{tikzpicture}}
 \;\;\right]
\;\;
=
 \;\;
 \left[ \;\;\;
\hackcenter{
\begin{tikzpicture}[scale=0.8]
  \draw[->] (3,0) .. controls ++(0,1.25) and (0,.25) .. (0,2);
  \draw[->] (0.0,0) .. controls ++(0,1) and ++(0,-1.2) .. (0.6,2);
  \draw[->] (0.6,0) .. controls ++(0,1) and ++(0,-1.2) .. (1.2,2);
  \draw[->] (1.2,0) .. controls ++(0,1) and ++(0,-1.2) .. (1.8,2);
  \draw[->] (2.4,0) .. controls ++(0,1) and ++(0,-1.2) .. (3,2);
  \node at (1.8,.35) {$\dots$};
  \node at (2.4,1.65) {$\dots$};
  \draw[<-] (3.6,0) -- (3.6,2);
  \filldraw  (.05,1.6) circle (2pt);
\end{tikzpicture}}
 \;\;\;\right] +\hspace{6pt}4 \h{2n-1}{}$
\end{center}
as desired. The base case of the induction is proved. The induction step follows from examination of the Jacobi identity, exactly as in \cite[Lemma 26]{CLLS}, using our Lemma \ref{VirasorowithUp} in place of \cite[Lemma 24]{CLLS}.

\end{proof}

\begin{lemma}\label{VirasoroGen} Let $n \in \mathbb{Z}$. We have $$[\h{1}{x_1^2}, \h{2n-1}{}] = 2\h{2n}{x_1 + \ldots + x_{2n}} +2\h{2n}{x_2+\ldots+x_{2n-1}}.$$\end{lemma}
\begin{proof} This is a straightforward diagrammatic calculation similar to Lemmas \ref{upVirasorowithUp} and \ref{upVirasorowithDown}. We have 
$$ \left[\hspace{2mm}\begin{tikzpicture}[baseline=(current bounding box.center),scale=0.75]
  \draw[->] (3.2,0) .. controls (3.2,1.25) and (0,.25) .. (0,2)
     node[pos=0.85, shape=coordinate](X){};
  \draw[->] (0,0) .. controls (0,1) and (.8,.8) .. (.8,2);
  \draw[->] (.8,0) .. controls (.8,1) and (1.6,.8) .. (1.6,2);
  \draw[->] (2.4,0) .. controls (2.4,1) and (3.2,.8) .. (3.2,2);
  \node at (1.6,.35) {$\dots$};
  \node at (2.4,1.65) {$\dots$};
  \filldraw  (-.3,1.5) circle (3pt);
\node at (-.7,1.5) {$2n$};
\draw[->](-.3,0) to (-.3,2);
\end{tikzpicture}\hspace{2mm}\right] 
\hspace{6pt}
=
\hspace{6pt} \left[ \;\;\;
\hackcenter{
\begin{tikzpicture}[scale=0.8]
  \draw (2.4,0) .. controls (2.4,1.25) and (0,.25) .. (0,1.5)
     node[pos=0.85, shape=coordinate](X){};
  \draw (0,0) .. controls (0,1) and (0.6,.8) .. (0.6,1.5);
  \draw (0.6,0) .. controls (0.6,1) and (1.2,.8) .. (1.2,1.5);
  \draw (1.8,0) .. controls (1.8,1) and (2.4,.8) .. (2.4,1.5);
  \node at (1.2,.35) {$\dots$};
  \node at (1.8,1.65) {$\dots$};
  \draw (3,0) -- (3,1.5);
  %% DOTS
  \filldraw  (0.05,2.25) circle (2pt);
\node at (-.4,2.25) {$2n$};
 %% BLUE LINES
    \draw[blue, dotted] (-0.4,0) -- (3.4,0);
    \draw[blue, dotted] (-0.4,1.5) -- (3.4,1.5);
    %% TOP
   \draw[->]  (3,1.5) .. controls ++(0,0.95) and ++(0,-1.15) .. (0,2.5);
   \draw[->]  (2.4,1.5) .. controls ++(0,0.5) and ++(0,-0.5) .. (3,2.5);
   \draw[->]  (0,1.5) .. controls ++(0,0.5) and ++(0,-0.5) .. (0.6,2.5);
   \draw[->]  (0.6,1.5) .. controls ++(0,0.5) and ++(0,-0.5) .. (1.2,2.5);
   \draw[->]  (1.2,1.5) .. controls ++(0,0.5) and ++(0,-0.5) .. (1.8,2.5);
   %% Bottom
   \draw  (3,0) .. controls ++(0,-0.95) and ++(0,+0.95) .. (0,-1);
   \draw  (0,0) .. controls ++(0,-0.5) and ++(0,+0.5) .. (0.6,-1);
   \draw  (0.6,0) .. controls ++(0,-0.5) and ++(0,+0.5) .. (1.2,-1);
    \draw (1.8,0) .. controls ++(0,-0.5) and ++(0,+0.5) .. (2.4,-1);
   \draw (2.4,0) .. controls (2.4,-.5) and (3,-.5) .. (3,-1);
\end{tikzpicture}}
\; \;\;\right]
$$

Sliding the dots all the way to the right side of the diagram results in $2(2n-1)$ resolution terms. Each of these resolution terms contains a $2n$-cycle and a single solid dot - there are 2 resolution terms containing a solid dot on the first strand and 2 containing a solid dot on the last strand, and 4 resolution terms with a dot on each other strand. All empty dots cancel in such a way that no resolution terms cancel with each other. The result follows.
\end{proof}
%\begin{lemma}\label{VirasoroRels} Suppose $m\not= -n$. We have the Virasoro-type relations:
%$$[\h{2n}{(x_1 + x_2+\ldots + x_{2n})}, \h{2m}{  (x_1+x_2 + \ldots + x_{2m})}] = (k-\ell) \h{2(m+n)}{ (x_1 + \ldots + x_{k+2m})}.$$
%\end{lemma} 
%\begin{proof}  \end{proof}

The following lemmas will allow us to generate bubbles with arbitrary numbers of dots using just $\h{\pm 1}{x_1^2}$.
\begin{lemma}\label{figure 8 split}
We have $$\ds \sum_{a+b=2n-1}
\begin{tikzpicture}[baseline=(current bounding box.center),rotate=90]
\raisebox{3mm}{
 \draw[->] (-0.6,0) .. controls ++(0,.4) and ++(0,-.5) .. (-0.0,1);
  \draw[<-] (0.0,0) .. controls ++(0,.4) and ++(0,-.5) .. (-0.6,1);
   \draw[<-] (-0.6,1) .. controls ++(0,.4) and ++(0, .4) .. (-0.0,1);
  \draw[->] (0.0,0) .. controls ++(0,-.5) and ++(0,-.5) .. (-0.6,0);
  \filldraw  (-0.05,1.2) circle (2pt);
  \filldraw   (-0.1,-.3) circle (2pt);
  \node at (0.1,1.3) {\small $a$};
  \node at (0.1,-.4) {\small $b$};}
\end{tikzpicture}\hspace{6pt}
=
\hspace{6pt}
\sum_{i+j=n-1}(1+2j)\hspace{1mm}
\begin{tikzpicture}[baseline=(current bounding box.center)]
    \draw[->] (3,2) arc (-180:180:5mm);
\fill (3.95,2.2) circle [radius=2pt];
\node at (4.2,2.2) {\small $2i$};
    \end{tikzpicture}
\begin{tikzpicture}[baseline=(current bounding box.center)]
    \draw[<-] (3,2) arc (-180:180:5mm);
\fill (3.95,2.2) circle [radius=2pt];
\node at (4.2,2.2) {\small $2j$};
    \end{tikzpicture}$$
\end{lemma}

\begin{proof}
We compute:
$$
\ds \sum_{a+b=2n-1}
\begin{tikzpicture}[baseline=(current bounding box.center),rotate=90]
\raisebox{3mm}{
 \draw[->] (-0.6,0) .. controls ++(0,.4) and ++(0,-.5) .. (-0.0,1);
  \draw[<-] (0.0,0) .. controls ++(0,.4) and ++(0,-.5) .. (-0.6,1);
   \draw[<-] (-0.6,1) .. controls ++(0,.4) and ++(0, .4) .. (-0.0,1);
  \draw[->] (0.0,0) .. controls ++(0,-.5) and ++(0,-.5) .. (-0.6,0);
  \filldraw  (-0.05,1.2) circle (2pt);
  \filldraw   (-0.1,-.3) circle (2pt);
  \node at (0.1,1.3) {\small $a$};
  \node at (0.1,-.4) {\small $b$};}
\end{tikzpicture}\hspace{6pt}
$$ $$=
\hspace{4pt}
\begin{tikzpicture}[baseline=(current bounding box.center),rotate=90]
\raisebox{3mm}{
 \draw[->] (-0.6,0) .. controls ++(0,.4) and ++(0,-.5) .. (-0.0,1);
  \draw[<-] (0.0,0) .. controls ++(0,.4) and ++(0,-.5) .. (-0.6,1);
   \draw[<-] (-0.6,1) .. controls ++(0,.4) and ++(0, .4) .. (-0.0,1);
  \draw[->] (0.0,0) .. controls ++(0,-.5) and ++(0,-.5) .. (-0.6,0);
  \filldraw  (-0.05,1.2) circle (2pt);
  \node at (0.2,1.3) {\small $2n$-$1$};}
\end{tikzpicture}\hspace{4pt}
+
\hspace{4pt}
\begin{tikzpicture}[baseline=(current bounding box.center),rotate=90]
\raisebox{3mm}{
 \draw[->] (-0.6,0) .. controls ++(0,.4) and ++(0,-.5) .. (-0.0,1);
  \draw[<-] (0.0,0) .. controls ++(0,.4) and ++(0,-.5) .. (-0.6,1);
   \draw[<-] (-0.6,1) .. controls ++(0,.4) and ++(0, .4) .. (-0.0,1);
  \draw[->] (0.0,0) .. controls ++(0,-.5) and ++(0,-.5) .. (-0.6,0);
  \filldraw  (-0.05,1.2) circle (2pt);
  \filldraw   (-0.1,-.3) circle (2pt);
  \node at (0.2,1.3) {\small $2n$-$2$};}
\end{tikzpicture}\hspace{4pt}
+
\hspace{4pt}
\begin{tikzpicture}[baseline=(current bounding box.center),rotate=90]
\raisebox{3mm}{
 \draw[->] (-0.6,0) .. controls ++(0,.4) and ++(0,-.5) .. (-0.0,1);
  \draw[<-] (0.0,0) .. controls ++(0,.4) and ++(0,-.5) .. (-0.6,1);
   \draw[<-] (-0.6,1) .. controls ++(0,.4) and ++(0, .4) .. (-0.0,1);
  \draw[->] (0.0,0) .. controls ++(0,-.5) and ++(0,-.5) .. (-0.6,0);
  \filldraw  (-0.05,1.2) circle (2pt);
  \filldraw   (-0.1,-.3) circle (2pt);
  \node at (0.2,1.3) {\small $2n$-$3$};
  \node at (0.2,-.4) {\small $2$};}
\end{tikzpicture}\hspace{4pt}
+
\hspace{4pt}
\dots
\hspace{4pt}
+
\hspace{4pt}
\begin{tikzpicture}[baseline=(current bounding box.center),rotate=90]
\raisebox{3mm}{
 \draw[->] (-0.6,0) .. controls ++(0,.4) and ++(0,-.5) .. (-0.0,1);
  \draw[<-] (0.0,0) .. controls ++(0,.4) and ++(0,-.5) .. (-0.6,1);
   \draw[<-] (-0.6,1) .. controls ++(0,.4) and ++(0, .4) .. (-0.0,1);
  \draw[->] (0.0,0) .. controls ++(0,-.5) and ++(0,-.5) .. (-0.6,0);
  \filldraw  (-0.05,1.2) circle (2pt);
  \filldraw   (-0.1,-.3) circle (2pt);
  \node at (0.2,-.4) {\small $2n$-$2$};}
\end{tikzpicture}
$$ $$=
\hspace{4pt}
\begin{tikzpicture}[baseline=(current bounding box.center),rotate=90]
\raisebox{3mm}{
 \draw[->] (-0.6,0) .. controls ++(0,.4) and ++(0,-.5) .. (-0.0,1);
  \draw[<-] (0.0,0) .. controls ++(0,.4) and ++(0,-.5) .. (-0.6,1);
   \draw[<-] (-0.6,1) .. controls ++(0,.4) and ++(0, .4) .. (-0.0,1);
  \draw[->] (0.0,0) .. controls ++(0,-.5) and ++(0,-.5) .. (-0.6,0);
  \filldraw  (-0.05,1.2) circle (2pt);
  \node at (0.2,1.3) {\small $2n$-$1$};}
\end{tikzpicture}\hspace{4pt}
+
\hspace{4pt}2
\begin{tikzpicture}[baseline=(current bounding box.center),rotate=90]
\raisebox{3mm}{
 \draw[->] (-0.6,0) .. controls ++(0,.4) and ++(0,-.5) .. (-0.0,1);
  \draw[<-] (0.0,0) .. controls ++(0,.4) and ++(0,-.5) .. (-0.6,1);
   \draw[<-] (-0.6,1) .. controls ++(0,.4) and ++(0, .4) .. (-0.0,1);
  \draw[->] (0.0,0) .. controls ++(0,-.5) and ++(0,-.5) .. (-0.6,0);
  \filldraw  (-0.05,1.2) circle (2pt);
  \filldraw   (-0.1,-.3) circle (2pt);
  \node at (0.2,1.3) {\small $2n$-$3$};
  \node at (0.2,-.4) {\small $2$};}
\end{tikzpicture}\hspace{4pt}
+
\hspace{4pt}2
\begin{tikzpicture}[baseline=(current bounding box.center),rotate=90]
\raisebox{3mm}{
 \draw[->] (-0.6,0) .. controls ++(0,.4) and ++(0,-.5) .. (-0.0,1);
  \draw[<-] (0.0,0) .. controls ++(0,.4) and ++(0,-.5) .. (-0.6,1);
   \draw[<-] (-0.6,1) .. controls ++(0,.4) and ++(0, .4) .. (-0.0,1);
  \draw[->] (0.0,0) .. controls ++(0,-.5) and ++(0,-.5) .. (-0.6,0);
  \filldraw  (-0.05,1.2) circle (2pt);
  \filldraw   (-0.1,-.3) circle (2pt);
  \node at (0.2,1.3) {\small $2n$-$5$};
  \node at (0.2,-.4) {\small $4$};}
\end{tikzpicture}\hspace{4pt}
+
\hspace{4pt}
\dots
\hspace{4pt}
+
\hspace{4pt}2
\begin{tikzpicture}[baseline=(current bounding box.center),rotate=90]
\raisebox{3mm}{
 \draw[->] (-0.6,0) .. controls ++(0,.4) and ++(0,-.5) .. (-0.0,1);
  \draw[<-] (0.0,0) .. controls ++(0,.4) and ++(0,-.5) .. (-0.6,1);
   \draw[<-] (-0.6,1) .. controls ++(0,.4) and ++(0, .4) .. (-0.0,1);
  \draw[->] (0.0,0) .. controls ++(0,-.5) and ++(0,-.5) .. (-0.6,0);
  \filldraw  (-0.05,1.2) circle (2pt);
  \filldraw   (-0.1,-.3) circle (2pt);
  \node at (0.2,-.4) {\small $2n$-$2$};}
\end{tikzpicture}
$$
because we have \begin{equation*}
\begin{tikzpicture}[baseline=(current bounding box.center),rotate=90]
\raisebox{3mm}{
 \draw[->] (-0.6,0) .. controls ++(0,.4) and ++(0,-.5) .. (-0.0,1);
  \draw[<-] (0.0,0) .. controls ++(0,.4) and ++(0,-.5) .. (-0.6,1);
   \draw[<-] (-0.6,1) .. controls ++(0,.4) and ++(0, .4) .. (-0.0,1);
  \draw[->] (0.0,0) .. controls ++(0,-.5) and ++(0,-.5) .. (-0.6,0);
  \filldraw  (-0.05,1.2) circle (2pt);
  \filldraw   (-0.1,-.3) circle (2pt);
  \node at (0.2,1.3) {\small $2n$-$2i$};
  \node at (0.2,-.4) {\small $2i$-$1$};}
\end{tikzpicture}\hspace{6pt}
=
\hspace{6pt}
\begin{tikzpicture}[baseline=(current bounding box.center),rotate=90]
\raisebox{3mm}{
 \draw[->] (-0.6,0) .. controls ++(0,.4) and ++(0,-.5) .. (-0.0,1);
  \draw[<-] (0.0,0) .. controls ++(0,.4) and ++(0,-.5) .. (-0.6,1);
   \draw[<-] (-0.6,1) .. controls ++(0,.4) and ++(0, .4) .. (-0.0,1);
  \draw[->] (0.0,0) .. controls ++(0,-.5) and ++(0,-.5) .. (-0.6,0);
  \filldraw  (-0.05,1.2) circle (2pt);
  \filldraw   (-0.1,-.3) circle (2pt);
  \node at (0.2,1.3) {\small $2n$-$2i$-$1$};
  \node at (0.2,-.4) {\small $2i$};}
\end{tikzpicture}.
\end{equation*}
Moreover, we can decompose these figure eights into a linear combination of products of two bubbles using dot slide relations $\ref{dotSlide: bottomLeft}$ and $\ref{dotSlide: topLeft}$ as follows:
\begin{equation*}
\begin{tikzpicture}[baseline=(current bounding box.center),rotate=90]
\raisebox{3mm}{
 \draw[->] (-0.6,0) .. controls ++(0,.4) and ++(0,-.5) .. (-0.0,1);
  \draw[<-] (0.0,0) .. controls ++(0,.4) and ++(0,-.5) .. (-0.6,1);
   \draw[<-] (-0.6,1) .. controls ++(0,.4) and ++(0, .4) .. (-0.0,1);
  \draw[->] (0.0,0) .. controls ++(0,-.5) and ++(0,-.5) .. (-0.6,0);
  \filldraw  (-0.05,1.2) circle (2pt);
  \filldraw   (-0.1,-.3) circle (2pt);
  \node at (0.2,1.3) {\small $2n$-$2a$-$1$};
  \node at (0.2,-.4) {\small $2a$};}
\end{tikzpicture}\hspace{6pt}
=
\hspace{6pt}
\ds \sum_{\substack{i+j=n-1\\j\geq a}}\begin{tikzpicture}[baseline=(current bounding box.center)]
    \draw[->] (3,2) arc (-180:180:5mm);
\fill (3.95,2.2) circle [radius=2pt];
\node at (4.2,2.2) {\small $2i$};
    \end{tikzpicture}
\begin{tikzpicture}[baseline=(current bounding box.center)]
    \draw[<-] (3,2) arc (-180:180:5mm);
\fill (3.95,2.2) circle [radius=2pt];
\node at (4.2,2.2) {\small $2j$};
    \end{tikzpicture}.
\end{equation*}
Combining these results, we get that 
\begin{equation*}
\ds \sum_{a+b=2n-1}
\begin{tikzpicture}[baseline=(current bounding box.center),rotate=90]
\raisebox{3mm}{
 \draw[->] (-0.6,0) .. controls ++(0,.4) and ++(0,-.5) .. (-0.0,1);
  \draw[<-] (0.0,0) .. controls ++(0,.4) and ++(0,-.5) .. (-0.6,1);
   \draw[<-] (-0.6,1) .. controls ++(0,.4) and ++(0, .4) .. (-0.0,1);
  \draw[->] (0.0,0) .. controls ++(0,-.5) and ++(0,-.5) .. (-0.6,0);
  \filldraw  (-0.05,1.2) circle (2pt);
  \filldraw   (-0.1,-.3) circle (2pt);
\node at (0.1,1.3) {\small $a$};
  \node at (0.1,-.4) {\small $b$};}
\end{tikzpicture}\hspace{6pt}
=
\hspace{6pt}
\sum_{i+j=n-1}(1+2j)\hspace{1mm}
\begin{tikzpicture}[baseline=(current bounding box.center)]
    \draw[->] (3,2) arc (-180:180:5mm);
\fill (3.95,2.2) circle [radius=2pt];
\node at (4.2,2.2) {\small $2i$};
    \end{tikzpicture}
\begin{tikzpicture}[baseline=(current bounding box.center)]
    \draw[<-] (3,2) arc (-180:180:5mm);
\fill (3.95,2.2) circle [radius=2pt];
\node at (4.2,2.2) {\small $2j$};
    \end{tikzpicture}.
\end{equation*}
\end{proof}

\begin{lemma}\label{updots with downdots}
 We have $$\ds [\h{1}{x_1^{2a}},\h{-1}{x_1^{2b}}]=-2\bar{d}_{2(a+b)}-\sum_{i+j=2(a+b)-1}(2+4j)\bar{d}_{2i}d_{2j}$$ for $a,b\in \mathbb{Z}_{\geq0}.$
\end{lemma}

\begin{proof}
We compute: \begin{align*}
\ds
\begin{tikzpicture}[baseline=(current bounding box.center),scale=0.75]
\draw[->] (0,0) to (0,2);
\draw[<-] (0.75,0) to (0.75,2);
\fill (0,1.2) circle (2pt);
\fill (0.75,1.2) circle (2pt);
\node at (-0.3,1.3) {\small $2a$};
\node at (1.1,1.3) {\small $2b$};
\end{tikzpicture}\hspace{6pt}
&=
\hspace{6pt}
\begin{tikzpicture}[baseline=(current bounding box.center),scale=0.75]
\draw[->] (-0.6,0) .. controls ++(0,.4) and ++(0,-.5) .. (-0.0,1);
\draw[<-] (0.0,0) .. controls ++(0,.4) and ++(0,-.5) .. (-0.6,1);
\draw[<-] (-0.6,1) .. controls ++(0,.4) and ++(0,-.5) .. (-0.0,2);
\draw[->] (0.0,1) .. controls ++(0,.4) and ++(0,-.5) .. (-0.6,2);
\fill  (-0.55,1.7) circle (2pt);
\fill  (-0.1,1.7) circle (2pt);
\node at (-.95,1.8) {\small $2a$};
\node at (0.35,1.8) {\small $2b$};
\end{tikzpicture}
=\hspace{6pt}
\begin{tikzpicture}[baseline=(current bounding box.center),scale=0.75]
 \draw[->] (-0.6,0) .. controls ++(0,.4) and ++(0,-.5) .. (-0.0,1);
  \draw[<-] (0.0,0) .. controls ++(0,.4) and ++(0,-.5) .. (-0.6,1);
   \draw[<-] (-0.6,1) .. controls ++(0,.4) and ++(0,-.5) .. (-0.0,2);
  \draw[->] (0.0,1) .. controls ++(0,.4) and ++(0,-.5) .. (-0.6,2);
  \filldraw  (-0.05,1.2) circle (2pt);
  \filldraw   (-0.1,1.7) circle (2pt);
  \node at (0.3,1.3) {\small $2a$};
  \node at (0.35,1.8) {\small $2b$};
\end{tikzpicture}\hspace{6pt}
-\hspace{6pt}2
\sum_{j=0}^{2a-1}\hspace{2mm}
\begin{tikzpicture}[baseline=(current bounding box.center),scale=0.75]
 \draw[->] (-0.6,0) .. controls ++(0,.4) and ++(0,-.5) .. (-0.0,1);
  \draw[<-] (0.0,0) .. controls ++(0,.4) and ++(0,-.5) .. (-0.6,1);
   \draw[<-] (-0.6,1) .. controls ++(0,.4) and ++(0, .4) .. (-0.0,1);
  \draw[->] (0.0,2) .. controls ++(0,-.5) and ++(0,-.5) .. (-0.6,2);
  \filldraw  (-0.05,1.2) circle (2pt);
  \filldraw   (0,1.8) circle (2pt);
  \node at (0.2,1.2) {\small $j$};
  \node at (1.2,1.8) {\small $2(a$+$b)$-$1$-$j$};
\end{tikzpicture}
\\&=
\hspace{6pt}
\begin{tikzpicture}[baseline=(current bounding box.center),scale=0.75]
 \draw[->] (-0.6,0) .. controls ++(0,.4) and ++(0,-.5) .. (-0.0,1);
  \draw[<-] (0.0,0) .. controls ++(0,.4) and ++(0,-.5) .. (-0.6,1);
   \draw[<-] (-0.6,1) .. controls ++(0,.4) and ++(0,-.5) .. (-0.0,2);
  \draw[->] (0.0,1) .. controls ++(0,.4) and ++(0,-.5) .. (-0.6,2);
  \filldraw  (-0.05,1.2) circle (2pt);
  \filldraw   (-0.55,1.2) circle (2pt);
  \node at (0.3,1.3) {\small $2a$};
  \node at (-0.9,1.3) {\small $2b$};
  \end{tikzpicture}\hspace{6pt}
-
\hspace{6pt}2
\sum_{i=0}^{2b-1}\hspace{2mm}
\begin{tikzpicture}[baseline=(current bounding box.center),scale=0.75]
 \draw[->] (-0.6,0) .. controls ++(0,.4) and ++(0,-.5) .. (-0.0,1);
  \draw[<-] (0.0,0) .. controls ++(0,.4) and ++(0,-.5) .. (-0.6,1);
   \draw[<-] (-0.6,1) .. controls ++(0,.4) and ++(0, .4) .. (-0.0,1);
  \draw[->] (0.0,2) .. controls ++(0,-.5) and ++(0,-.5) .. (-0.6,2);
  \filldraw  (-0.05,1.2) circle (2pt);
  \filldraw   (0,1.8) circle (2pt);
  \node at (0.66,1.2) {\small $2a$+$i$};
  \node at (0.66,1.8) {\small $2b$-$1$-$i$};
\end{tikzpicture}\hspace{6pt}
-
\hspace{6pt}2
\sum_{j=0}^{2a-1}\hspace{2mm}
\begin{tikzpicture}[baseline=(current bounding box.center),scale=0.75]
 \draw[->] (-0.6,0) .. controls ++(0,.4) and ++(0,-.5) .. (-0.0,1);
  \draw[<-] (0.0,0) .. controls ++(0,.4) and ++(0,-.5) .. (-0.6,1);
   \draw[<-] (-0.6,1) .. controls ++(0,.4) and ++(0, .4) .. (-0.0,1);
  \draw[->] (0.0,2) .. controls ++(0,-.5) and ++(0,-.5) .. (-0.6,2);
  \filldraw  (-0.05,1.2) circle (2pt);
  \filldraw   (0,1.8) circle (2pt);
  \node at (0.2,1.2) {\small $j$};
  \node at (1.2,1.8) {\small $2(a$+$b)$-$1$-$j$};
\end{tikzpicture}
\\&=
\hspace{6pt}
\begin{tikzpicture}[baseline=(current bounding box.center),scale=0.75]
\draw[<-] (-0.6,0) .. controls ++(0,.4) and ++(0,-.5) .. (-0.0,1);
\draw[->] (0.0,0) .. controls ++(0,.4) and ++(0,-.5) .. (-0.6,1);
\draw[->] (-0.6,1) .. controls ++(0,.4) and ++(0,-.5) .. (-0.0,2);
\draw[<-] (0.0,1) .. controls ++(0,.4) and ++(0,-.5) .. (-0.6,2);
\fill  (-0.55,1.7) circle (2pt);
\fill  (-0.1,1.7) circle (2pt);
\node at (-.95,1.8) {\small $2b$};
\node at (0.35,1.8) {\small $2a$};
\end{tikzpicture}\hspace{6pt}
-
\hspace{6pt}2
\sum_{j=0}^{2(a+b)-1}\hspace{2mm}
\begin{tikzpicture}[baseline=(current bounding box.center),scale=0.75]
 \draw[->] (-0.6,0) .. controls ++(0,.4) and ++(0,-.5) .. (-0.0,1);
  \draw[<-] (0.0,0) .. controls ++(0,.4) and ++(0,-.5) .. (-0.6,1);
   \draw[<-] (-0.6,1) .. controls ++(0,.4) and ++(0, .4) .. (-0.0,1);
  \draw[->] (0.0,2) .. controls ++(0,-.5) and ++(0,-.5) .. (-0.6,2);
  \filldraw  (-0.05,1.2) circle (2pt);
  \filldraw   (0,1.8) circle (2pt);
  \node at (0.2,1.2) {\small $j$};
  \node at (1.2,1.8) {\small $2(a$+$b)$-$1$-$j$};
  \end{tikzpicture}
\\&=
\hspace{6pt}
\begin{tikzpicture}[baseline=(current bounding box.center),scale=0.75]
\draw[<-] (-0.6,0) .. controls ++(0,.4) and ++(0,-.5) .. (-0.0,1);
\draw[->] (0.0,0) .. controls ++(0,.4) and ++(0,-.5) .. (-0.6,1);
\draw[->] (-0.6,1) .. controls ++(0,.4) and ++(0,-.5) .. (-0.0,2);
\draw[<-] (0.0,1) .. controls ++(0,.4) and ++(0,-.5) .. (-0.6,2);
\fill  (-0.55,1.7) circle (2pt);
\fill  (-0.1,1.7) circle (2pt);
\node at (-.95,1.8) {\small $2b$};
\node at (0.35,1.8) {\small $2a$};
\end{tikzpicture}\hspace{6pt}
-
\hspace{6pt}2
\sum_{j=0}^{2(a+b)-1}\hspace{2mm}
\begin{tikzpicture}[baseline=(current bounding box.center),scale=0.75]
 \draw[->] (-0.6,0) .. controls ++(0,.4) and ++(0,-.5) .. (-0.0,1);
  \draw[<-] (0.0,0) .. controls ++(0,.4) and ++(0,-.5) .. (-0.6,1);
   \draw[<-] (-0.6,1) .. controls ++(0,.4) and ++(0, .4) .. (-0.0,1);
  \draw[->] (0.0,0) .. controls ++(0,-.5) and ++(0,-.5) .. (-0.6,0);
  \filldraw  (0,1.2) circle (2pt);
  \filldraw   (-0.1,-.3) circle (2pt);
  \node at (0.3,1.3) {$j$};
  \node at (1.2,-.2) {$2(a$+$b)$-$1$-$j$};
\end{tikzpicture}
\\&=
\begin{tikzpicture}[baseline=(current bounding box.center),scale=0.75]
\draw[<-] (0,0) to (0,2);
\draw[->] (0.75,0) to (0.75,2);
\fill (0,1.2) circle (2pt);
\fill (0.75,1.2) circle (2pt);
\node at (-0.3,1.3) {\small $2b$};
\node at (1.1,1.3) {\small $2a$};
\end{tikzpicture}\hspace{6pt}
-
\hspace{6pt}2\hspace{2mm}
\begin{tikzpicture}[baseline=(current bounding box.center)]
    \draw[->] (3,2) arc (-180:180:5mm);
\fill (3.95,2.2) circle [radius=2pt];
\node at (4.5,2.2) {\small $2(a$+$b)$};
    \end{tikzpicture}\hspace{6pt}
-
\hspace{6pt}2
\sum_{j=0}^{2(a+b)-1}\hspace{2mm}
\begin{tikzpicture}[baseline=(current bounding box.center),scale=0.75]
 \draw[->] (-0.6,0) .. controls ++(0,.4) and ++(0,-.5) .. (-0.0,1);
  \draw[<-] (0.0,0) .. controls ++(0,.4) and ++(0,-.5) .. (-0.6,1);
   \draw[<-] (-0.6,1) .. controls ++(0,.4) and ++(0, .4) .. (-0.0,1);
  \draw[->] (0.0,0) .. controls ++(0,-.5) and ++(0,-.5) .. (-0.6,0);
  \filldraw  (0,1.2) circle (2pt);
  \filldraw   (-0.1,-.3) circle (2pt);
  \node at (0.3,1.3) {$j$};
  \node at (1.2,-.2) {$2(a$+$b)$-$1$-$j$};
\end{tikzpicture}.
\end{align*}

Therefore $\ds [\h{1}{x_1^{2a}},\h{-1}{x_1^{2b}}]=-2\bar{d}_{2(a+b)}-\sum_{i+j=2(a+b)-1}(2+4j)\bar{d}_{2i}d_{2j}$.
\end{proof}

%---------TRACE AS A VECTOR SPACE--------%

\section{Algebra isomorphism}
In this section, we will study the structure of $\TrHEv$, first as a vector space and then as an algebra. We show that $\TrHEv$ has a triangular decomposition into two copies of the trace of $\HC_n$ and a polynomial algebra. We then describe a generating set for $\TrHEv$, which allows us to define the algebra homomorphism to $W^-$. Finally, we prove that this homomorphism is an isomorphism.

\subsection{Trace of $\mathcal{H}_{tw}$ as a vector space}
Let $m,n \geq 0$ and define $J_{m,n}$ to be the 2-sided ideal in $\End_{\mathcal{H}_{tw}}(P^mQ^n)$ generated by diagrams which contain at least one arc connecting a pair of upper points. 
\begin{lemma}\label{ses}
There exists a split short exact sequence 
\begin{equation*}
0 \rightarrow J_{m,n} \rightarrow \End_{\mathcal{H}_{tw}}(P^mQ^n) \rightarrow (\HC)^{op}_m \otimes \HC_n \otimes \mathbb{C}[d_0,d_2,d_4....] \rightarrow 0.
\end{equation*}
\end{lemma}

\begin{proof}
In $\End_{\mathcal{H}_{tw}}(P^mQ^n)$, due to the middle diagram in relation \eqref{R2}, we can assume our diagrams have no crossing between opposite oriented strands. Taking the quotient $\End_{\mathcal{H}_{tw}}(P^mQ^n)/J_{m,n}$ kills diagrams with cups connecting two upper points, and those with caps connecting two lower points. Therefore we are left with diagrams, possibly with bubbles, which have no caps or cups and have crossings only among like-oriented strands.
   Note that in the quotient $\End_{\mathcal{H}_{tw}}(P^mQ^n)/J_{m,n}$, the diagram in relation \eqref{R3} simplifies to 
\begin{equation*}
\begin{tikzpicture}[baseline=(current bounding box.center),scale=0.75]
\draw [<-](0,0) to [out=45,in=-45] (0,2);
\draw [->](0.5,0) to [out=135,in=-135] (0.5,2);
\end{tikzpicture}\hspace{6pt}
=
\hspace{6pt}
\begin{tikzpicture}[baseline=(current bounding box.center),scale=0.75]
\draw[<-] (1.5,0) to (1.5,2);
\draw[->] (2,0) to (2,2);
\end{tikzpicture}
\end{equation*}
and therefore we can move the bubbles to the rightmost part of our diagrams for free. This gives us a short exact sequence 
\begin{equation*}
0 \rightarrow J_{m,n} \rightarrow \End_{\mathcal{H}_{tw}}(P^mQ^n) \rightarrow \End_{\mathcal{H}_{tw}}(P^m) \otimes \End_{\mathcal{H}_{tw}}(Q^n) \otimes  \End_{\mathcal{H}_{tw}}(1) \rightarrow 0.
\end{equation*}
By \cite[Proposition 7.1]{CS}, we have that $\End_{\mathcal{H}_{tw}}(P^m)$ is isomorphic to $(\HC)^{op}_m$ and that $\End_{\mathcal{H}_{tw}}(Q^n)$ is isomoprhic to $\HC_n$. By Proposition \ref{bubbles}, it follows that $\End_{\mathcal{H}_{tw}}(1)$ is isomorphic to $\mathbb{C}[d_0,d_2,d_4....]$. Hence the result follows.
\end{proof}

\begin{lemma}\label{fg}
If $f,g\in \HC_n$ such that $fg=1$,  then $f,g\in \Cl \rtimes \mathbb{C}[S_n]\subset \HC_n$.
\end{lemma}
\begin{proof}
There is an $\mathbb{N}$-filtration on $\HC_n$ given by $\deg(x_i)=1$ for $i\in \{1,...,n\}$ and other generators have degree zero. Under this filtration, the degree zero part of $\HC_n$ is the semidirect product $Cl_n\rtimes \mathbb{C}[S_n]$. Therefore, in the associated graded object, we see that if $fg=1$, $\deg(gr(f)gr(g))=\deg(gr(f))+\deg(gr(g))=\deg(1)=0$, hence $gr(f),gr(g)$ are in degree zero part. Therefore $f,g\in \Cl \rtimes \mathbb{C}[S_n]$.
\end{proof}

\begin{lemma}\label{indec objects}
The indecomposable objects of $\mathcal{H}_{tw}$ are of the form $P^mQ^n$ for $m,n\in \mathbb{Z}_{\geq 0}$.
\end{lemma}
\begin{proof}
First, note that if $QP$ appears in an object, that object can be decomposed into more components using the diagram in relation \eqref{R3}. Hence all indecomposable objects must be of the form $P^mQ^n$. 

On the other hand, to see that every sequence of the form $P^mQ^n$ is an indecomposable object, we will show that any idempotent in $\End(P^mQ^n)$ has to be the identity.

Let $f,g$ be two maps as mentioned in Lemma \ref{fg}. Note that $gf$ is an idempotent since $(gf)(gf)=g(fg)f=gf$. Since we had the splitting short exact sequence $0 \rightarrow J_{m,n} \rightarrow \End_{\mathcal{H}_{tw}}(P^mQ^n) \rightarrow \End(P^m) \otimes \End (Q^n) \otimes \End(id) \rightarrow 0$ in Lemma \ref{ses}, we know that the maps $f$ and $g$ will decompose into $(f_1,f_2)$ and $(g_1,g_2)$ where $f_1,g_1:P^m\rightarrow P^m$ and $(f_2,g_2):Q^n\rightarrow Q^n$. Now $g_1f_1$ is the identity map in $\End(P^m)$, and by the above lemma $g_1,f_1\in \mathcal{C}\ell_n\rtimes \mathbb{C}[S_n]$. Similarly, $f_2, g_2\in \mathcal{C}\ell_n\rtimes \mathbb{C}[S_n]$.

 But in $\mathcal{C}\ell_n\rtimes \mathbb{C}[S_n]$, $g_1f_1=1$ implies that $f_1g_1=1$ as well. To see this, consider the diagrams corresponding to $g_1$ and $f_1$ which consist of a permutation and some hollow dots on top. After composing these diagrams, we can collect all the hollow dots on the top since hollow dots can pass through crossing for free, possibly gaining a sign. Furthermore, each strand has an even number of hollow dots, since this composition is the identity map. So, the hollow dots cancel with each other. This shows that the corresponding permutations of $f_1$ and $g_1$ are inverses of each other, and in particular they commute. Therefore $f_1g_1=1$. Similarly, $f_2g_2=1$. Thus we have that $fg=1$.
\end{proof}

\begin{proposition}\label{triangularDecomposition}
We have the triangular decomposition of $\TrH$:
\begin{equation*}
\TrH \cong \bigoplus_{m,n\in \mathbb{Z}_{\geq0}} \Tr((\HC_m)^{op} \otimes \HC_n \otimes \mathbb{C}[d_0,d_2,d_4....]).
\end{equation*}
\end{proposition}

\begin{proof}
As shown in \cite{BGHL}, to find $\TrH$, it is enough to consider the direct sum over indecomposable objects of endomorphism spaces of objects of $H_{tw}$. Let $I=\operatorname{span}_\mathbb{C}\{fg-gf\}$ where $f:x\rightarrow y$ and $g:y\rightarrow x$ for $x,y$ objects of a $\mathbb{C}$-linear category. Therefore by Lemma \ref{indec objects} we have
\begin{equation*}
    \TrH\cong \bigg(\bigoplus_{m,n\in \mathbb{Z}_{\geq0}} \End_{\Heis}(P^mQ^n)\bigg)\big/I.
\end{equation*}

By Lemma \ref{ses}, this gives us 
\begin{equation*}
    \TrH\cong \bigg(\bigoplus_{m,n\in \mathbb{Z}_{\geq0}} ((\HC_m)^{op} \otimes \HC_n \otimes \mathbb{C}[d_0,d_2,d_4....])\oplus J_{m,n}\bigg)\big/I.
\end{equation*}
Recall that the ideal $J_{m,n}$ is generated by diagrams containing at least one cup connecting two upper points. Therefore, the diagrams in $J_{m,n}$ must also contain caps, since they are dealing with endomorphisms. Using the trace relation and the relations in $\mathcal{H}_{tw}$, we can express the elements of $J_{m,n}$ as direct sum of endomorphisms of $P^{m'}Q^{n'}$ for $m'\leq m$ and $n'\leq n$. Hence we have 
\begin{align*}
    \TrH &\cong \bigoplus_{m,n\in \mathbb{Z}_{\geq0}} \Tr((\HC_m)^{op} \otimes \HC_n \otimes \mathbb{C}[d_0,d_2,d_4....])
    \\&\cong \bigg(\bigoplus_{m,n\in \mathbb{Z}_{\geq0}} \Tr((\HC_m)^{op} \otimes \HC_n)\bigg) \otimes \mathbb{C}[d_0,d_2,d_4....].
\end{align*}

\end{proof}

%-----------TRACE AS AN ALGEBRA------%

\subsection{Generators of the algebra $\TrHEv$}

The following gives a generating set for $\TrHEv$ as an algebra.

\begin{lemma} \label{GenSetTr} The algebra $\TrHEv$ is generated by $\h{-1}{}$, $\h{\pm 2}{(x_1 + x_2)}$, and $d_0 + d_2$. \end{lemma}
\begin{proof} 

First, Proposition \ref{clockwise bubbles with down explicit} implies that $\h{1}{}$ and $(d_0+d_2)$ allow us generate a differential degree two element $\h{1}{x_1^2}$; since all relations in $\mathcal{H}_{tw}$ are local, we can evaluate the commutator $[\h{1}{x_1^2}, (d_0+d_2)]$ by moving the dot to the bottom of the upward strand and sliding the bubbles over the upper portion. We can therefore apply Lemma \ref{clockwise bubbles with down explicit} repeatedly to show that $\operatorname{ad}(d_0+d_2)^n\h{1}{}$ has a leading term of $\h{1}{x_1^{2n}}$.

By Lemma \ref{upVirasorowithUp}, the elements $\h{-1}{}$ and $\h{2}{x_1+x_2}$ are sufficient to generate $\h{2m+1}{}$ for all integers $m>0$. Then we can generate $\h{2n}{x_1+\ldots+x_n}$ from $\h{1}{x_1^2}$ and $\h{2m+1}{}$ by using Lemma \ref{VirasoroGen}.  Lemma \ref{upVirasorowithDown}, $\h{-1}{}$ and $\h{2n}{x_1+x_2+\ldots + x_n}$ allow us to generate $\h{2r+1}{}$ for all integers $r$.

 Proposition \ref{red to hn} implies that all elements with nonzero rank degree can be written as a sum of elements of the form $\h{ \pm n}{x_1^{\ell} c_1^{k}}$. By Propositions \ref{evenCyclesZero} and \ref{odd with odd dots zero}, all elements of this form except for the ones generated in the preceding paragraphs are 0 in $\TrHEv$, so we have generated all of $\TrHPl$ and $\TrHMi$.

Finally, Lemma \ref{updots with downdots} allows us to generate $d_{2n}$, applying Lemma \ref{bubbleDecomp} to split up the $\overline{d}_{2n}$ terms. \end{proof}

\subsection{The isomorphism}

There is an obvious isomorphism of vector spaces between the Fock space representations of $\TrHEv$ and $W^-$: 

\begin{equation*}\label{VSIsom}
\phi: V= \mathbb{C}[\h{1}{}, \h{2}{}, \ldots] \rightarrow \mathbb{C}[w_{-1,0}, w_{-2,0}, \ldots] = \mathcal{V}_{1,0}.
\end{equation*}

Recall that each algebra acts faithfully on its Fock space representation.

\begin{lemma}\label{Heis commutes} The map $\phi$ in Equation \eqref{VSIsom} commutes with the action of the twisted Heisenberg subalgebras in $V$ and $\mathcal{V}_{1,0}$, i.e.:
$$\phi(\h{r}{}v) =\sqrt{2}w_{-r,0}\phi(v).$$ \end{lemma}
\begin{proof}
The vector space realizations of $V$ and $\mathcal{V}_{1,0}$ in Equation \eqref{VSIsom} imply that the action of $\h{r}{}$ on $V$ is simply the adjoint action of $\h{r}{}$ on the subalgebra $\TrHPl$, and the action of $w_{-r,0}$ on $\phi(v)$ is the adjoint action of $w_{-r,0}$ on $(W^-)^-$. The Lemma follows from our computation of these twisted Heisenberg relations in Propositions \ref{FockSpaceW action} and \ref{twistedHeisRels}.
\end{proof}

\begin{lemma}\label{bubbles commutes} For any $v\in V$ we have $\phi((d_0+d_2)v) = -2w_{0,3} \phi(v)$. \end{lemma}
\begin{proof} Propositions \ref{FockSpaceW action} and \ref{clockwise bubbles explicit} give that $w_{0,3}$ maps $w_{-1,0}$ to an element with leading term $w_{-1,2}$, and $(d_0+d_2)$ maps $\h{1}{}$ to an element with leading term $\h{1}{x_1^2}$. Comparision of the actions of these terms on the twisted Heisenberg subalgebras on either side gives that that their images in the endomorphisms of the Fock space are identical. \end{proof}

\begin{lemma} \label{Virasoro commutes} For any $v \in V$ we have $\phi(\h{\pm 2}{(x_1 + x_2)} v) = 2\sqrt{2}(w_{\mp 2, 1} + w_{\mp 2, 0}) \phi(v)$. \end{lemma}
\begin{proof} This follows from comparision of Lemma \ref{upVirasorowithUp} and Proposition \ref{FockSpaceW action}. \end{proof}

Now extend $\phi$ to a map $$\Phi: \TrHEv \longrightarrow W^-/\langle w_{0,0}, C-1 \rangle$$ by mapping $$\h{1}{} \mapsto \sqrt{2}w_{-1,0} \qquad \h{\pm 2}{(x_1 + x_2)} \rightarrow 2\sqrt{2}w_{\mp 2, 1} + w_{\mp 2, 0}  \qquad d_2+d_0 \mapsto -2w_{0,3}$$ and extending algebraically, i.e. $$\Phi(a_1 \ldots a_k) = \Phi(a_1) \ldots \Phi(a_k)$$ for generators $a_1, \ldots, a_k$ of $\TrHEv$. 

\begin{lemma}\label{well defined} The map $\Phi$ above is well defined. \end{lemma}

\begin{proof} Suppose  $A\in \TrHEv$ has two representations in terms of generators, $A= a_{i_1} \ldots a_{i_k} = a_{j_1} \ldots a_{j_{\ell}}$. Then $a_{i_1} \ldots a_{i_k}. V = a_{j_1}\ldots a_{j_{\ell}}.V$, so applying $\Phi$ gives $\Phi(a_{i_1} \ldots a_{i_k}).\mathcal{V}_{1,0}=\Phi(a_{j_1} \ldots a_{j_\ell}).\mathcal{V}_{1,0}$. Hence $\Phi(a_{i_1} \ldots a_{i_k}) = \Phi(a_{j_1} \ldots a_{j_\ell})$ by the faithfulness of the Fock space representation for $W^-$.\end{proof}

\begin{theorem} \label{main thm} The map $\Phi$ is an isomorphism of algebras. \end{theorem}

\begin{proof} We immediately have that $\Phi$ is surjective, because it maps generators to generators. Thus, it remains to show that $\Phi$ is injective. Let $A := a_{i_1} \ldots a_{i_k}\in \TrHEv$ and assume that $\Phi(A).\mathcal{V}_{1,0} = 0$. Then $\Phi(A) = 0$ by the faithfulness of the representation. But then $\Phi(a_{i_1})\ldots \Phi(a_{i_k}) .\mathcal{V}_{1,0}=0$. Then, by Lemmas \ref{Heis commutes}, \ref{bubbles commutes}, and \ref{Virasoro commutes}, we have $\Phi(a_{i_1})\ldots \Phi(a_{i_k} ). \mathcal{V}_{1,0} = \phi(a_{i_1}\ldots a_{i_k}. V) = \phi(A.v)=0$. But $\phi$ is an isomorphism, so this implies that $A.V=0$. Hence $A = 0$ by the faithfulness of the Fock space representation of $\TrHEv$. \end{proof}

\nocite{*}
\bibliographystyle{amsalpha}
\bibliography{references}

\newcommand{\etalchar}[1]{$^{#1}$}
\providecommand{\bysame}{\leavevmode\hbox to3em{\hrulefill}\thinspace}
\providecommand{\MR}{\relax\ifhmode\unskip\space\fi MR }
% \MRhref is called by the amsart/book/proc definition of \MR.
\providecommand{\MRhref}[2]{%
  \href{http://www.ams.org/mathscinet-getitem?mr=#1}{#2}
}
\providecommand{\href}[2]{#2}
\begin{thebibliography}{FKRW00}

\bibitem[AFMO94]{AFMO}
H.~{Awata}, M.~{Fukuma}, Y.~{Matsuo}, and S.~{Odake}, \emph{Determinant
  formulae of quasi-finite representation of {$\mathcal{W}_{1+\infty}$} algebra
  at lower levels}, Phys. Lett. B \textbf{332} (1994), no.~3-4, 336--344.

\bibitem[BGHL14]{BGHL}
A.~{Beliakova}, Z.~{Guliyev}, K.~{Habiro}, and A.D. {Lauda}, \emph{Trace as an
  alternative decategorification functor}, Acta Math. Viet. \textbf{39} (2014),
  425--480.

\bibitem[BHLW17]{BHLW}
A.~{Beliakova}, K.~{Habiro}, A.~{Lauda}, and B.~{Webster}, \emph{Current
  algebras and categorified quantum groups}, to appear in J. of the London
  Math. Soc. (2017).

\bibitem[BHLZ16]{BHLZ}
A.~{Beliakova}, K.~{Habiro}, A.D. {Lauda}, and M.~{Zivkovic}, \emph{Trace
  decategorification of categorified quantum $\mathfrak{sl}_2$}, Math. Ann.
  (2016), 1--44.

\bibitem[CLL{\etalchar{+}}16]{CLLSS}
S.~{Cautis}, A.~D. {Lauda}, A.~{Licata}, P.~{Samuelson}, and J.~{Sussan},
  \emph{{The elliptic Hall algebra and the deformed Khovanov Heisenberg
  category}}, arXiv 1609.03506 (2016).

\bibitem[CLLS16]{CLLS}
S.~Cautis, A.D. Lauda, A.M. Licata, and J.~Sussan, \emph{{W}-algebras from
  {H}eisenberg categories}, J. Inst. Math. Jussieu (2016), 1--37.

\bibitem[CS15]{CS}
S.~Cautis and J.~Sussan, \emph{On a categorical {B}oson--{F}ermion
  correspondence}, Comm. Math. Phys. \textbf{336} (2015), no.~2, 649--669.

\bibitem[EL16]{EL}
B.~Elias and A.D. Lauda, \emph{Trace decategorification of the {H}ecke
  category}, Journal of Algebra \textbf{449} (2016), 615--634.

\bibitem[FKRW00]{FKRW}
E.~{Frenkel}, V.~{Kac}, A.~{Radul}, and W.~{Wang}, \emph{{$W_{1+\infty} $} and
  {$ W(\mathfrak{gl}_N) $} with central charge {$N$}}, Commun. Math. Phys.
  \textbf{170} (2000), 337--357.

\bibitem[{Kho}14]{Khovanov}
M.~{Khovanov}, \emph{{H}eisenberg algebra and a graphical calculus}, Fund.
  Math. \textbf{225} (2014), 169--210.

\bibitem[KWY98]{Wang}
V.G. {Kac}, W.~{Wang}, and C.H. {Yan}, \emph{Quasifinite representations of
  classical {L}ie subalgebras of {$W_{1+\infty}$}}, Adv. Math. \textbf{139}
  (1998), no.~1, 56--140.

\bibitem[LRS16]{Jack}
A.~{Licata}, D.~{Rosso}, and A.~{Savage}, \emph{{A graphical calculus for the
  {J}ack inner product on symmetric functions}}, arXiv1610.01862 (2016).

\bibitem[{Naz}97]{Naz}
M.~{Nazarov}, \emph{Young's symmetrizers for projective representations of the
  symmetric group}, Adv. in Math. \textbf{127} (1997), 190--257.

\bibitem[{Ree}17]{Mike}
M.~{Reeks}, \emph{Cocenters of {H}ecke-{C}lifford and spin {H}ecke algebras},
  J. Algebra \textbf{476} (2017), 85--112.

\bibitem[RS16]{Rosso2016}
D.~Rosso and A.~Savage, \emph{A general approach to {H}eisenberg
  categorification via wreath product algebras}, Mathematische Zeitschrift
  (2016), 1--53.

\bibitem[SV13]{SV}
O.~{Schiffmann} and E.~{Vasserot}, \emph{Cherednik algebras, {W}-algebras and
  the equivariant cohomology of the moduli space of instantons on
  $\mathbb{A}^2$}, Publ. Math. Inst. Hautes. Etudes Sci. \textbf{118} (2013),
  213--342.

\bibitem[SVV14]{SVV}
P.~{Shan}, M.~{Varagnolo}, and E.~{Vasserot}, \emph{On the center of
  quiver-{H}ecke algebras}, arXiv:1411.4392v3 (2014).

\bibitem[WW12]{WW}
J.~{Wan} and W.~{Wang}, \emph{Frobenius character formula and spin generic
  degrees for {H}ecke-{C}lifford algebra}, Proc. London Math. Soc. (2012),
  287--317.

\end{thebibliography}
\end{document}